\documentclass[reqno]{amsart}
\pdfoutput=1\relax\pdfpagewidth=8.26in\pdfpageheight=11.69in\pdfcompresslevel=9
\usepackage{mathptmx,amsmath,mathbbol,lineno}
\usepackage{hyperref}
\usepackage{amsfonts}
\usepackage{amssymb}
\usepackage{amscd}
\usepackage{latexsym}
\usepackage{graphics}
\usepackage[all,knot,poly,arc,web]{xy}
\usepackage{ifthen}
\usepackage{epsfig}

\newtheorem{theorem}{Theorem}
\newtheorem{Hypotheses (H)}{Hypotheses (H)}
\newtheorem{definition}{Definition}
\newtheorem{proposition}{Proposition}
\newtheorem{corollary}{Corollary}
\newtheorem{lemma}{Lemma}
\newtheorem{remark}{Remark}
\newtheorem{example}{Example}

\newtheorem{prop}{Proposition}

\def\neweq#1{\begin{equation}\label{#1}}
\def\endeq{\end{equation}}

\def\phi{\varphi}
\def\RR{{\mathbb R} }

\def\ri{\rightarrow}

\date{}
\def\RR{{\mathbb R} }
\def\ri{\rightarrow}
\def\pp{\overrightarrow{p}(\cdot) }
\def\qq{\overrightarrow{q}(\cdot) }
\def\rr{\overrightarrow{r}(\cdot) }
\def\ss{\overrightarrow{s}(\cdot) }
\def\pxti{\partial_{x_{\tau_j}}}
\def\pxi{\partial_{x_i}}
\def\vpq{\mathcal{V}_{\pp,\qq}}
\def\vpqx{\mathcal{V}_{\pp,\qq}}
\def\vpqb{\mathbf{V}_{\pp,\qq}}
\def\vrsx{\mathcal{V}_{\rr,\ss}}
\def\vrsb{\mathbf{V}_{\rr,\ss}}

\def\fun(#1,#2,#3){\mathcal{E}_{_{#1}}(#2, #3)}
\def\sob(#1,#2){W^{1,#1}(#2)}
\def\forma(#1){\fun({\pp,\qq},\cdot,\cdot)}

\def\l2{L^2(\Omega)}

\def\funn(#1,#2,#3,#4,#5){\langle J_{#1} #2, #3 \rangle_{_{#4,#5}}}
\def\nn(#1,#2){\|#1\|_{_{#2}}}

\newcounter{appendix}
\newenvironment{apequation}{
\addtocounter{equation}{-1}
\refstepcounter{appendix}

\begin{equation}}
{\end{equation}}

\AtBeginDocument{{\noindent\small
{\em To appear in: J. d'Analyse Math\'ematique (2022)}}
\vspace{9mm}}

\begin{document}

\title[Anisotropic dynamical Wentzell heat equation with nonstandard growth]
{The generalized anisotropic dynamical Wentzell heat equation with nonstandard growth conditions}

\author[C. Carvajal\,-\,Ariza,\,\,\,J. Henr\'iquez\,-\,Amador,\,\,\,A. V\' elez\,-\,Santiago]
{Carlos Carvajal\,-\,Ariza,\,\,\,J. Henr\'iquez\,-\,Amador,\,\,\,Alejandro V\' elez\,-\,Santiago}

\address{Carlos Carvajal\,-\,Ariza\hfill\break
Department of Mathematics\\
University of Puerto Rico at R\'io Piedras\\
R\'io Piedras, PR \,00936}
\email{carlos.carvajal@upr.edu}

\address{Javier Henr\'iquez\,-\,Amador\hfill\break
Departamento de Ciencias Naturales y Exactas\\
Universidad de la Costa\\
Barranquilla, Atl\'antico, Colombia \,080002}
\email{jhenriqu13@cuc.edu.co}

\address{Alejandro V\'elez\,-\,Santiago\hfill\break
Department of Mathematical Sciences\\
University of Puerto Rico at Mayag\"uez\\
Mayag\"uez, PR \,00681}
\email{alejandro.velez2@upr.edu,\,\,\,alejovelez32@gmail.com}

\subjclass[2020]{35K05, 35K59, 35K92, 35D35, 35B65}
\keywords{Anisotropic heat equation, Wentzell boundary conditions, Nonstandard growth conditions, nonlinear $C_0$-semigroups,
Ultracontractivity properties}

\numberwithin{equation}{section}

\begin{abstract}
Let $\Omega\subseteq\mathbb{R\!}^N$ be a bounded Lipschitz domain with the anisotropic extension property, for $N\geq3$.
The aim of this paper is to establish the solvability and global regularity theory for a new class
of generalized anisotropic heat-type boundary value problems involving the anisotropic $\overset{\rightarrow}p(\cdot)$-Laplace operator
$\Delta_{\overset{\rightarrow}p(\cdot)}$, with (pure) dynamical anisotropic Wentzell boundary conditions
$$u_t+\displaystyle\sum_{i=1}^N |\partial_{x_i}u|^{p_i(\cdot)-2}\partial_{x_i}u\; \nu_{i}-\Delta_{_{\overset{\rightarrow}q(\cdot),\Gamma}}u+
\beta|u|^{q_M(\cdot)-2}u\,\ni\,0.$$
Here $\overset{\rightarrow}p(\cdot)$ and $\overset{\rightarrow}q(\cdot)$ are Lipschitz continuous vector fields, which can be unrelated between each other,
and $\Delta_{_{\overset{\rightarrow}q(\cdot),\Gamma}}$ represents the anisotropic $\overset{\rightarrow}q(\cdot)$-Laplace-Beltrami operator
acting on $\Gamma:=\partial\Omega$. We first prove that the operator $\Delta_{\overset{\rightarrow}p(\cdot)}$ with the above boundary conditions
generates a nonlinear order-preserving submarkovian $C_0$-semigroup $\{T_{\sigma}(t)\}$ over $\mathbb{X\!}^{\,r(\cdot)}(\overline{\Omega}):=L^{r(\cdot)}(\Omega)
\times L^{r(\cdot)}(\Gamma)$ for all measurable functions $r(\cdot)$ on $\overline{\Omega}$ with $1\leq r^-\leq r^+<\infty$. Consequently, the
corresponding anisotropic dynamical Wentzell problem is well-posed over $\mathbb{X\!}^{\,r(\cdot)}(\overline{\Omega})$. Furthermore, we show
that the nonlinear $C_0$-semigroup $\{T_{\sigma}(t)\}$ enjoys a H\"older-type ultracontractivity property in the sense that
there exist constants $C_1,\,C_2,\,\kappa>0$, and $\gamma\in(0,1)$, such that
$$|\|T_{\sigma}(t)\mathbf{u}_0-T_{\sigma}(t)\mathbf{v}_0\||_{_{\infty}}\,\leq\,
C_1\,e^{C_2t}t^{-\kappa}|\|\mathbf{u}_0-\mathbf{v}_0\||^{\gamma}_{_{r(\cdot),s(\cdot)}}$$
for every $\mathbf{u}_0,\,\mathbf{v}_0\in\mathbb{X\!}^{\,r(\cdot),s(\cdot)}(\overline{\Omega})$ and for all $t>0$, where
$r:\Omega\rightarrow[2,\infty]$ and $s:\Gamma\rightarrow[2,\infty]$ are measurable functions. We will provide explicit
formulations for the constants $\kappa$ and $\gamma$, and will discuss a particular case in which $\gamma\equiv1$.
\end{abstract}
\maketitle

\section{Introduction}\label{sec1}

\indent During the last years, the subject of variable exponent differential equations and problems with nonstandard growth structure
has drawn a lot of attention in the mathematical world (e.g. \cite{RadRep}), motivated by their abstract mathematical formulations, extensions, and beauty,
but even more, by their multiple types of applications. Among the list of applications,
we have elastic materials \cite{ANS,z0},  image restoration
\cite{ABOU-MESK-SOU08,C-G-L-Z08,CHEN-LEV-RAO06,D-G-H-Y-Z-Z}, electrorheological fluids \cite{ACER-MING02,DIE2002, DIE-RUZ03, RUZ04, RUZ00, SobSimmonds, Sob5},
thermorheological fluids
\cite{AnRo}, mathematical biology \cite{frag}, dielectric breakdown, and electrical
resistivity and polycrystal plasticity \cite{BMP} (among others). Thus one can say that allowing variations in an exponent $p$ of a function space
or a differential equation produce a big wave of interest, which keeps growing as the theory continues being mastered and applied to particular models.\\
\indent Adding anisotropy to the exponents opens the opportunity for an increase in the scale of applications, but it also adds major complications,
and these models are less known, and less investigated. Here we treat problems that are cast in the framework of the anisotropic nonstandard structures
given by vectorial functions
$$\overset{\rightarrow}{r}:D\rightarrow[1,\infty]^d;\,\,\,\,\,\,\,\,\,\,\,\,\,\,
\overset{\rightarrow}r(\cdot):=\left(r_1(\cdot),r_2(\cdot),\dots,r_d(\cdot)\right),$$
where $r_i:D\rightarrow[1,\infty]$ are measurable functions ($i\in\{1,\ldots,d\}$). We then denote
\begin{equation}\label{pM si pm}
r_M(\cdot)=\max\{r_1(\cdot), \ldots,r_d(\cdot)\},\quad r_m(\cdot)=\min\{r_1(\cdot),\ldots,r_d(\cdot)\}.
\end{equation}
Before we amplify the discussion about anisotropic function spaces and differentia equations and their importance in the field, let us
formulate explicitly the problem under consideration in this paper.\\
\indent Let $\Omega\subseteq\mathbb{R\!}^N$ be a bounded domain with the (anisotropic) $W^{1,\overset{\rightarrow} p(\cdot)}$-extension property
(see Definition \ref{p-extension}), and with Lipschitz boundary
$\Gamma:=\partial\Omega$, for $N\geq3$, and let $p_i(\cdot),\,q_j(\cdot)$ be
Lipschitz continuous functions over $\overline{\Omega}$, such that $1<p^-_i:=\textrm{ess}\displaystyle
\inf_{\overline{\Omega}}p_i(x)\leq p^+_i:=\textrm{ess}\displaystyle\sup_{\overline{\Omega}}p_i(x)<\infty$ and $1<
q^-_j:=\textrm{ess}\displaystyle\inf_{\overline{\Omega}}q_j(x)\leq q^+_j:=\textrm{ess}\displaystyle\sup_{\overline{\Omega}}q_j(x)<\infty$,
for each $i\in\{1,\ldots,N\}$ and $j\in\{1,\ldots,N-1\}$. Since $\Omega$ is bounded and has Lipschitz boundary, we stress that $\Gamma$
becomes a $(N-1)$-dimensional compact Riemannian manifold; this plays a crucial role.
We are concerned with the solvability of the generalized quasi-linear parabolic equation formally given by\\
\begin{equation}\label{1.1.03}
\frac{\partial u(t,x)}{\partial t}\,=\,\Delta_{\overset{\rightarrow}p(\cdot)}u(t,x)-\alpha|u(t,x)|^{p_M(\cdot)-2}u(t,x),\,\,\,\,\,\,\textrm{for each}
\,\,\,(t,x)\in(0,\infty)\times\Omega,
\end{equation}
\indent\\
with an initial condition $u(0,x)=u_0(x)$ for $x\in\overline{\Omega}$, and dynamical Wentzell boundary conditions\\
\begin{equation}\label{1.1.04}
\frac{\partial u(t,x)}{\partial t}|_{_{\Gamma}}+\displaystyle\sum_{i=1}^N |\partial_{x_i}u(t,x)|^{p_i(\cdot)-2}\partial_{x_i}u(t,x)\; \nu_{i}
-\Delta_{_{\overset{\rightarrow}q(\cdot),\Gamma}}u(t,x)+\beta|u(t,x)|^{q_M(\cdot)-2}u(t,x)\,=\,0,\,\,\,\,\,(t,x)\in(0,\infty)\times\Gamma,
\end{equation}
\indent\\
for $\alpha\in L^{\infty}(\Omega)$ and $\beta\in L^{\infty}(\Gamma)$ with $\displaystyle\inf_{x\in\Omega}\alpha(x)\geq\alpha_0$
and $\displaystyle\inf_{x\in\Gamma}\beta(x)\geq\beta_0$ for some constants $\alpha_0,\,\beta_0>0$, for
$\nu_i$ the $i$-th component of the outward unit normal at $\Gamma$, where
$$\Delta_{\overset{\rightarrow}p(\cdot)}u:=\displaystyle\sum_{i=1}^N\displaystyle\partial_{x_i}
\left(|\partial_{x_i}u|^{p_i(\cdot)-2}\partial_{x_i}u \right)$$ denotes the (anisotropic) $\overset{\rightarrow}p(\cdot)$-Laplace operator,
and $$\Delta_{_{\overset{\rightarrow}q(\cdot),\Gamma}}u:=\displaystyle\sum_{j=1}^{N-1}\displaystyle\pxti
\left(|\pxti u|^{q_j(\cdot)-2}\pxti u \right)$$ represents the (anisotropic)
$\overset{\rightarrow}q(\cdot)$-Laplace-Beltrami operator, where
$\pxti u$ denotes the directional derivative of $u$ along the tangential directions $\tau_j$ at each point on $\Gamma$.\\
\indent The above model equation (\ref{1.1.03}) with boundary conditions (\ref{1.1.04}) is regarded as an {\it anisotropic partial differential equation}, since
the growth of the operator with respect to each $i$-th partial derivative is governed by
different powers (see for example \cite[Chapter 1, Section 4.2]{ANT-DIAZ-SHM02}). As pointed out in \cite{B-D-F-G-M},
very often, in many relevant applications, the materials and phenomena have an important
{\it anisotropy}, presenting different properties in different directions, in contrast to the more usual isotropy property. Anisotropy leads to
mathematical models presenting peculiar constitutive laws which correspond to material's physical or mechanical
properties with different behaviour according the directions. The spectrum of fields in which such situations arise is
broad; it integrates areas, such as: computer science (e.g. image processing), physics (e.g. atmospheric radiative transfer), chemistry (e.g.
materials science \cite{NEWNH05}), geophysics and geology, engineering (e.g. wastewater reactors \cite{DIAZ-GOM16}), or neuroscience (among others).
Commonly, this anisotropy comes as a diffusion operator described by three principal different directional
coefficients. This is the case, for example, of thermal and electrical conductivity in heterogeneous
media, or of the study of crystals (see \cite{FONS91}). Moreover, many of the recent innovations
in the field of electroceramics have exploited the anisotropy of nonlinearities modelling different material properties
such as electric field, mechanical stress or temperature (see for example \cite[Chapter 15]{NEWNH05}). We also mention here that it is
well-known that homogenization techniques generate anisotropic diffusion limit problems (e.g. \cite{LIONS-SOUG05,ZHIK99}).\\
\indent The study of anisotropic boundary value problems involving variable exponents is less known in the mathematical community.
Some existence and regularity results for particular classes of Neumann-type or Robin-type anisotropic elliptic problem (with variable exponents) on
smooth domains can be found in \cite{NABR,ELLA-HACH2017,FanNoDEA2010,HEN-AVS19-1,IBRA-OUA2015}. However, the realization of the corresponding
parabolic problem with nonstandard anisotropic structure has been considered mainly for the Dirichlet problem; the realization
of a Robin problem involving nonstandard anisotropic structure has been developed in \cite{BOU-AVS18}, but the dynamical
problem has not been investigated, up to the present time.\\
\indent On the other hand, Wentzell boundary conditions (also referred as pure Wentzell boundary conditions), are viewed as a
kind of ``dynamical" boundary condition, and the physical interpretation
of this kind of boundary condition was given by Goldstein \cite{GOL06}. Moreover, Wentzell boundary conditions arise in many applications, such as
phase-transition phenomena, fluid diffusion, heat flow subject to nonlinear cooling on the boundary, suspended transport energy,
fermentation, population dynamics, and climatology, among many others (see \cite{ARON-WIN78,DIAZ-TEL08,EVANS76-77,GAL-GRAS-MIR08,GAL-MIR09,PAO92}
and the references therein), and have been investigated and used by several authors
(e.g. \cite{AG-DOU59,PEE61,VAS-VIT11,VISH52,WAR12-2}, among others). However, the realization of Wentzell-type problems
of nonstandard growth structure is a new subject, first considered by V\'elez\,-\,Santiago \cite{VELEZ2012-3}.\\
\indent In the present paper, we deal with a general class of time-dependent Wentzell-type evolution equations involving the anisotropic Laplace operator
$\Delta_{\overset{\rightarrow}p(\cdot)}$ in the interior of the region, while its boundary is governed by a second-order anisotropic Laplace-Beltrami
operator $\Delta_{_{\overset{\rightarrow}q(\cdot),\Gamma}}$. Our goal consists in establishing the well-posedness and ultracontractivity
bounds for the parabolic problem (\ref{1.1.03}) with dynamical boundary conditions
(\ref{1.1.04}). In fact, we first show that for each measurable function $r(\cdot)$ over $\overline{\Omega}$ with $1\leq r^-\leq r^+<\infty$,
the subdifferential associated with the Wentzell functional $\Phi_{\sigma}$ (see Section \ref{sec3})
generates an order-preserving non-expansive nonlinear $C_0$-semigroup $\{T_{\sigma}(t)\}$
over $\mathbb{X\!}^{\,r(\cdot)}(\overline{\Omega}):=L^{r(\cdot)}(\Omega)\times L^{r(\cdot)}(\Gamma)$. Consequently, for each
$\mathbf{u}_0:=(u_0,u_0|_{_{\Gamma}})\in\mathbb{X\!}^{\,r(\cdot)}(\overline{\Omega})$, the function $\mathbf{u}(t):=T_{\sigma}(t)\mathbf{u}_0$
turns out to be the unique strong solution to problem (\ref{1.1.03}) with boundary conditions (\ref{1.1.04}). Then, under the additional assumption
$\min\{p^-_m,q^-_m\}\geq2$, we prove that the nonlinear Wentzell semigroup $\{T_{\sigma}(t)\}$ satisfies a H\"older-type ultracontractivity
property, in the sense that there exist constants $C_1,\,C_2,\,\kappa>0$, and $\gamma\in(0,1]$, such that
\begin{equation}\label{HUP}
|\|T_{\sigma}(t)\mathbf{u}_0-T_{\sigma}(t)\mathbf{v}_0\||_{_{\infty}}\,\leq\,
C_1\,e^{C_2t}t^{-\kappa}|\|\mathbf{u}_0-\mathbf{v}_0\||^{\gamma}_{_{r(\cdot),s(\cdot)}}
\end{equation}
for each $\mathbf{u}_0,\,\mathbf{v}_0\in\mathbb{X\!}^{\,r(\cdot),s(\cdot)}(\overline{\Omega})$ and for all $t>0$, where
$r:\Omega\rightarrow[2,\infty]$ and $s:\Gamma\rightarrow[2,\infty]$ are measurable functions. The constants $C_1,\,C_2$, $\kappa$, and $\gamma$,
have been explicitly calculated. In particular, the ultracontractivity property shows that solutions of problem (\ref{1.1.03}) with
boundary conditions (\ref{1.1.04}) are globally bounded.\\
\indent The above general model equation brings several important novelties. First, problems
involving the anisotropic Laplace-Beltrami operator correspond to a completely new subject in the field. In fact, the appearance of such anisotropic operator
is only known in the recent work \cite{DIAZ-VELEZ20} (for an elliptic boundary value problem), but it has never appeared in a parabolic-type
differential equation. The framework in the present paper not only motivates the study of anisotropic problems over the interior of a domain,
but also the realization of anisotropic parabolic-type differential equations over Riemannian manifolds. This is a subject never considered
before, up to the present time.\\
\indent H\"older-type ultracontractivity results for nonlinear semigroups have been established in \cite{CIP-GR01,CIP-GR02,WAR09}
(for the constant isotropic case $p=q$),
where elegant methods have been developed, which rely on logarithmic Sobolev inequalities (a subject motivated by the work of Davies \cite{DAV}).
In \cite{BOU-AVS18}, the authors presented for the first time ultracontracitvity results for a class of anisotropic Robin problems with variable exponents.
However, the authors found a mistake in \cite[proof of Lemma 10]{BOU-AVS18} (which allowed for the remaining of the proofs to run as in
the previous cases in the literature), which has happened as a consequence of absence of a positive constant at the last
inequality of \cite[proof of Lemma 10]{BOU-AVS18}. It turns out that such missing positive constant may bring complications in the subsequent results,
so one needs a way to avoid it, or try another approach. Fortunately, in our present situation we
were able to find the way to correct this mistake, by keeping the problem in the anisotropic structure
(instead of reducing it to $p_m(\cdot)$-gradient term; see the last inequality of \cite[proof of Lemma 10]{BOU-AVS18}), which complicates
substantially several subsequent calculations. Nevertheless, we were able to develop a technical procedure to deal with the anisotropy
and furthermore, the fact that we have two (possibly independent) vector fields exponents acting on the interior and boundary, respectively.
A more detailed explanation is provided in Remark \ref{JDE}.
In \cite{VELEZ2012-3}, the results in \cite{WAR09} were extended to the variable exponent isotropic case $p(\cdot)=q(\cdot)$, by establishing
(\ref{HUP}); however, the explicit value of the constants in (\ref{HUP}) (with the exception of $\gamma$) were not achieved, and in addition
some key typos in the proofs of key lemmas in \cite{VELEZ2012-3} were corrected.
It is also important to outline the works in \cite{BON-CIP-GR03,BON-GR05}, where the authors developed H\"older-type ultracontractivity bounds for
nonlinear semigroups over compact Riemannian manifolds. However, ultracontractivity properties associated with semigroups generated by nonlinear
anisotropic-type operators have never been considered before (over Riemannian manifolds), even in the constant case. This will be an important result which we
will investigate in this paper. Furthermore, we will obtain ultracontractivity properties, not only on the interior of the domain, but
also at the boundary, which in our case is regarded as a compact Riemannian manifold. This opens the idea of having a $C_0$-semigroup $\{T_{\sigma}(t)\}$
over $\mathbb{X\!}^{\,r(\cdot)}(\overline{\Omega})$ associated with problem (\ref{1.1.03}) with boundary conditions (\ref{1.1.04}), which could
behave as an order pair $$T_{\sigma}(t)=\left(T_{\overset{\rightarrow}p(\cdot)}(t),T_{\overset{\rightarrow}q(\cdot)}(t)\right),$$
where $\{T_{\overset{\rightarrow}p(\cdot)}(t)\}$ is a semigroup associated with a differential equation involving the
$\overset{\rightarrow}p(\cdot)$-Laplace operator $\Delta_{\overset{\rightarrow}p(\cdot)}$, and $\{T_{\overset{\rightarrow}q(\cdot)}(t)\}$
is a semigroup associated with a differential equation (over a Riemannian manifold) involving the $\overset{\rightarrow}q(\cdot)$-Laplace-Beltrami operator $\Delta_{_{\overset{\rightarrow}q(\cdot),\Gamma}}$. The validity of the above
equality is unknown, up for the present time; however, we will show that the Wentzell semigroup generated in this paper will possess some properties exhibiting a sort
of ``splitting" behavior into interior and boundary operators.\\
\indent Another important point to consider, is the fact that the boundary value problem (\ref{1.1.03}) with (\ref{1.1.04}) deals with multiple variables, in
the sense that it contains an interior vector field $\overset{\rightarrow}p(\cdot)$, and a boundary vector field $\overset{\rightarrow}q(\cdot)$. These two
vectorial functions are in general not equal, and unrelated each other. We comment that Gal and Warma \cite{GAL-WAR} were the first ones to consider a
Wentzell-type quasi-linear boundary value problem involving two (constant) exponents $p\neq q$, but in their case they only consider an elliptic-type
differential equation, in which $p\neq q$, but $q$ is related to $p$ in some way. This is no longer the case in the present situation. In
\cite{DIAZ-VELEZ20}, the authors recently generalized the results in \cite{GAL-WAR} to a class of elliptic anisotropic boundary value problems
where $\overset{\rightarrow}p(\cdot)\neq\overset{\rightarrow}q(\cdot)$ are unrelated each other. Here we turn our attention into the corresponding
time-dependent problem with dynamical boundary conditions, which has interest on its own, since the methods employed in the parabolic case are
completely different than the ones applied in the elliptic case.\\
\indent The organization of the paper is the following. Section \ref{sec2} is divided into five parts, in which we outline the main definitions, known tools,
and known results, which will be applied in the subsequent section. We will present an overview on all the machinery related to variable exponent isotropic
and anisotropic function spaces, relative capacity, and the corresponding embedding results.
Some useful analytical tools are provided, and at the end, we give an overview on nonlinear semigroups. Then, section \ref{sec3} concerns the
realization and global regularity theory for problem (\ref{1.1.03}) with boundary conditions (\ref{1.1.04}), and it is divided in two main parts.
In the first subsection, we show that the boundary value problem (\ref{1.1.03}) with (\ref{1.1.04}) is well-posed over
$\mathbb{X\!}^{\,r(\cdot)}(\overline{\Omega})$ for any measurable function $r:\overline{\Omega}\rightarrow[1,\infty)$ with $r^+<\infty$.
We also show that the associated Wentzell $C_0$-semigroup is order-preserving and non-expansive over $\mathbb{X\!}^{\,r(\cdot)}(\overline{\Omega})$.
In the second part, we establish an H\"older-type ultracontractivity result for the Wentzell semigroup generated in the first part. In particular,
we establish the validity of the inequality (\ref{HUP}), where the constants in (\ref{HUP}) are explicitly given. We also discuss a particular
case in which we get a Lipschitz ultracontractivity property, namely, the fulfillment of (\ref{HUP}) for $\gamma=1$. At the end, we add an Appendix A,
in which we derive the explicit values of the constants appearing in (\ref{HUP}).

\section{Preliminaries and intermediate results}\label{sec2}

\indent In this section we set up the notations, and collect the necessary intermediate results that will be frequently used throughout the
rest of the paper.\\

\subsection{Variable exponent function spaces}\label{subsec2.1}

Everywhere in this section, let $\Omega\subset \RR^N$, $N\geq2$, with $\text{meas}(\Omega)>0$, be a bounded open set and let us denote by
${\mathfrak{D}}$ either $\Omega$, or its boundary $\Gamma$, in order to uniformly recall definitions and properties.
Also,
$${\mathcal P}({\mathfrak{D}})=\{h:\mathfrak{D}\ri\RR  \mbox{ a
	measurable function such that }h^-\geq 1\}.$$
For $r\in\mathcal{P}({\mathfrak{D}})$  we define $\mathfrak{D}_{r,\infty}:=\{x\in\mathfrak{D}\mid r(x)=\infty\}$,
and $\mathfrak{D}_{r,0}:=\mathfrak{D}\setminus\mathfrak{D}_{\infty}$. Then, for a measure $\eta$ on $\mathfrak{D}$,
we define the Lebesgue space with variable exponent,
$$L^{r(\cdot)}({\mathfrak{D}}):=L^{r(\cdot)}({\mathfrak{D}},d\eta):=\left\{u:\mathfrak{D}\ri\RR  \mbox{ a
	measurable function }\mid
\Theta_{_{r(\cdot),\mathfrak{D}}}(u)<\infty\right\},$$
where the functional $\Theta_{_{r(\cdot),\mathfrak{D}}}:L^{r(\cdot)}({\mathfrak{D}})\rightarrow\RR$ is defined as
\begin{equation}\label{p-functional}
\Theta_{_{r(\cdot),\mathfrak{D}}}(u):=\displaystyle\int_{\mathfrak{D}_{r,0}}|u|^{r(x)}\,d\eta+\|u\|_{_{\infty,\mathfrak{D}_{r,\infty}}}.
\end{equation}
$\Theta_{_{r(\cdot),\mathfrak{D}}}$ is called the $r(\cdot)$-modular of the $L^{r(\cdot)}({\mathfrak{D}})$ space.
We endow the space $L^{r(\cdot)}({\mathfrak{D}})$ with the {\it  Luxemburg norm}
\begin{equation}\nonumber\label{elnorm}
\|u\|_{_{r(\cdot),\mathfrak{D}}}=\inf\left\{\mu>0\mid\,\Theta_{_{p(\cdot),\mathfrak{D}}}(|u|/\mu)\,\leq\,1\right\}.
\end{equation}

The space $\left(L^{r(\cdot)}({\mathfrak{D}}),\, \|\cdot\|_{_{r(\cdot),\mathfrak{D}}}
\right)$ is a Banach space (see \cite[Theorem 2.5]{KOV-RAK91}), which is reflexive whenever $1<r^-\leq r^+<\infty$ (see \cite[Corollary 2.7]{KOV-RAK91}).
For simplicity, in the case when $1<r^-\leq r^+<\infty$, we write $\mathfrak{D}_{r,0}=\mathfrak{D}$ and $\mathfrak{D}_{r,\infty}=\emptyset$.
In the subsequent sections, the $L^{r(\cdot)}$-spaces under consideration will be $L^{r(\cdot)}(\Omega):=L^{p(\cdot)}(\Omega,dx)$,
and $L^{r(\cdot)}(\Gamma):=L^{r(\cdot)}(\Gamma,d\sigma)$ (for $\sigma$ the surface measure on $\Gamma:=\partial\Omega$).
Also, for $r, \,s\in\mathcal{P}({\mathfrak{D}})$ we set
$$\mathbb{X\!}^{\,r(\cdot),s(\cdot)}(\overline{\Omega}):=L^{r(\cdot)}({\Omega})\times L^{s(\cdot)}({\Gamma}),$$
endowed with the norm
$$|\|(f,g)\||_{r(\cdot),s(\cdot)}:=\|f\|_{_{r(\cdot),\Omega}}+\|g\|_{_{s(\cdot),\Gamma}}\,,
\indent\,\textrm{if}\,\,\,r,\,s\in[1,\infty),$$
and, for $r=s=\infty$, we endow $\mathbb{X\!}^{\,\infty,\infty}(\overline{\Omega})$ with
$$|\|(f,g)\||_{\infty}:=
\max\left\{\|f\|_{_{\infty,\Omega}},\,\|g\|_{_{\infty,\Gamma}}\right\}.$$
If $r(\cdot)=s(\cdot)$ over $\overline{\Omega}$, we write $\mathbb{X\!}^{\,r(\cdot),s(\cdot)}(\overline{\Omega})=\mathbb{X\!}^{\,r(\cdot)}(\overline{\Omega})$
and $|\|(\cdot,\cdot)\||_{r(\cdot),s(\cdot)}=|\|(\cdot,\cdot)\||_{r(\cdot)}$. Also, if the trace $u|_{_{\Gamma}}$ of $u$ exists, we
will denote $\mathbf{u}:=(u,u|_{_{\Gamma}})$. More properties concerning the Lebesgue spaces are recalled below.
For $r\in L^\infty({\mathfrak{D}})$, the following H\"older-type
inequality
\begin{equation}\label{Hol}
\left|\int_{\mathfrak{D}} u(x)v(x)\;d\eta\right|
\leq2\,\|u\|_{_{r(\cdot),\mathfrak{D}}}\|v\|_{_{r'(\cdot),\mathfrak{D}}}
\end{equation}
holds for all $u\in L^{r(\cdot)}({\mathfrak{D}})$ and $v\in
L^{r'(\cdot)}({\mathfrak{D}})$ (see \cite[Theorem 2.1]{KOV-RAK91}), where  $L^{r'(\cdot)}({\mathfrak{D}})$ represents the conjugate space of
$L^{r(\cdot)}({\mathfrak{D}})$, obtained by conjugating the exponent
pointwise, that is,  $1/r(x)+1/r'(x)=1$ (see \cite[Corollary 2.7]{KOV-RAK91}).

If $u\in L^{r(\cdot)}({\mathfrak{D}})$ and $r^+<\infty$ we have the following well-known properties:
\begin{equation}\label{L40}
\|u\|_{_{r(\cdot),\mathfrak{D}}}<1\;(=1;\,>1)\quad\mbox{if and only if}\quad\Theta_{_{r(\cdot),\mathfrak{D}}}(u) <1\;(=1;\,>1);
\end{equation}
\begin{equation}\label{L4}
\mbox{if}\quad\|u\|_{_{r(\cdot),\mathfrak{D}}}>1\quad\mbox{then}\quad\|u\|_{_{r(\cdot),\mathfrak{D}}}^{r^-}\leq\Theta_{_{r(\cdot),\mathfrak{D}}}(u)
\leq\|u\|_{_{r(\cdot),\mathfrak{D}}}^{r^+};
\end{equation}
\begin{equation}\label{L5}
\mbox{if}\quad\|u\|_{_{r(\cdot),\mathfrak{D}}}<1\quad\mbox{then}\quad\|u\|_{_{r(\cdot),\mathfrak{D}}}^{r^+}\leq \Theta_{_{r(\cdot),\mathfrak{D}}}(u)\leq\|u\|_{_{r(\cdot),\mathfrak{D}}}^{r^-};
\end{equation}
\begin{equation}\label{L06}
\|u\|_{_{r(\cdot),\mathfrak{D}}}\rightarrow
0\;(\rightarrow\infty)\quad\mbox{if and only if}\quad\Theta_{_{r(\cdot),\mathfrak{D}}} (u)\rightarrow 0\;(\rightarrow\infty);
\end{equation}
see for example \cite[Theorem 1.3, Theorem 1.4]{FAN01}.
\smallskip

Let us pass to the (isotropic) Sobolev space with variable exponent, that is,
$$W^{1,r(\cdot)}({\Omega})=\left\{u\in L^{r(\cdot)}({\Omega})\mid\,\pxi  u\in L^{r(\cdot)}(\Omega),\,\, i\in\{1,\dots,N\} \right\},$$
where $\pxi  u$,  $i\in\{1,\dots,N\}$, represent the partial derivatives of $u$ with respect to $x_i$ in the weak sense.
This space, endowed with the norm
\begin{equation}\label{norm isotropic}
\|u\|_{W^{1,r(\cdot)}({\Omega})}=\|u\|_{_{r(\cdot),\Omega}}+\|\nabla u\|_{_{r(\cdot),\Omega}},
\end{equation}
is a separable and reflexive Banach space
(see \cite[Theorem 3.1]{KOV-RAK91}) whenever $r\in {\mathcal P}({\Omega})$ with $1<r^-\leq r^+<\infty$.
For more classical properties of the variable exponent
Lebesgue and Sobolev spaces, refer to \cite{D-H-H-R11,FAN01,KOV-RAK91,MUS}, among others.

Next, we introduce the anisotropic space $W^{1,\rr}(\Omega)$, where $\rr:\Omega\longrightarrow \mathbb{R}^N$ is the vectorial function
\begin{align*}
	\rr = \left( r_1(\cdot),\cdots, r_N(\cdot)\right)
\end{align*}
and $r_i\in \mathcal{P}(\Omega)$ with $1<r_{i}^{-}\leq r_{i}^{+}<\infty$ for all $i\in\left\lbrace 1,\ldots,N\right\rbrace$. The \textbf{anisotropic variable exponent Sobolev space} is defined by
\begin{eqnarray}
W^{1,\rr}(\Omega)=\left\lbrace u\in L^{r_{M}(\cdot)}(\Omega)\mid \partial_{x_i}u\in L^{r_{i}(\cdot)}(\Omega)\quad\mbox{for all} \quad i\in \{1,\ldots ,N\} \right\rbrace
\end{eqnarray}
endowed with the norm
\begin{eqnarray}
\|u\|_{W^{1,\rr}(\Omega)} =  \|u\|_{_{r_M(\cdot),\Omega}}+\sum_{i=1}^{N}\|\partial_{x_i}u\|_{_{r_{i}(\cdot),\Omega}}.
\end{eqnarray}
\indent On the other hand, for $s\in\mathcal{P}(\Gamma)$,
we define the variable exponent Sobolev space at the boundary $W^{1,s(\cdot)}(\Gamma)$
as the completion of the space $C^1(\Gamma)$ with respect to the norm $$\|u\|_{_{W^{1,q(\cdot)}(\Gamma)}}:=
\inf\left\{\lambda>0\mid\Theta_{_{q(\cdot),\Gamma}}(u/\lambda)+\Theta_{_{q(\cdot),\Gamma}}(|\nabla_{^{\Gamma}}u|/\lambda)\leq1\right\},$$
where we recall that $\nabla_{^{\Gamma}}u:=(\partial_{x_{\tau_1}}u,\ldots,\partial_{x_{\tau_{N-1}}}u)$ denotes
the tangential gradient at $\Gamma:=\partial\Omega$ (also called the Riemannian gradient of $u$),
for $\partial_{x_{\tau_j}}u$ the directional derivative of $u$ along the tangential directions $\tau_j$ at each point on $\Gamma$.
For a treatment of this space for $q\in(1,\infty)$ constant, we refer to \cite{HEB96,JOST}. Then proceeding as above,
given $\ss:\Gamma\longrightarrow \mathbb{R}^N$ a measurable vectorial function, one can define
the corresponding anisotropic Sobolev space over $\Gamma$ by:
\begin{eqnarray}
W^{1,\ss}(\Gamma)=\left\lbrace u\in L^{s_{M}(\cdot)}(\Gamma)\mid\pxti u\in L^{s_{j}(\cdot)}(\Gamma)\quad\mbox{for all} \quad j\in \{1,\ldots ,N-1\} \right\rbrace
\end{eqnarray}
endowed with the norm
\begin{eqnarray}
\|u\|_{W^{1,\ss}(\Gamma)}:=\|u\|_{_{s_M(\cdot),\Gamma}}+\sum_{j=1}^{N-1}\|\pxti u\|_{_{s_{j}(\cdot),\Gamma}}.
\end{eqnarray}
Then one can deduce in the same way as in the variable exponent Lebesgue spaces that
$W^{1,\rr}(\Omega)$ and $W^{1,\ss}(\Gamma)$ are both Banach spaces, which are reflexive,
provided that $1<r^-_m\leq r^{+}_M<\infty$ and $1<s^-_m\leq s^{+}_M<\infty$ (e.g. \cite{fananis}). \\
\indent Now, we define the space\\
\begin{equation}\label{2.1.01}
\vrsx:=\left\{u\in W^{1,\rr}(\Omega)\mid u|_{_{\Gamma}}\in W^{1,\ss}(\Gamma)\right\}.
\end{equation}\indent\\
Then, one have that the space $\vrsx$ is a Banach space with respect to the norm\\
\begin{equation}\label{2.1.02}
\|\cdot\|_{_{\vrsx}}:=\|\cdot\|_{_{W^{1,\rr}(\Omega)}}+\|\cdot\|_{_{W^{1,\ss}(\Gamma)}}.
\end{equation}
Finally, we define our target function space by
\begin{equation}\label{Pair-Space}
\vrsb:=\left\{\mathbf{u}:=(u,u|_{_{\Gamma}})\mid u\in\vrsx\right\},\,\,\,\,\,\textrm{endowed with the norm:}\,\,\,\,\,
|\|\mathbf{u}\||_{_{\vrsb}}:=\|u\|_{_{\vrsx}}.
\end{equation}
\indent The next definitions will play a crucial role on identifying the class of domains under consideration in this paper.
We start by introducing to the mathematical literature the notion of anisotropic $p(\cdot)$-extension domains.\\

\begin{definition}\label{p-extension}
	Let $r_i\in C(\mathbb{R\!}^N\,)$, for $i\in\{1,\ldots,N\}$. A domain
	$\Omega\subseteq\mathbb{R\!}^N$ is called an \textbf{anisotropic $r(\cdot)$-extension domain},
	if $\Omega$ has the $W^{1,\overset{\rightarrow} r(\cdot)}$-\textbf{extension property},
	that is, if there exists a bounded linear operator
	$S:W^{1,\overset{\rightarrow} r(\cdot)}(\Omega)
	\rightarrow W^{1,\overset{\rightarrow} r(\cdot)}(\mathbb{R\!}^N)$ such that $Su=u$\, a.e. on $\Omega$.
	In the case when $r_i(\cdot)=r(\cdot)$ over $\mathbb{R\!}^N$ for all $i\in\{1,\ldots,N\}$, we
	call $\Omega$ a $r(\cdot)$-\textbf{extension domain}.\\
\end{definition}

There are many examples of domains satisfying the statement of Definition \ref{p-extension} (for the case when $r_1=r_2=\ldots=r_N=$\,constant; see for example
\cite[Example 6.7]{VELEZ2013-1}), and it is
well-known that Lipschitz domains are $r(\cdot)$-extension domains for all ``sufficiently smooth" $r(\cdot)$ with $1\leq r^-\leq r^+\leq\infty$.
However, in the present case, as we are dealing with
anisotropic Sobovev spaces, the situation becomes more critical, namely because for many domains, the extension operator may
``mix" the derivatives in such way that the extended function no more belongs to the required anisotropic Sobolev space
(see for instance \cite{monatsh}). So in order for a domain to be an anisotropic $r(\cdot)$-extension domain, we are assuming
that the situation previously mentioned does not occur, that is, we are requiring that the extended function may lie in the corresponding
anisotropic Sobolev space without going through the ``mixing" problem mentioned above. Consequently, the class of domains with the anisotropic property
in Definition \ref{p-extension} is smaller than the one considered in \cite{VELEZ2013-1}. Known examples of anisotropic $r(\cdot)$-extension domains include rectangular domains, rectangular-like domains, domains with the semirectangular restriction, and
admissible domains (in the sense of \cite{RAK81}); see for instance \cite{fananis, monatsh, RAK81}. In particular, it is unknown (up to the present time) whether
bounded Lipschitz domains are anisotropic $r(\cdot)$-extension domain, even for a smooth vector field $\overset{\rightarrow} p(\cdot)$. Thus throughout
the body of the paper, we will assume this condition (in addition to the preliminary assumption over the domain $\Omega$).

\begin{remark}\label{Equiv-Extension}
If $p_i\in C^{0,1}(\overline{\Omega})$, then it is known that each $p_i$ can be extended to a function $\mathfrak{p}_i\in C^{0,1}(\mathbb{R\!}^N)$
with $\mathfrak{p}^+_i=p^+_i$, and $\mathfrak{p}^-_i=p^-_i$ (in fact this holds even for log-H\"older continuous functions; see for instance \cite{DIE04}).
Then for such case, one can regard an anisotropic $p(\cdot)$-extension domain as an order pair $(\pp,\Omega)$ possessing an extension property in the following way:
\begin{enumerate}
\item[$\bullet$]\,\, $\pp=(p_1(\cdot),\ldots,p_{_N}(\cdot))$ can be extended to a vectorial function $\overset{\rightarrow}{\mathfrak{p}}(\cdot)=
(\mathfrak{p}_1(\cdot),\ldots,\mathfrak{p}_{_N}(\cdot))$.
\item[$\bullet$]\,\, There exists a bounded extension operator from $W^{1,\pp}(\Omega)$ into $W^{1,\overset{\rightarrow}{\mathfrak{p}}(\cdot)}(\mathbb{R\!}^N)$
such that the derivatives may not get mixed under the extension operator.
\end{enumerate}
In \cite{fananis, monatsh}, the authors show that a rectangular domain satisfies all the conditions in the formulation given in Remark \ref{Equiv-Extension}.
\end{remark}
\indent\\

\subsection{Sobolev embedding theorems, and inequalities}\label{subsec2.2}

\indent The following embedding results will be frequently applied in the subsequent sections. We recall that for now on we will assume
that $\Omega\subseteq\mathbb{R\!}^N$ is a bounded Lipschitz domain, for $N\geq2$. We begin by stating the following know results.\\

\begin{theorem}\,(see \cite{BOU-AVS18})\,\,\label{tembcont}
Let ${\Omega}\subset\mathbb{R\!}^N$ ($N\geq 2$) be a bounded $p_m(\cdot)$-extension domain, for
$p_m(\cdot)\in C^{0,1}(\overline{\Omega})$ with $1<p^-_m\leq p^+_m<N$,
and suppose that $\overset{\rightarrow} p(\cdot)\in C^{0,1}(\overline\Omega)^N$.
\begin{enumerate}
\item[(a)]\,\,\,If $r(\cdot)\in {\mathcal P}(\overline{{\Omega}})$ is such that $r^->1$, and
\begin{equation}\label{critical-I}
r(\cdot)\,\leq\,\displaystyle\frac{Np_m(\cdot)}{N-p_m(\cdot)}:=p^{\star}_m(\cdot)\,,\qquad\textrm{over}\,\,\overline{\Omega},
\end{equation}
then the interior embedding $ W^{1,\overset{\rightarrow} p(\cdot)}(\Omega)\hookrightarrow L^{r(\cdot)} (\Omega)$ is continuous. Moreover,
if there exists $\epsilon>0$ such that $$r(x)\,\leq\,p^{\ast}_m(\cdot)-\epsilon,\qquad\forall\, x\in \overline{\Omega},$$
then the embedding $W^{1,\overset{\rightarrow} p(\cdot)}(\Omega)\hookrightarrow L^{r(\cdot)} (\Omega)$ is compact.
\item[(b)]\,\,\,Assume now that the boundary $\Gamma$ of $\Omega$ is an upper $d$-set with respect to a measure $\mu$ on $\Gamma$,
for $d\in(N-p^-_m,N)$. If $s(\cdot)\in {\mathcal P}(\overline{{\Omega}})$ is such that $s^->1$, and
$$s(x)\,\leq\,\displaystyle\frac{dp_m(\cdot)}{N-p^-_m}\,,\qquad\forall\, x\in \overline{\Omega},$$
then the trace embedding $ W^{1,\overset{\rightarrow} p(\cdot)}(\Omega)\hookrightarrow L^{s(\cdot)} (\Gamma)$ is continuous. Furthermore,
if there exists $\epsilon>0$ such that $$s(x)\,\leq\,\displaystyle\frac{dp_m(\cdot)}{N-p^-_m}-\epsilon,\qquad\forall\, x\in \overline{\Omega},$$
Then the trace embedding $W^{1,\overset{\rightarrow} p(\cdot)}(\Omega)\hookrightarrow L^{s(\cdot)} (\Gamma)$ is compact.
\end{enumerate}
\end{theorem}

The next embedding result for the variable exponent case has been established over
general compact Riemannian manifolds by Gaczkowski and G\'orka \cite{GACZ-GORK13}. In the case of a bounded Lipschitz domain,
clearly $\Gamma:=\partial\Omega$ is a compact Riemannian manifold, and if $q\in C^{0,1}(\Gamma)$ with $1\leq q^-\leq q^+<\infty$,
a simpler proof has been developed in \cite{DIAZ-VELEZ20}.\\

\begin{theorem}\,(see \cite{DIAZ-VELEZ20,GACZ-GORK13})\,\,\label{q-b}
Given $N\geq3$ and $q\in C^{0,1}(\Gamma)$ with $1\leq q^-\leq q^+<N-1$,
let $\Omega\subseteq\mathbb{R\!}^N$ be a bounded domain with Lipschitz boundary $\Gamma$. Then there exists a linear continuous mapping
$W^{1,q(\cdot)}(\Gamma)\hookrightarrow L^{^{\frac{(N-1)q(\cdot)}{N-1-q(\cdot)}}}(\Gamma)$ and a constant $c_3>0$ such that\\
\begin{equation}\label{2.2.03}
\|u\|_{_{\frac{(N-1)q(\cdot)}{N-1-q(\cdot)},\Gamma}}\,\leq\,c_3\|u\|_{_{W^{1,q(\cdot)}(\Gamma)}},
\indent\,\textit{for all}\,\,\,u\in W^{1,q(\cdot)}(\Gamma).
\end{equation}\indent\\
{\it Also, if $s\in\mathcal{P}(\Gamma)$ fulfills $1\leq s(x)\leq (N-1)q(x)(N-1-q(x))^{-1}$ for all $x\in\Gamma$, then
the linear mapping $W^{1,q(\cdot)}(\Gamma)\hookrightarrow L^{^{\frac{(N-1)q(\cdot)}{N-1-q(\cdot)}}}(\Gamma)$ is continuous.
Furthermore, if $\tilde{q}\in\mathcal{P}(\Gamma)$ is such that}\\
$$q(x)\leq\tilde{q}(x)\,\,\,\textit{for a.e.}\,\,\,x\in\Gamma\indent\textit{and}\indent
\textrm{ess}\displaystyle\inf_{x\in\Gamma}\left\{\frac{(N-1)q(x)}{N-1-q(x)}-\tilde{q}(x)\right\}>0,$$\indent\\
{\it then the embedding $W^{1,q(\cdot)}(\Gamma)\hookrightarrow L^{\tilde{q}(\cdot)}(\Gamma)$ is compact}.
\end{theorem}

Now, the anisotropic case follows easily by applying H\"older's inequality, together with Theorem \ref{q-b}.

\begin{theorem}(see \cite{DIAZ-VELEZ20})\label{turmanoua2cont}
	Let ${\Omega}\subset\mathbb{R\!}^N$ ($N\geq 3$) be a bounded $q_m(\cdot)$-extension domain and be a bounded domain with Lipschitz boundary $\Gamma$, for
	$q_i\in C(\Gamma)$ with $q_i^->1$ for all $i\in\{1,\dots,N-1\}$, with $q_m(\cdot)$ log-H\"older continuous
	over $\Gamma$ with $q_m^+<N-1$. If $q\in {\mathcal P}(\Gamma)$ with $q^->1$ satisfies the condition
	\begin{equation}\label{critical-b}
	q(\cdot) \,\leq\, \displaystyle\frac{(N-1)q_m(\cdot)}{N-1-q_m(\cdot)}:=q_m^\partial(\cdot),\qquad\textrm{over}\,\,\Gamma,
	\end{equation}
	then there is a continuous embedding
	$$W^{1,\overset{\rightarrow} q(\cdot)}(\Gamma) \hookrightarrow L^{q(\cdot)}
	({\Gamma}).$$
	Moreover, if $\tilde{q}\in\mathcal{P}(\Gamma)$ is such that
$$q(x)\leq\tilde{q}(x)\,\,\,\textit{for a.e.}\,\,\,x\in\overline{\Omega}\indent\textit{and}\indent
\textrm{ess}\displaystyle\inf_{x\in\Gamma}\left\{q^\partial(x)-\tilde{q}(x)\right\}>0,$$
then the embedding $W^{1,q(\cdot)}(\Gamma)\hookrightarrow L^{\tilde{q}(\cdot)}(\Gamma)$ is compact.
\end{theorem}

\begin{remark}\label{Rem-V-space}
The conclusions of Theorem \ref{tembcont} and Theorem \ref{turmanoua2cont} are clearly valid, if one
replace the Sobolev spaces by the space $\vpqx$ defined by (\ref{2.1.01}).
\end{remark}

\begin{remark}\label{ext-emb}
The conclusions of Theorem \ref{tembcont} and Theorem \ref{turmanoua2cont} are valid under a more general setting. In fact,
one can easily show that both theorems hold if one replaces the functions $p^{\star}_m(\cdot)$ and $q_m^\partial(\cdot)$ by the functions
$$\hat{p}_{_{M,m}}(\cdot):=\max\left\{p_m^{\star}(\cdot),\,p_M(\cdot)\right\}\indent\,\,\,\,\,\textrm{and}
\,\,\,\,\,\indent\check{q}_{_{M,m}}(\cdot):=\max\left\{q_m^\partial(\cdot),\,q_M(\cdot)\right\}.$$
\end{remark}

\indent We conclude this subsection by providing an anisotropic logarithmic inequality, which will be of use later on.
We recall that a similar inequality
has been obtained by Warma \cite[Lemma 2.5]{WAR09} for the constant isotropic case, and the generalization to the isotropic nonstandard case
runs similarly; cf. \cite[Lemma 3.2.1.]{VELEZ2012-3}. The result is unknown for the anisotropic case.\\

\begin{theorem}\label{log-sobolev}
Given $\Omega\subseteq\mathbb{R\!}^N$ a bounded Lipschitz domain ($N>2$),
let $(\overset{\rightarrow} p,\,p^-_m)\in C^{0,1}(\overline{\Omega})^{N+1}$ with $1<p^-_m<N$, and let
$(\overset{\rightarrow} q,\,q^-_m)\in C^{0,1}(\Gamma)^{N}$ with $1<q^-_m<N-1$.
\begin{enumerate}
\item[(a)]\,\, If $u\in W^{1,\overset{\rightarrow}p(\cdot)}(\Omega)$ is nonnegative with
\begin{equation}\label{2.2.24}
\psi_{_{\Omega}}(u^{p_m(x)}):=\displaystyle\int_{\Omega}u^{p_m(x)}\,dx\,=\,1,
\end{equation}\indent\\
then for every $\epsilon>0$ there exists a constant $C'_{\epsilon}>0$ such that for every $\epsilon_1>0$, one has\\
\begin{equation}\label{2.2.25}
\psi_{_{\Omega}}(u^{p_m(x)}\log(u))\,\leq\,\frac{1}{M_1}\left(-\log(\epsilon_1)+
C'_{\epsilon}\epsilon_1\displaystyle\sum^N_{i=1}\|\pxi u\|_{_{p_i(\cdot),\Omega}}+\epsilon_1C'_{\epsilon}\epsilon\right),
\end{equation}\indent\\
for some constant $M_1>0$.
\item[(b)]\,\, If $u\in W^{1,\overset{\rightarrow}q(\cdot)}(\Gamma)$ is nonnegative with
\begin{equation}\label{2.2.24b}
\psi_{_{\Gamma}}(u^{q_m(x)}):=\displaystyle\int_{\Gamma}u^{q_m(x)}\,dx\,=\,1,
\end{equation}\indent\\
then for every $\epsilon'>0$ there exists a constant $C''_{\epsilon'}>0$ such that for every $\epsilon_2>0$, one has\\
\begin{equation}\label{2.2.25b}
\psi_{_{\Gamma}}(u^{q_m(x)}\log(u))\,\leq\,\frac{1}{M_2}\left(-\log(\epsilon_2)+
C''_{\epsilon'}\epsilon_2\displaystyle\sum^{N-1}_{j=1}\|\pxti u\|_{_{q_j(\cdot),\Gamma}}+\epsilon_2C''_{\epsilon'}\epsilon'\right),
\end{equation}\indent\\
for some constant $M_2>0$.\\
\end{enumerate}
\end{theorem}

\begin{proof}
We first prove part (b). Put $s_m(\cdot):=\frac{q_m(\cdot)[q_m(\cdot)-1]}{N-1-q_m(\cdot)}$. Then by Jensen's inequality one has\\[2ex]
\indent$\psi_{_{\Gamma}}(u^{q_m(x)}\log(u))\,\leq\,\displaystyle\frac{1}{s^-_m}\log\left(\psi_{_{\Gamma}}(u^{q_m^\partial(x)})\right)\,\leq\,
\displaystyle\frac{1}{s^-_{m}\hat{q}_m^\partial}\log\left(\|u\|_{_{q_m^\partial(\cdot),\Gamma}}\right)$\\
\begin{equation}\label{2.2.26}
\indent\indent\indent\indent\indent\indent\indent\indent\indent
\,\leq\,\frac{1}{s^-_m\hat{q}_m^\partial}\left(-\log(\epsilon_2)+C_{_{\Gamma}}\epsilon_2\|u\|_{_{q_m^\partial(\cdot),\Gamma}}\right),
\end{equation}\indent\\
for every $\epsilon_2>0$ and for some constant $C_{_{\Gamma}}>0$, where
$$\hat{q}_m^\partial:=\left\{
\begin{array}{lcl}
\frac{(N-1)q^{+}_m}{N-1-q^{+}_m},\,\indent\,\textrm{if}\,\,\,\,\|u\|_{_{q_m^\partial(\cdot),\Gamma}}>1,\\[1ex]
\frac{(N-1)q^-_{m}}{N-1-q^-_{m}},\,\indent\,\textrm{if}\,\,\,\,\|u\|_{_{q_m^\partial(\cdot),\Gamma}}\leq1.\\[1ex]
\end{array}
\right.$$
Now, from Theorem \ref{turmanoua2cont} and (\ref{2.2.24b}) one has\\
\begin{equation}\label{2.2.27}
C_{\epsilon}\|u\|_{_{q_m^\partial(\cdot),\Gamma}}\,\leq\,\displaystyle\sum^{N-1}_{j=1}\|\pxti u\|_{_{q_j(\cdot),\Gamma}}+\epsilon,
\end{equation}\indent\\
for all $\epsilon>0$, and for some constant $C_{\epsilon}>0$. Therefore,
combining (\ref{2.2.26}) and (\ref{2.2.27}) yield the inequality (\ref{2.2.25b}). Now part (a)
follows in the exact way, with the help of Theorem \ref{tembcont}. The proof is complete.
\end{proof}

\subsection{Some analytical tools}\label{subsec2.3}

\indent Next we state some well-known analytical results that will be applied throughout the
subsequent sections.\\

\begin{prop}\label{B1}\,(see \cite{BIE10.1})\,
Let $a,\,b\in\mathbb{R\!}^N$\, and $r\in(1,\infty)$. Then there exists a constant $c_r>0$ such that
\begin{equation}
\label{2.3.1}\left(|a|^{r-2}a-|b|^{r-2}b\right)(a-b)\,\geq\,c_r\left(|a|+|b|\right)^{r-2}|a-b|^2\,\geq\,0.
\end{equation}
If in addition $r\in[2,\infty)$, then there exists a constant $c^{\ast}_r\in(0,1]$ such that
\begin{equation}
\label{2.3.2}\left(|a|^{r-2}a-|b|^{r-2}b\right)(a-b)\,\geq\,c^{\ast}_r\,|a-b|^r,
\end{equation}
and also in this case there is a constant $c'_r\in(0,1]$ such that
\begin{equation}
\label{2.3.3}\textrm{sgn}(a-b)\left(|a|^{r-2}a-|b|^{r-2}b\right)\,\geq\,c'_q\,|a-b|^{r-1}.
\end{equation}
Finally, if $r\in(1,2]$ and $\epsilon>0$, then for each $a,\,b\in\mathbb{R\!}^N$\, with $|a-b|\geq\epsilon\min\{|a|,|b|\}$,
we find a constant $c_{r,\epsilon}>0$ such that
\begin{equation}
\label{2.3.4}\left\langle |a|^{r-2}a-|b|^{r-2}b,a-b\right\rangle\,\geq\,c_{r,\epsilon}\,|a-b|^r,
\end{equation}
and
\begin{equation}
\label{2.3.5}\textrm{sgn}(a-b)\left(|a|^{r-2}a-|b|^{r-2}b\right)\,\geq\,c_{r,\epsilon}\,|a-b|^{r-1}.
\end{equation}
\end{prop}
\indent\\

\subsection{Basics concepts about the relative $p(\cdot)$-capacity}\label{subsec2.4}

\indent In this part we will summarize the necessary facts regarding the theory of the relative $p(\cdot)$-capacity, and its application
to measure theory. Recall that we will assume that $p\in C^{0,1}(\overline{\Omega})$ with $1<p^-\leq p^{+}<\infty$ (unless specified
differently). A complete treatment of this topic and its applications has been recently given in \cite{VELEZ2012-2} (under much
more general settings on the domain and on $p(\cdot)$; see also \cite[Appendix B]{VELEZ2012-3}).\\

\begin{definition}\label{Definition 2.4.1.}
Let $E\subseteq\mathbb{R\!}^N$ and let $r\in\mathcal{P}(\mathbb{R\!}^N)$ be such that
$1\leq r^-\leq r^{\ast}<\infty$. The $r(\cdot)${\bf -capacity} of $E$ is defined by
$$\displaystyle\textrm{Cap}_{r(\cdot)}(E):=\displaystyle\inf_{u\in S_{r(\cdot)}(E)}
\left\{\Theta_{_{r(\cdot),\mathbb{R\!}^N}}(u)+\Theta_{_{r(\cdot),\mathbb{R\!}^N}}(|\nabla u|)\right\},$$
where $$S_{r(\cdot)}(E):=\left\{u\in W^{1,r(\cdot)}(\mathbb{R\!}^N)\mid u\geq0\,\,\,\,\textrm{and}\,\,\,\,
u\geq1\,\,\,\,\textrm{in an open set containing}\,\,\,E\right\}.$$
\end{definition}

\begin{definition}\label{Definition 2.4.2.}
Let $\Omega\subseteq\mathbb{R\!}^N$ be an open set, and let $\widetilde{W}^{1,r(\cdot)}(\Omega)$
denote the closure in $W^{1,r(\cdot)}(\Omega)$ of the space $W^{1,r(\cdot)}(\Omega)\cap C_c(\overline{\Omega})$.
The {\bf relative $r(\cdot)$-capacity} of $E$ with respect to $\Omega$ is defined by
$$\displaystyle\textrm{Cap}_{_{r(\cdot),\Omega}}(E):=\displaystyle\inf_{u\in S_{_{r(\cdot),\Omega}}(E)}
\left\{\Theta_{_{r(\cdot),\Omega}}(u)+\Theta_{_{r(\cdot),\Omega}}(|\nabla u|)\right\},$$
where $$S_{_{r(\cdot),\Omega}}(E):=\left\{u\in\widetilde{W}^{1,r(\cdot)}(\Omega)\mid\,\exists\,\,\,
U\subseteq\mathbb{R\!}^N\,\,\,\textrm{open},\,\,\,\,E\subseteq U\,\,\,\textrm{and}\,\,\,u\geq1\,\,\,
\textrm{a.e. on}\,\,\,\Omega\cap U\right\}.$$
\end{definition}

\indent Observe that if $\Omega=\mathbb{R\!}^N$, then $\displaystyle\textrm{Cap}_{_{r(\cdot),\Omega}}(E)=
\displaystyle\textrm{Cap}_{r(\cdot)}(E)$ for each $E\subseteq\mathbb{R\!}^N$.\\

\begin{definition}\label{Definition 2.4.3.}
Let $\Omega\subseteq\mathbb{R\!}^N$ be an open set, and let $E\subseteq\overline{\Omega}$.
\begin{enumerate}
\item[(a)] $E$ is said to be $\displaystyle\textrm{Cap}_{_{r(\cdot),\Omega}}$-{\bf polar}, if $\displaystyle\textrm{Cap}_{_{r(\cdot),\Omega}}(E)=0$.
\item[(b)] We say that a property holds $\displaystyle\textrm{Cap}_{_{r(\cdot),\Omega}}$-{\bf quasi everywhere} (abbreviated $r(\cdot)$-q.e.) on $E$,
if there exists a $\displaystyle\textrm{Cap}_{_{r(\cdot),\Omega}}$-polar set $D$ such that the property holds over $E\setminus D$.
\item[(c)] A function $u$ is called $\displaystyle\textrm{Cap}_{_{r(\cdot),\Omega}}$-{\bf quasi continuous} on $E$, if for all $\epsilon>0$, there
exists a open set $U\subseteq\overline{\Omega}$ such that $\displaystyle\textrm{Cap}_{_{r(\cdot),\Omega}}(U)\leq\epsilon$ and
$u|_{_{E\setminus U}}$ is continuous.
\end{enumerate}
\end{definition}

\indent The following are important results, which we state here.\\

\begin{theorem}\label{Theorem 2.4.9.}\,(see \cite{VELEZ2012-2,VELEZ2012-3})\,\,
Given $p\in C^{0,1}(\overline{\Omega})$ with $1\leq p^-\leq p^+\leq\infty$,
if $\Omega\subseteq\mathbb{R\!}^N$ is a bounded $W^{1,p(\cdot)}$-extension domain, then there exists a constant $c_{_{\Omega}}=c(\Omega)>0$ such that
\begin{equation}\label{2.4.02}
\displaystyle\textrm{Cap}_{p(\cdot)}(E)\,\leq\,c_{_{\Omega}}\displaystyle\textrm{Cap}_{_{p(\cdot),\Omega}}(E)^{^{\hat{p}/\bar{p}}}
\,\leq\,c_{_{\Omega}}\displaystyle\textrm{Cap}_{p(\cdot)}(E)^{^{\hat{p}/\bar{p}}}
\end{equation}
for every set $E\subseteq\overline{\Omega}$, where $\hat{p}$ and $\bar{p}$ are positive constants.
\end{theorem}

\begin{corollary}\label{Corollary 2.4.10.}\,(see \cite{VELEZ2012-2,VELEZ2012-3})\label{Cor1}\,\,
Under the assumptions of Theorem \ref{Theorem 2.4.9.}, let $E\subseteq\overline{\Omega}$. Then $E$ is
$\displaystyle\textrm{Cap}_{_{p(\cdot),\Omega}}$-polar if and only if $E$ is
$\displaystyle\textrm{Cap}_{_{p(\cdot),\mathbb{R\!}^N}}$-polar.
\end{corollary}

\indent We conclude by presenting briefly a relation between the relative $p(\cdot)$-capacity with the measures defined on the
boundary of an open bounded set. We begin with the following definition.\\

\begin{definition}\label{Definition 2.4.11.}
Let $\Omega\subseteq\mathbb{R\!}^N$ be an open set with boundary $\Gamma$. We say that a Borel measure $\mu$ is\,
$\displaystyle\textrm{Cap}_{_{r(\cdot),\Omega}}$-{\bf admissible}, if $\displaystyle\textrm{Cap}_{_{r(\cdot),\Omega}}(F)=0$
implies that $\mu(F)=0$ for every Borel set $F\subseteq\Gamma$.
\end{definition}

\begin{example}\label{Example 2.4.12.}
If $\Omega\subseteq\mathbb{R\!}^N$ be a bounded Lipschitz domain and $p\in C^{0,1}(\overline{\Omega})$ with $1\leq p^-\leq p^+\leq\infty$, then, it is well-known that
$\Omega$ is a $W^{1,p(\cdot)}$-extension domain. Moreover, by \cite[Proposition 10.4.2]{D-H-H-R11}, for any $E\subseteq\overline{\Omega}$,
if $\,\displaystyle\textrm{Cap}_{p(\cdot)}(E)=0$, then $\mathcal{H}^s(E)=0$ for every $s>N-p_{\ast}$ (where $\mathcal{H}^s$ denotes the
$s$-dimensional Hausdorff measure). Therefore, from these arguments together with Corollary \ref{Cor1} and the fact that $\mathcal{H}^{N-1}=\sigma$ over
$\Gamma$, it follows that $\sigma$ is $\,\displaystyle\textrm{Cap}_{_{p(\cdot),\Omega}}$-admissible.\\
\end{example}

\subsection{Basic concepts for nonlinear semigroups}\label{subsec2.5}

\indent Now we put together some definitions and results on nonlinear semigroups. Indeed, let $H$ be a Hilbert space, and let
$\varphi:H\rightarrow(-\infty,\infty]$ be a proper, convex, lower semicontinuous functional with effective domain $D(\varphi):=
\{u\in H\mid\varphi(u)<\infty\}$. Clearly $D(\varphi)\subseteq H$ is convex. Then the subdifferential $\partial\varphi$ of $\varphi$ is defined by
$$\left\{\begin{array}{lcl}
D(\partial\varphi):=\left\{u\in D(\varphi)\mid\exists\,w\in H,\,\,\,
\varphi(v)-\varphi(u)\geq\langle w,v-u\rangle_{_H},\,\,\,\textrm{for all}\,\,\,v\in H\right\},\\
\partial\varphi(u):=\left\{w\in H\mid\varphi(v)-\varphi(u)\geq\langle w,v-u\rangle_{_H},\,\,\,\textrm{for all}\,\,\,v\in H\right\},\\
\end{array}
\right.$$ where $\langle\cdot,\cdot\rangle_{_H}$ denotes the inner product on $H$. The following classical result is fundamental.\\

\begin{theorem}\label{minty}\,(see \cite{MINT62})
\, The subdifferential $\partial\varphi$ is a maximal monotone operator. Moreover,
$\overline{D(\varphi)}=\overline{D(\partial\varphi)}$. The subdifferential $\partial\varphi$ generates a (nonlinear) $C_0$-semigroup
$\{T(t)\}_{t\geq0}$ on $\overline{D(\varphi})$ in the following sense: for each $u_0\in\overline{D(\varphi)}$,
the function $u:=T(\cdot)u_0$ is the unique strong solution of the problem
$$\left\{
\begin{array}{lcl}
u\in C(\mathbb{R\!}_{\,+};H)\cap W^{1,\infty}_{loc}((0,\infty);H)\,\,\,\textit{and}\,\,\,u(t)\in\partial\varphi\,\,\,\textit{a.e.},\\
\displaystyle\frac{\partial u}{\partial t}+\partial\varphi(u)\,=\,0\,\,\,\,\textit{a.e. on}\,\,\,\mathbb{R\!}_{\,+},\\
u(0,x)\,=u_0(x).
\end{array}
\right.$$
In addition, the subdifferential $\partial\varphi$ generates a (nonlinear) semigroup $\{\tilde{T}(t)\}_{t\geq0}$ on $H$, where for every $t\geq0$,
$\tilde{T}(t)$ is the composition of the semigroup $T(t)$ on $\overline{D(\varphi)}$ with the projection on the convex set $\overline{D(\varphi)}$.
\end{theorem}

\begin{definition}\label{NS01}
Let $\{T(t)\}_{t\geq0}$ be a (nonlinear) semigroup on a Hilbert lattice $H$ with ordering $\leq$,
let $X$ be a locally compact metric space, and $\nu$ a Borel regular measure on $X$.
\begin{enumerate}
\item[(a)]\,\,\,$\{T(t)\}_{t\geq0}$ is said to be {\bf order-preserving}, if
$$T(t)u\,\leq\,T(t)v\,\,\,\,\,\textrm{for all}\,\,\,t\geq0,\,\,\textrm{whenever}\,\,\,u,\,v\in H,\,\,u\leq v.$$
\item[(b)]\,\,\,$\{T(t)\}_{t\geq0}$ is called {\bf non-expansive} over $L^{q(\cdot)}(X,d\nu)$ ($q\in\mathcal{P}(X)$), if
$$\indent\indent\|T(t)u-T(t)v\|_{_{q(\cdot),X}}\,\leq\,\|u-v\|_{_{q(\cdot),X}},$$ for every $t\geq0$
and $u,\,v\in L^2(X,d\nu)\cap L^{q(\cdot)}(X,d\nu)$.
\item[(c)]\,\,\,$\{T(t)\}_{t\geq0}$ is said to be {\bf submarkovian}, if $\{T(t)\}_{t\geq0}$ is non-expansive over $L^{\infty}(X,d\nu)$.
\end{enumerate}
\end{definition}

\indent The next two well-known results characterize the order-preserve property and the submarkovian property of the
functional $\varphi$, respectively.\\

\begin{proposition}\label{CG01}
\,(see \cite{CIP-GR03})\, Let $\varphi:\;H\to(-\infty,+\infty]$ be a proper, convex,
lower semicontinuous functional on a real Hilbert lattice $H$, with effective domain $D(\varphi)$. Let
$\{T(t)\}_{t\geq0}$ be the (nonlinear) semigroup on $H$ generated by $\partial\varphi$. Then the following
assertions are equivalent.
\begin{enumerate}
\item[(i)] The semigroup $\{T(t)\}_{t\geq0}$ is order preserving.
\item[(ii)] For all $u,\,v\in H$ one has
$$\varphi\left(\frac{1}{2}\left(u+u\wedge v\right)\right)+\varphi\left(\frac{1}{2}\left(v+u\vee v\right)\right)\,\leq\,\varphi(u)+\varphi(v),$$
where $u\wedge v:=\inf\{u,v\}$ and $u\vee v:=\sup\{u,v\}$.
\end{enumerate}
\end{proposition}

\begin{proposition}\label{CG02}
\,(see \cite{CIP-GR03,NITTKA2010})\, Let $\varphi:L^2(X,d\nu)\rightarrow(-\infty,+\infty]$ be a proper,
convex, lower semicontinuous functional. Let $\{T(t)\}_{t\geq0}$ be the (nonlinear) semigroup on $L^2(X,d\nu)$ generated by
$A=\partial\varphi$. Assume that $\{T(t)\}_{t\geq0}$ is order preserving. Then, the following assertions are equivalent.
\begin{enumerate}
\item[(i)] The semigroup $\{T(t)\}_{t\geq0}$ is submarkovian.
\item[(ii)] For all $u,v\in L^2(X,d\nu)$ and $\alpha>0$,
$$\varphi\bigg(v+g_\alpha(u,v)\bigg)+\varphi\bigg(u-g_\alpha(u,v)\bigg)\,\leq\,\varphi(u)+\varphi(v),$$
where
$$g_{\alpha}(u,v):= \frac{1}{2}\bigg[(u-v+\alpha)^+-(u-v-\alpha)^-\bigg],$$
with $u^+:=\sup\{u,0\}$, and $u^-:=\sup\{-u,0\}$,
\item[(iii)] The operator $A$ is accretive on $L^{\infty}(X,d\nu)$, that is,
$$\|u-v\|_{_{\infty,X}}\,\leq\,\|(I-\eta A)(u-v)\|_{_{\infty,X}},\,\,\,\,\,\,\,\textit{for all}\,\,\,u,\,v
\in D(A)\,\,\,\textit{and}\,\,\,\eta>0.$$
\end{enumerate}
\end{proposition}

\indent The next definition brings the notion of ultracontractivity for nonlinear semigroups. If the semigroup is linear, then
the classical ultracontractivity property in the sense of \cite[Section 2]{DAV} coincide with the Lipschitz-ultracontractivity
property explained below.\\

\begin{definition}\label{NS02}
Let $X$ be a locally compact metric space, and $\nu$ a Borel regular measure on $X$.
Let $\{T(t)\}_{t\geq0}$ be a (nonlinear) submarkovian semigroup on $L^2(X,d\nu)$.
\begin{enumerate}
\item[(a)]\,\,\,$\{T(t)\}_{t\geq0}$ is said to be {\bf Lipschitz-utracontractive}, if there exist constants $c_1>0$ and $\alpha>1/2$
such that $$\|T(t)u-T(t)v\|_{_{\infty}}\,\leq\,c_1\,t^{-\alpha}\|u-v\|_{_{2}},\,\,\,\indent\,\textrm{for all}\,\,\,\,\,\,\,\,u,\,v\in L^2(X,d\nu)\,\,\,\,\textrm{and}\,\,\,\,t\in(0,1].$$
\item[(b)]\,\,\,$\{T(t)\}_{t\geq0}$ is said to be {\bf H\"older-utracontractive}, if there exist constants $c_1,\,\alpha>0$,
and $\gamma\in(0,1)$ such that $$\|T(t)u-T(t)v\|_{_{\infty}}\,\leq\,c_1\,t^{-\alpha}\|u-v\|^{\gamma}_{_{2}},\,\,\,\,\,\,\,\,\indent\,
\textrm{for all}\,\,\,\,u,\,v\in L^2(X,d\nu)\,\,\,\,\textrm{and}\,\,\,\,t\in(0,1].$$
\end{enumerate}
\end{definition}
\indent\\

\section{The parabolic problem}\label{sec3}

\indent In this section we turn our attention of the well-posedness of the quasi-linear parabolic problem with variable exponents
and pure Wentzell boundary conditions, on bounded anisotropic $p(\cdot)$-extension domains with Lipschitz boundaries.

\subsection{The Cauchy problem over $\mathbb{X\!}^{\,r(\cdot)}(\overline{\Omega})$}\label{subsec3.1}

\indent Let $\Omega\subseteq\mathbb{R\!}^N$ be a bounded anisotropic $p(\cdot)$-extension domain with Lipschitz boundary,
for $N\geq3$, and let $(\pp,\,\qq)\in C^{0,1}(\overline{\Omega})^N\times C^{0,1}(\overline{\Omega})^{N-1}$
be such that $1<p^-_m\leq p^+_M<\infty$ and $1<q^-_m\leq q^+_M<\infty$.
Given $\alpha\in L^{\infty}(\overline{\Omega})$ and $\beta\in L^{\infty}(\Gamma)$ with $\displaystyle\inf_{x\in\overline{\Omega}}\alpha(x)\geq\alpha_0$
and $\displaystyle\inf_{x\in\Gamma}\beta(x)\geq\beta_0$ for some constants $\alpha_0,\,\beta_0>0$,  and
$\mathbf{u}_0\in\mathbb{X\!}^{\,r(\cdot)}(\overline{\Omega})$ for $r\in\mathcal{P}(\overline{\Omega})$ with
$1\leq r^-\leq r^+\leq\infty$, we want to investigate the
well-posedness of the quasi-linear parabolic equation with fully dynamical anisotropic Wentzell boundary conditions, formally defined by\\
\begin{equation}
\label{4.1.00}\left\{
\begin{array}{lcl}
u_t-\Delta_{\pp}u+\alpha|u|^{p_M(\cdot)-2}u\,=\,0\indent\indent\indent\indent\indent\,\,\,\,\,\,\,\,\,
\,\,\,\,\,\,\,\,\,\,\,\indent\indent\textrm{in}\,\,\,\Omega\times(0,\infty),\\
u_t+\displaystyle\sum_{i=1}^N |\partial_{x_i}u|^{p_i(\cdot)-2}\partial_{x_i}u\; \nu_{i}-\Delta_{_{\qq,\Gamma}}u+
\beta|u|^{q_M(\cdot)-2}u\,=\,0\,\,\,\,\,\,\,\,\,\textrm{on}\,\,\,\Gamma\times(0,\infty),\\
u(0,x)\,=\,u_0(x)\indent\indent\indent\indent\indent\,\,\,\,\,\,\,\,
\,\,\,\,\indent\indent\,\,\,\,\,\,\,\,\,\indent\indent\indent\indent\indent\,\,\textrm{in}\,\,\,\overline{\Omega},\\
\end{array}
\right.
\end{equation}\indent\\
To begin with the discussion, we define the functional $\Phi_{\sigma}:\mathbb{X\!}^{\,2}(\overline{\Omega})\rightarrow[0,\infty]$ by
\begin{equation}\label{4.1.01}
\Phi_{\sigma}(\mathbf{u}):=\left\{
\begin{array}{lcl}
\mathcal{E}_{_{\Omega}}(\mathbf{u},\mathbf{u})+\mathcal{E}_{_{\Gamma}}(\mathbf{u},\mathbf{u}),
\,\,\,\,\,\textrm{if}\,\,\mathbf{u}\in D(\Phi_{\sigma}),\\[1.5ex]
+\infty,\,\,\,\,\,\,\,\,\,\,\,\,\,\textrm{if}\,\,\mathbf{u}\in\mathbb{X\!}^{\,2}(\overline{\Omega})\setminus D(\Phi_{\sigma}),\\
\end{array}
\right.
\end{equation}
with effective domain
$$D(\Phi_{\sigma}):=\vpqb\cap\mathbb{X\!}^{\,2}(\overline{\Omega}),$$\indent\\
where we recall that $\vpqb$ is defined by (\ref{Pair-Space}), and we split the functional into the interior part\\
$$\mathcal{E}_{_{\Omega}}(\mathbf{u},\mathbf{v}):=\displaystyle\int_{\Omega}\displaystyle\sum^N_{i=1}\left(\frac{|\pxi u|^{p_i(x)-2}\pxi u\pxi v}
{p_i(x)}\right)\,dx+\displaystyle\int_{\Omega}\alpha\frac{|u|^{p_M(x)-2}uv}{p_M(x)}\,dx$$
and the boundary part
$$\mathcal{E}_{_{\Gamma}}(\mathbf{u},\mathbf{v}):=\displaystyle\int_{\Gamma}\displaystyle\sum^{N-1}_{j=1}\left(\frac{|\pxti u|^{q_j(x)-2}\pxti u\pxti v}
{q_j(x)}\right)\,d\sigma+\displaystyle\int_{\Gamma}\beta\frac{|u|^{q_M(x)-2}uv}{q_M(x)}\,d\sigma,$$\indent\\
for each $\mathbf{u},\,\mathbf{v}\in D(\Phi_{\sigma})$. We start by establishing the following three results, whose arguments run in a
similar way as in \cite{VELEZ2012-3}.\\

\begin{lemma}\label{A01}
The functional $\Phi_{\sigma}$ defined in (\ref{4.1.01}) is a proper, convex,
lower semicontinuous functional on $\mathbb{X\!}^{\,2}(\overline{\Omega})$.
\end{lemma}

\begin{proof}
It is clear that $\Phi_{\sigma}$ is proper and convex. To establish lower semicontinuity, take a sequence
$\{\mathbf{u}_n\}\subseteq D(\Phi_{\sigma})$ such that $\displaystyle\lim_{n\rightarrow\infty}\mathbf{u}_n=(u,w)$ in $\mathbb{X\!}^{\,2}(\overline{\Omega})$.
If \,$\displaystyle\liminf_{n\rightarrow\infty}\Phi_{\sigma}(\mathbf{u}_n)=+\infty$ we are finished, so we assume that
\,$\displaystyle\liminf_{n\rightarrow\infty}\Phi_{\sigma}(\mathbf{u}_n)<+\infty$. Take a subsequence of $\{\mathbf{u}_n\}$\,,
(denoted in the same way by $\mathbf{u}_n$), such that $\displaystyle\lim_{n\rightarrow\infty}\Phi_{\sigma}(\mathbf{u}_n)$ equals a constant.
Since $\pxi u_n$ and $\pxti u_n$ are bounded over $L^{p_i(\cdot)}(\Omega)$ and $L^{q_j(\cdot)}(\Gamma)$ respectively, for each $i\in\{1,\ldots,N\}$
and $j\in\{1,\ldots,N-1\}$, it follows that the sequences $\{(\pxi u_n,\pxti u_n)\}$ are bounded on the space
$\mathbb{X\!}^{\,p_m(\cdot),q_m(\cdot)}(\overline{\Omega})$ for each $(i,j)\in\{1,\ldots,N\}\times\{1,\ldots,N-1\}$.
Thus, passing to subsequences if necessary, we may assume that
$\{\nabla u_n\}$ and $\{\nabla_{^{\Gamma}}u_n\}$ converge weakly in $L^{p_m(\cdot)}(\Omega)^N$ and $L^{q_m(\cdot)}(\Gamma)^{N-1}$, respectively.
Next, consider the Banach space $D(\Phi_{\sigma}):=\vpqb\cap\mathbb{X\!}^{\,2}(\overline{\Omega})$, endowed with the norm
$$|\|\mathbf{w}\||_{_{D(\Phi_{\sigma})}}:=|\|\mathbf{w}\||_{_{\vpqb}}+|\|\mathbf{w}\||_{_{2}},\indent
\textrm{for all}\,\,\,\mathbf{w}\in D(\Phi_{\sigma})$$
Then we see that $\{\mathbf{u}_n\}$ is a bounded sequence in the uniformly convex set $D(\Phi_{\sigma})$.
Then, we let $\mathbf{v}_n$ be a convex combination of $\mathbf{u}_n$, such that $v_n\stackrel{n\rightarrow\infty}{\longrightarrow}v$ in $\vpq\cap L^2(\Omega)$.
Then $\displaystyle\lim_{n\rightarrow\infty}\mathbf{v}_n=(v,\tilde{v})$ in $D(\Phi_{\sigma})$. Then $v_n\stackrel{n\rightarrow\infty}{\longrightarrow}v$ in $\widetilde{W}^{1,p_m(\cdot)}(\Omega)$. By the uniqueness of the limit we see
that $u=v$ a.e. on $\Omega$, and moreover it follows (by taking a subsequence if necessary) that $v_n\stackrel{n\rightarrow\infty}
{\longrightarrow}v$ $\,p_m(\cdot)$-q.e. on $\Gamma$ (see Definition \ref{Definition 2.4.3.}(b)).
As $\sigma$ is $\,\displaystyle\textrm{Cap}_{_{p_m(\cdot),\Omega}}$-admissible
(see Example \ref{Example 2.4.12.}), the previous convergence implies that
$v_n\stackrel{n\rightarrow\infty}{\longrightarrow}v$ $\,\sigma$-a.e. on $\Gamma$. Since in addition
$\displaystyle\lim_{n\rightarrow\infty}v_n=v$ in $L^{q_M(\cdot)}(\Gamma)$, by virtue of the uniqueness of the limit we
conclude that $\tilde{v}=v|_{_{\Gamma}}$ and $v|_{_{\Gamma}}=w$ $\,\sigma$-a.e. on $\Gamma$. If $w\neq u|_{_{\Gamma}}$, since $\{\mathbf{v}_n\}
\subseteq D(\Phi_{\sigma})$, we have $\Phi_{\sigma}(\mathbf{v}_n)<+\infty$, but $\Phi_{\sigma}(u,w)=+\infty$. However, since the operators
$\pxi$ and $\pxti$ are both closed in $L^{p_i(\cdot)}(\Omega)$ and $L^{q_j(\cdot)}(\Gamma)$, respectively (for each $i\in\{1,\ldots,N\}$
and $j\in\{1,\ldots,N-1\}$), it follows that
$\displaystyle\lim_{n\rightarrow\infty}\pxi v_n=\pxi u$ and $\displaystyle\lim_{n\rightarrow\infty}\pxti v_n=\pxti v$. Then, by virtue of the
Dominated Convergence Theorem of Lebesgue, one has $+\infty=\Phi_{\sigma}(u,w)=\displaystyle\lim_{n\rightarrow\infty}\Phi_{\sigma}(\mathbf{v}_n)<+\infty$,
a clear contradiction. Thus $w=u|_{_{\Gamma}}$. Finally, the convexity of $\Phi_{\sigma}$ entails that
$$\Phi_{\sigma}(\mathbf{u})\,=\,\displaystyle\liminf_{n\rightarrow\infty}\Phi_{\sigma}(\mathbf{v}_n)\,\leq\,
\displaystyle\liminf_{n\rightarrow\infty}\Phi_{\sigma}(\mathbf{u}_n),$$ and thus $\Phi_{\sigma}$ is
semicontinuous over $\mathbb{X\!}^{\,2}(\overline{\Omega})$, as desired.
\end{proof}

\indent In the next result we compute the subdifferential of the functional $\Phi_{\sigma}$, which is established in the same way as in
\cite{VELEZ2012-2,VELEZ2012-3}. Thus we will omit its proof.\\

\begin{lemma}\label{A02}
Let $\partial\Phi_{\sigma}$ be the subdifferential associated with $\Phi_{\sigma}$, and let $\mathbf{f}:=(f,f|_{_{\Gamma}})
\in\mathbb{X\!}^{\,2}(\overline{\Omega})$ and $\mathbf{u}\in D(\Phi_{\sigma})$. Then $\mathbf{f}\in\partial\Phi_{\sigma}(u)$ if and only if
\begin{equation}\label{4.1.02}
\left\{
\begin{array}{lcl}
\Delta_{\pp}u+\alpha|u|^{p_M(\cdot)-2}u\,=\,f\indent\indent\indent\indent\indent\indent\indent\indent\,\,\,\,\,
\indent\indent\indent\indent\indent\indent\,\,\,\,\,\,\,\indent\indent\indent\textit{in}\,\,\mathcal{D}(\Omega)^{\ast},\\
(\Delta_{\pp}u+\alpha|u|^{p_M(\cdot)-2}u)|_{_{\Gamma}}+\displaystyle\sum_{i=1}^N
|\partial_{x_i}u|^{p_i(\cdot)-2}\partial_{x_i}u\; \nu_{i}-\Delta_{_{\qq,\Gamma}}u
+\beta|u|^{q_M(\cdot)-2}u\,=\,0\,\,\,\,\,\,\textit{weakly on}\,\,\Gamma.\\
\end{array}
\right.
\end{equation}
\end{lemma}

\indent Now we to establish the well-posedness of the parabolic problem (\ref{4.1.00}).\\

\begin{theorem}\label{A03}
The operator $\mathcal{A}_{\sigma}:=\partial\Phi_{\sigma}$ generates a order-preserving submarkovian
$C_0$-semigroup $\left\{T_{\sigma}(t)\right\}_{t\geq0}$ on $\mathbb{X\!}^{\,2}(\overline{\Omega})$, which is
is non-expansive over $\mathbb{X\!}^{\,r(\cdot)}(\overline{\Omega})$ for all $r\in\mathcal{P}(\overline{\Omega})$ fulfilling $1\leq r^-\leq r^{+}\leq\infty$.
Consequently, for each $r\in\mathcal{P}(\overline{\Omega})$ with $1\leq r^-\leq r^{+}<\infty$, and for each $\mathbf{u}_0\in
\mathbb{X\!}^{\,r(\cdot)}(\overline{\Omega})$, the function $\mathbf{u}(\cdot)=T_{\sigma}(\cdot)\mathbf{u}_0$
is the (unique) strong solution of the parabolic problem (\ref{4.1.00}).
\end{theorem}

\begin{proof}
We begin by recalling that $\pp\in C^{0,1}(\overline{\Omega})^N$ and $\qq\in C^{0,1}(\Gamma)^{N-1}$
with $1<p^-_m\leq p^{+}_M<\infty$ and $1<q^-_m\leq q^+_M<\infty$. Since $\overline{D(\Phi_{\sigma})}=\overline{D(\partial\Phi_{\sigma})}
=\mathbb{X\!}^{\,2}(\overline{\Omega})$, we see that $D(\Phi_{\sigma})$ is dense in $\mathbb{X\!}^{\,2}(\overline{\Omega})$,
and hence by Theorem \ref{minty}, the operator $\mathcal{A}_{\sigma}:=
\partial\Phi_{\sigma}$ generates a (nonlinear) $C_0$-semigroup $\left\{T_{\sigma}(t)\right\}_{t\geq0}$ on $\mathbb{X\!}^{\,2}(\overline{\Omega})$.
Hence, for each $\mathbf{u}_0\in\mathbb{X\!}^{\,2}(\overline{\Omega})$, the function $\mathbf{u}(\cdot):=
T_{\sigma}(\cdot)\mathbf{u}_0$ solves (uniquely) the abstract Cauchy problem
\begin{equation}
\label{4.1.03}\left\{
\begin{array}{lcl}
\mathbf{u}\in C(\mathbb{R\!}_{\,+};\mathbb{X\!}^{\,2}(\overline{\Omega}))\cap W^{1,\infty}_{loc}((0,\infty);
\mathbb{X\!}^{\,2}(\overline{\Omega}))\,\,\,\textrm{and}\,\,\,\mathbf{u}(t)\in\partial\Phi_{\sigma}\,\,\,\textit{a.e.},\\
\displaystyle\frac{\partial\mathbf{u}}{\partial t}+\partial\Phi_{\sigma}(\mathbf{u})\,=\,0\,\,\,\,\textit{a.e. on}\,\,\,\mathbb{R\!}_{\,+},\\
\mathbf{u}(0,x)\,=\mathbf{u}_0(x).
\end{array}
\right.
\end{equation}
To complete the proof we now show that $\left\{T_{\sigma}(t)\right\}_{t\geq0}$ is order-preserving and non-expansive.
Given $\mathbf{u},\,\mathbf{v}\in D(\Phi_{\sigma})$, then to prove the order-preserving property, define
	$$g_{_{u,v}}:=\frac{1}{2}(u+u\wedge v)\indent\,\,\,\textrm{and}\,\,\,\indent h_{_{u,v}}:=\frac{1}{2}(v+u\vee v)$$
	where $u\wedge v:=\min\{u,v\}$ and $u\vee v:=\max\{u,v\}$. One can note that
	$$	
	g_{_{u,v }}= \left\{ \begin{array}{lcc}
	u &   \mbox{ if }  & u \leq v \\
	\dfrac{u+v}{2} &  \mbox{ if }     & u > v \\
	\end{array}
	\right.
	\qquad
	\mbox{ and }
	\qquad
	h_{_{u,v}}= \left\{ \begin{array}{lcc}
	v &   \mbox{ if }  & u \leq v \\
	\dfrac{u+v}{2} &  \mbox{ if }     & u>v \\
	\end{array}
	\right.
	.
	$$	
	Since it is clear that $D(\Phi_{\sigma})$ is a lattice, we have that
$\mathbf{g}_{_{u,v}},\,\mathbf{h}_{_{u,v}}\in D(\Phi_{\sigma})$. Moreover, the following calculation holds:\\[2ex]
	$\Phi_{\sigma}(\mathbf{g}_{_{u,v}})+\Phi_{\sigma}(\mathbf{h}_{_{u,v}})$\\
	$$=\int_{\Omega}\displaystyle\sum^N_{i=1} \left( \frac{
		|\partial_{x_i} g_{u,v}|^{p_i(x)}+|\partial_{x_i} h_{u,v}|^{p_i(x)}}{p_i(x)} \right)\, dx+
	\int_{\Omega}\frac{\alpha(x)}{p_M(x)}\left(|g_{u,v}|^{^{p_M(x)}}+|h_{u,v}|^{^{p_M(x)}}\right)\,dx+\indent\,\,\,\,$$
	$$\int_{\Gamma}\displaystyle\sum^{N-1}_{j=1} \left( \frac{
		|\pxti g_{u,v}|^{q_j(x)}+|\pxti h_{u,v}|^{q_j(x)}}{q_j(x)} \right) \,d\sigma	+\int_{\Gamma}\frac{\beta(x)}{q_M(x)}\left(|g_{u,v}|^{q_M(x)}+|h_{u,v}|^{q_M(x)}\right)\,d\sigma$$
	$$=\int \limits_{\Omega \,  \cap  \, \left\{u \leq v \right\}}\displaystyle\sum^N_{i=1}
	\left( \frac{	|\partial_{x_i} u|^{p_i(x)}+|\partial_{x_i} v|^{p_i(x)}}{p_i(x)} \right) \,dx+
	\int \limits_{\Omega \, \cap \,  \left\{u \leq v \right\}} \frac{\alpha(x)}{p_M(x)}\left(|u|^{^{p_M(x)}}+|v|^{^{p_M(x)}}\right)\,dx+\indent\,\,\,\,$$
	$$\int \limits_{\Gamma \cap \left\{u \leq v \right\}}\displaystyle\sum^{N-1}_{j=1}\frac{
		\left(|\pxti u|^{q_j(x)}+|\pxti v|^{q_j(x)} \right)}{q_j(x)}\,d\sigma	+
	\int \limits_{\Gamma \cap \left\{ u \leq v \right\}}\frac{\beta(x)}{q_M(x)}\left(|u|^{q_M(x)}+|v|^{q_M(x)}\right)\,d\sigma \, +\,$$	
	$$2\int \limits_{\Omega \,\cap \, \left\{u > v \right\}}\displaystyle\sum^N_{i=1}
	\dfrac{1}{p_i(x)} \left( \dfrac{\left|\partial_{x_i} \left( {u + v} \right) \right|}{2}\right)^{p_i(x)} \,dx \, + \,
	2\int \limits_{\Omega \, \cap \, \left\{u > v \right\}} \frac{\alpha(x)}{p_M(x)}\left(\dfrac{|u+v|}{2}\right)^{p_M(x)}\,dx \,+ \,\indent\,\,\,\,$$
	$$2\int \limits_{\Gamma \,  \cap \, \left\{u > v \right\}}\displaystyle
	\sum^{N-1}_{j=1} \dfrac{1}{q_j(x)}
	\left( \dfrac{|\pxti (u + v)|}{2}   \right)^{q_j(x)}  \,  d\sigma	+
	2\int \limits_{\Gamma \, \cap \, \left\{ u > v \right\}}
	\frac{\beta(x)}{q_M(x)}\left(\dfrac{|u + v|}{2}\right)^{q_M(x)}  \,d\sigma.$$
	Using the well-known equality  \,
	\begin{equation} \label{inequality}
	|a+b|^s\leq2^{s-1} \left(|a|^s+|b|^s \right),
	\end{equation}
	 valid for all $a,\,b\in\mathbb{R\!}^N$ and for all
	$s\in[1,\infty)$, the above calculation implies that
	$$\Phi_{\sigma}(\mathbf{g}_{u,v}) + \Phi_{\sigma} (\mathbf{h}_{u,v}) \leq \Phi_{\sigma}(\mathbf{u}) + \Phi_{\sigma}(\mathbf{v}).$$
	Thus, a direct application of Proposition \ref{CG01} asserts that $\left\{  T_{\sigma}(t) \right\}_{t\geq 0}$ is order preserving.\\
	
Next, given $\xi>0$, we put
	\begin{eqnarray}
	g_{_{u,v,\xi}} :=  \left[  \dfrac{1}{2}  \left(u-v+\xi \right)_{+} -\left( u-v-\xi \right)_{-}   \right],
	\end{eqnarray}
	and observe that
	\begin{eqnarray}
	g_{_{u,v,\xi}}
	=
	\left\lbrace
	\begin{array}{lclcl}
	u-v,  &\mbox{ if }&  |u-v|\leq  \xi    \\
	& & \\
	\dfrac{u-v-\xi}{2},    &\mbox{ if }& u-v < - \xi \\
	& & \\
	\dfrac{u-v+\xi}{2}, &\mbox{if} & u-v > \xi.
	\end{array}
	\right.
	\end{eqnarray}
Again, it is clear that $\mathbf{g}_{_{u,v,\xi}}\in D(\Phi_{\sigma})$.
	Then, we first deal with the terms involving partial derivatives. In fact, we have: \\[2ex]
	$\displaystyle 	\int_{\Omega} \sum_{i=1}^{N}  \dfrac{\left| \partial_{x_i} (v+ g_{_{u,v,\xi}}) \right|^{p_i(x)} }{p_i(x)} dx$
	\begin{eqnarray*}
		&=\displaystyle
		\int \limits_{ {\Omega \cap \left\{|u-v|\leq \xi \right\} }} \sum_{i=1}^{N}  \dfrac{\left| \partial_{x_i} (v+ g_{_{u,v,\xi}}) \right|^{p_i(x)} }{p_i(x)} dx +
		\int \limits_{ {\Omega \cap  \left\{|u-v|\geq \xi \right\} } }  \sum_{i=1}^{N}  \dfrac{\left| \partial_{x_i} (v+ g_{_{u,v,\xi}}) \right|^{p_i(x)} }{p_i(x)} dx
	\end{eqnarray*}
	\begin{eqnarray*}
		&=& \displaystyle \int \limits_{ {\Omega \cap \left\{|u-v|\leq \xi \right\} }}
		\sum_{i=1}^{N}  \dfrac{\left| \partial_{x_i} u \right|^{p_i(x)} }{p_i(x)} dx+
		\int \limits_{{ \Omega \cap \left\{ u-v < - \xi \right\}}} \sum_{i=1}^{N}  \dfrac{\left| \partial_{x_i}  \left( \frac{u+v-\xi}{2} \right)  \right|^{p_i(x)} }{p_i(x)} dx \\
		&+&
		\int \limits_{ {\Omega \,  \cap \, \left\{u-v > -\xi \right\} }}
		\sum_{i=1}^{N}  \dfrac{\left| \partial_{x_i} \left( \frac{u+v+\xi}{2} \right) \right|^{p_i(x)} }{p_i(x)} dx  \\
		&=&
		\displaystyle \int \limits_{ {\Omega \, \cap \, \left\{|u-v|\leq \xi \right\} }}
		\sum_{i=1}^{N}  \dfrac{\left| \partial_{x_i} u \right|^{p_i(x)} }{p_i(x)} dx+
		\int \limits_{ {\Omega \, \cap \, \left\{ u-v< -\xi \right\} }} \sum_{i=1}^{N}  \dfrac{\left| \partial_{x_i}  \left( \frac{u+v}{2} \right)  \right|^{p_i(x)} }{p_i(x)} dx \\
		&+&
		\int \limits_{ {\Omega \, \cap \, \left\{u-v> \xi \right\} }}
		\sum_{i=1}^{N}  \dfrac{\left| \partial_{x_i} \left( \frac{u+v}{2} \right) \right|^{p_i(x)} }{p_i(x)} dx.
	\end{eqnarray*}
	Similarly,\\[2ex]
	$ \displaystyle    \int_{\Gamma} \sum_{j=1}^{N-1}  \dfrac{\left|\pxti(u- g_{_{u,v,\xi}}) \right|^{q_j(x)} }{q_j(x)} d\sigma$
	\begin{eqnarray*}
		&=&
		\displaystyle \int \limits_{ {\Gamma \, \cap \, \left\{|u-v|\leq \xi \right\} }}
		 \sum_{j=1}^{N}  \dfrac{\left| \pxti v \right|^{q_j(x)} }{q_j(x)} d\sigma+
		\int \limits_{ {\Gamma \, \cap \, \left\{u-v < - \xi \right\} }}
		\sum_{j=1}^{N}  \dfrac{\left| \pxti \left( \frac{u+v}{2} \right)  \right|^{q_j(x)} }{q_j(x)} d\sigma \\
		&+&
		\int \limits_{ {\Gamma \,  \cap \, \left\{u-v > \xi \right\} }}
		\sum_{j=1}^{N-1}  \dfrac{\left| \pxti \left( \frac{u+v}{2} \right) \right|^{q_j(x)} }{q_j(x)} d\sigma.
	\end{eqnarray*}
	Adding the two last results and applying the inequality (\ref{inequality}), we have
\begin{equation}\label{2.10}
\displaystyle\int_{\Omega} \sum_{i=1}^{N}  \dfrac{\left| \partial_{x_i} (v+ g_{_{u,v,\xi}}) \right|^{p_i(x)} }{p_i(x)} dx +  \displaystyle\int_{\Gamma} \sum_{j=1}^{N-1}  \dfrac{\left| \pxti(u- g_{_{u,v,\xi}}) \right|^{q_j(x)} }{q_j(x)} d\sigma\,\leq\,\displaystyle\int_{\Omega} \sum_{i=1}^{N}  \dfrac{\left| \partial_{x_i} u \right|^{p_i(x)} }{p_i(x)} dx +\displaystyle\int_{\Gamma} \sum_{j=1}^{N-1}  \dfrac{\left| \pxti v \right|^{q_j(x)} }{q_j(x)} d\sigma.
\end{equation}
	For the terms not involving the partial derivatives, given $\xi > 0$, we put
	\begin{eqnarray}
	\lambda_{\xi} :=  \chi_{_{\left\{ u \neq v  \right\} } } \dfrac{g_{_{u,v,\xi}} }{u-v} .
	\end{eqnarray}	
	Then $\lambda_{\xi} $ satisfies the condition $0 \leq \lambda_{\xi} \leq 1, $ with
	\begin{eqnarray}
	v+g_{_{u,v,\xi}} =	\lambda_{\xi} u + (1- \lambda_{\xi} )v \qquad
	\mbox{ and }
	u-g_{_{u,v,\xi}} = \lambda_{\xi} v + (1- \lambda_{\xi} ) u.
	\end{eqnarray}	
	Since the maps $t \longmapsto |t|^{p_M(\cdot)}$	 and $t \longmapsto |t|^{q_M(\cdot)}$ are convex, we get
	\begin{eqnarray*}
		& & \int_{\Omega} \dfrac{\alpha (x)}{ p_{_M}(x) } \left( |v+ g_{_{u,v,\xi}} |^{p_{_{M}}(x)} + |u-g_{u,v,\xi}|^{p_M(x)} \right)dx \\
		&+&
		\int_{\Gamma} \dfrac{\beta (x)}{ q_M(x) } \left( |v+ g_{_{u,v,\xi}} |^{q_{_M}(x)} + |u-g_{u,v,\xi}|^{q_{_M}(x)} \right)d\sigma \\
		&\leq&
		\lambda_{\xi} \int_{\Omega}  \dfrac{\alpha(x)}{p_{_M}(x)} \left(  |u|^{p_{_M}(x)}  +  |v|^{q_{M}(x)}  \right) dx
		+
		(1-\lambda_{\xi}) \int_{\Omega} \dfrac{\alpha(x) }{p_{_M}(x)} \left(  |u|^{p_{_M}(x)} + |v|^{p_{_M}(x)} \right) dx \\
		&+&
		\lambda_{\xi} \int_{\Omega}  \dfrac{\beta(x)}{q_{_M}(x)} \left(  |u|^{q_{_M}(x)}  +  |v|^{q_{M}(x)}  \right) d \sigma
		+
		(1-\lambda_{\xi}) \int_{\Omega} \dfrac{\beta(x) }{q_{_M}(x)} \left(  |u|^{q_{_M}(x)} + |v|^{q_{_M}(x)} \right)  d\sigma\\
		& \leq &
		\int_{\Omega} \dfrac{\alpha(x)}{p_{_M}(x)} \left(  |u|^{p_{_M}(x)} + |v|^{p_{_M}(x)} \right)dx +
		\int_{\Gamma} \dfrac{\beta(x)}{q_{_M}(x)} \left(  |u|^{q_{_M}(x)} + |v|^{q_{_M}(x)} \right)d\sigma.
	\end{eqnarray*}
	Combining (\ref{2.10}) and this last calculation yields  	
	\begin{eqnarray}
	\Phi_{\sigma} (\mathbf{u}- \mathbf{g}_{_{u,v,\xi}}) + \Phi_{\sigma}(\mathbf{v} + \mathbf{g}_{_{u,v,\xi}})
\,\leq\, 	 \Phi_{\sigma} (\mathbf{u}) + \Phi_{\sigma}(\mathbf{v}),
	\end{eqnarray}	
and thus $\left\{T_{\sigma}(t)\right\}_{t\geq0}$ is non-expansive on $\mathbb{X\!}^{\,\infty}(\overline{\Omega})$ by virtue of Proposition \ref{CG02}.
By \cite[Theorem 1]{BROW69} and \cite[Corollary 3]{MAL89} we see that $\left\{T_{\sigma}(t)\right\}_{t\geq0}$ can be extended to a non-expansive semigroup on
$\mathbb{X\!}^{\,r(\cdot)}(\overline{\Omega})$ for every $r\in\mathcal{P}(\overline{\Omega})$ with $1\leq r^-\leq r^{+}\leq\infty$.
To see the strong continuity of $\left\{T_{\sigma}(t)\right\}_{t\geq0}$ over $\mathbb{X\!}^{\,r(\cdot)}(\overline{\Omega})$
(for $r\in\mathcal{P}(\overline{\Omega})$ with $1\leq r^-\leq r^{+}<\infty$), first take
$\mathbf{u}\in\mathbb{X\!}^{\,2}(\overline{\Omega})\cap\mathbb{X\!}^{\,\infty}(\overline{\Omega})$, let $\hat{r}:=\min\{2,r^-\}$, and let
$a(x):=\frac{\hat{r}}{r(x)}\left(\frac{r^{+}-r(x)}{r^{+}-\hat{r}}\right),\,\,
b(x):=\frac{r^{+}}{r(x)}\left(\frac{r(x)-\hat{r}}{r^{+}-\hat{r}}\right)$ for all $x\in\overline{\Omega}$.
Then by \cite[Corollary 2.2]{KOV-RAK91} and H\"older's inequality (if necessary) we obtain that for $t>0$,
\begin{equation}\label{4.1.04}
|\|T_{\sigma}(t)\mathbf{u}-\mathbf{u}\||_{_{r(\cdot)}}\,\leq\,c|\|T_{\sigma}(t)\mathbf{u}-\mathbf{u}\||^{\gamma_1}_{_{2}}\,
|\|T_{\sigma}(t)\mathbf{u}-\mathbf{u}\||^{\gamma_2}_{_{\infty}}\,\,\stackrel{t\rightarrow0^+}{\longrightarrow}\,0,
\end{equation}
where
$$\gamma_1:=\left\{
\begin{array}{lcl}
a^{+},\,\indent\,\textrm{if}\,\,\,\,|\|T_{\sigma}\mathbf{u}-\mathbf{u}\||_{_{2}}>1,\\
a^{-},\,\indent\,\textrm{if}\,\,\,\,|\|T_{\sigma}\mathbf{u}-\mathbf{u}\||_{_{2}}\leq1,\\
\end{array}
\right.
\,\,\,\,\,\,\textrm{and}\,\,\,\,\,\,\gamma_2:=\left\{
\begin{array}{lcl}
b^{+},\,\indent\,\textrm{if}\,\,\,\,|\|T_{\sigma}\mathbf{u}-\mathbf{u}\||_{_{r^{+}}}>1,\\
b^{-},\,\indent\,\textrm{if}\,\,\,\,|\|T_{\sigma}\mathbf{u}-\mathbf{u}\||_{_{r^{+}}}\leq1.\\
\end{array}
\right.$$
Finally, for a general $\mathbf{u}\in\mathbb{X\!}^{\,r(\cdot)}(\overline{\Omega})$, fix $\epsilon>0$ and choose
$\mathbf{v}\in\mathbb{X\!}^{\,2}(\overline{\Omega})\cap\mathbb{X\!}^{\,\infty}(\overline{\Omega})$ such that
$|\|\mathbf{u}-\mathbf{v}\||_{_{r(\cdot)}}<\epsilon/3$. Then for a sufficiently small $t>0$, we get from (\ref{4.1.04}) that
\begin{equation}\label{4.1.05}
|\|T_{\sigma}(t)\mathbf{u}-\mathbf{u}\||_{_{r(\cdot)}}\,\leq\,|\|T_{\sigma}(t)\mathbf{u}-T_{\sigma}(t)\mathbf{v}\||_{_{r(\cdot)}}+
|\|T_{\sigma}(t)\mathbf{v}-\mathbf{v}\||_{_{r(\cdot)}}+|\|\mathbf{u}-\mathbf{v}\||_{_{r(\cdot)}}<\epsilon.
\end{equation}
Since $\epsilon>0$ was chosen arbitrary, the inequality (\ref{4.1.05}) implies that $\left\{T_{\sigma}(t)\right\}_{t\geq0}$
is a $C_0$-semigroup over $\mathbf{u}\in\mathbb{X\!}^{\,r(\cdot)}(\overline{\Omega})$, and this completes the proof.
\end{proof}

\begin{remark}
The results of this section can be established inthe same way under more general assumptions on the vector fields $\pp$ and $\qq$. In fact, for this part
it suffices to have $\pp\in \mathcal{P}^{\log}(\overline{\Omega})^N$ and $\qq\in\mathcal{P}^{\log}(\Gamma)^{N-1}$, where
$\mathcal{P}^{\log}(\overline{\Omega})$ denotes the set of functions $u\in\mathcal{P}(\overline{\Omega})$
such that the function $v:=1/u$ is globally log-H\"older continuous, that is,
if there exist constants $c_1,\,c_2>0$ and a constant $\alpha\in\mathbb{R\!}\,$ such that
$$|v(x)-v(y)|\,\leq\,\frac{c_1}{\log(e+1/|x-y|)}\indent\,\textrm{and}\indent\,|v(x)-\alpha|\,\leq\,
\frac{c_2}{\log(e+|x|)}$$ for all $x,\,y\in\overline{\Omega}$.
\end{remark}

\subsection{Ultracontractivity of semigroups}\label{subsec3.2}

\indent In this section we develop ultracotractivity properties for the nonlinear semigroup $\left\{T_{\sigma}(t)\right\}_{t\geq0}$,
under some additional conditions over $\pp$ and $\qq$ (we are assuming all the conditions in subsection \ref{subsec3.1}).
Cipriani and Grillo \cite{CIP-GR01,CIP-GR02} developed some key
tools to develop ultracontractivity properties for Dirichlet-type nonlinear semigroups. Such methods were generalized and extended
to a dynamical Robin-Wentzell type problem by Warma \cite{WAR09}, and these were extended further in \cite{VELEZ2012-3} to the case of isotropic
Robin-Wentzell and pure Wentzell differential equations with nonstandard growth conditions (for $p(\cdot)=q(\cdot)$).
In our problem, we will follow the approach in \cite{VELEZ2012-3,WAR09}; however,
in our case we are dealing with anisotropic operators involving different unrelated vector fields, one defined in the interior, and the other at the boundary.
These facts bring new difficulties, which require substantial modifications in the arguments previously applied by multiple authors.
Consequently, we provide complete proofs to all the results in this subsection.\\
\indent We begin with three technical lemmas which are needed in order to establish our desired ultracontractivity bounds.\\

\begin{lemma}\label{Lem01}
Let $\left\{T_{\sigma}(t)\right\}_{t\geq0}$ be the submarkovian $C_0$-semigroup on
$\mathbb{X\!}^{\,2}(\overline{\Omega})$ generated by $\partial\Phi_{\sigma}$. Given $t\geq0$ and $\mathbf{u}_0,\,\mathbf{v}_0\in
\mathbb{X\!}^{\,\infty}(\overline{\Omega})$, put $u(t):=T_{\sigma}(t)\mathbf{u}_0$ and $\mathbf{v}(t):=T_{\sigma}(t)\mathbf{v}_0$.
Then for every real number $r\geq2$ and for almost every $t\geq0$, there exists a
constant $C_1>0$ such that\\[2ex]
$\displaystyle\frac{d}{dt}|\|\mathbf{u}(t)-\mathbf{v}(t)\||^r_{_{r}}\,
\leq\,-C_1r\left((r-1)\displaystyle\int_{\Omega}|u(t)-v(t)|^{r-2}\displaystyle\sum^N_{i=1}|\pxi(u(t)-v(t))|^{p_i(x)}\,dx+
\displaystyle\int_{\Omega}\alpha|u(t)-v(t)|^{r+p_M(x)-2}dx\right)+$\\
\begin{equation}\label{4.2.01}
\,\,\,\,\,\,\,\,\,-C_1r\left((r-1)\displaystyle\int_{\Gamma}|u(t)-v(t)|^{r-2}\displaystyle\sum^{N-1}_{j=1}\left|\pxti(u(t)-v(t))\right|^{q_j(x)}\,d\sigma+
\displaystyle\int_{\Gamma}\beta|u(t)-v(t)|^{r+q_M(x)-2}d\sigma\right).
\end{equation}
\end{lemma}
\indent\\
\begin{proof}
Given $t\geq0$ and $\mathbf{u}_0,\,\mathbf{v}_0\in
\mathbb{X\!}^{\,\infty}(\overline{\Omega})$, we set $\mathbf{u}(t):=T_{\sigma}(t)\mathbf{u}_0$ and
$\mathbf{v}(t):=T_{\sigma}(t)v_0$. Then $\mathbf{u}(t),\,\mathbf{v}(t)\in D(\Phi_{\sigma})\cap L^{\infty}(\Omega)$.
For $r\geq2$, we define the function $G_r:[0,\infty)\rightarrow[0,\infty)$ by
$$G_r(t):=|\|\mathbf{u}(t)-\mathbf{v}(t)\||_{_{r}}.$$ Clearly $G_r$ is differentiable for almost all $t\geq0$.
Setting $\mathbf{w}(t):=\mathbf{u}(t)-\mathbf{v}(t)$, using the fact that $\mathbf{u}(t),\,\mathbf{v}(t)$
solve the Eq. (\ref{4.1.03}) (and thus also Eq. (\ref{4.1.00}))
with the respective initial data $\mathbf{u}_0$ and $\mathbf{v}_0$, taking into account (\ref{4.1.02}), making integration
by parts, and applying the inequality (\ref{2.3.2}), we obtain the following calculation.\\[2ex]
$G'_r(t)=r\displaystyle\int_{\Omega}|u(t)-v(t)|^{r-1}(u'(t)-v'(t))\textrm{sgn}(u(t)-v(t))\,dx$\\[2ex]
$=-r\displaystyle\int_{\Omega}|w(t)|^{r-1}[\partial\Phi_{\sigma}(u(t))-\partial\Phi_{\sigma}(v(t))]\textrm{sgn}(w(t))\,dx$\\[2ex]
$=r\displaystyle\int_{\Omega}|w(t)|^{r-1}\left(\Delta_{\pp}(u(t))-\alpha|u(t)|^{p_M(x)-2}u(t)\right)\textrm{sgn}(w(t))\,dx
-r\displaystyle\int_{\Omega}|w(t)|^{r-1}\left(\Delta_{\pp}(v(t))-\alpha|v(t)|^{p_M(x)-2}v(t)\right)\textrm{sgn}(w(t))\,dx+$\\
$$-r\displaystyle\int_{\Gamma}|w(t)|^{r-1}\left(\displaystyle\sum_{i=1}^N |\partial_{x_i}u(t)|^{p_i(\cdot)-2}\partial_{x_i}u\; \nu_{i}
-\Delta_{_{\qq,\Gamma}}u(t)+\beta|u(t)|^{q_M(x)-2}u(t)\right)\textrm{sgn}(w(t))\,d\sigma+$$
$$+r\displaystyle\int_{\Gamma}|w(t)|^{r-1}\left(\displaystyle\sum_{i=1}^N |\partial_{x_i}v(t)|^{p_i(\cdot)-2}\partial_{x_i}v(t)\; \nu_{i}
-\Delta_{_{\qq,\Gamma}}v(t)+\beta|v(t)|^{q_M(x)-2}v(t)\right)\textrm{sgn}(w(t))\,d\sigma$$
$=-r(r-1)\displaystyle\int_{\Omega}|w(t)|^{r-2}\displaystyle\sum_{i=1}^N\left(|\pxi u(t)|^{p_i(\cdot)-2}\pxi u-
|\pxi v(t)|^{p_i(x)-2}\pxi v(t)\right)\pxi w(t)\,dx+$\\[2ex]
\indent\indent\indent\indent\indent$-r\displaystyle\int_{\Omega}|w(t)|^{r-2}\left(\alpha|u(t)|^{p_M(x)-2}u(t)-\alpha|v(t)|^{p_M(x)-2}v(t)\right)(w(t))\,dx+$\\
$$-r(r-1)\displaystyle\int_{\Gamma}|w(t)|^{r-2}\displaystyle\sum_{j=1}^{N-1}\left(|\pxti u(t)|^{q_j(\cdot)-2}\pxti u-
|\pxti v(t)|^{q_j(x)-2}\pxti v(t)\right)\pxti w(t)\,d\sigma+$$
$$\indent\indent\indent\indent-r\displaystyle\int_{\Gamma}|w(t)|^{r-2}
\left(\beta|u(t)|^{q_M(x)-2}u(t)-\beta|v(t)|^{q_M(x)-2}v(t)\right)(w(t))\,d\sigma$$
$$\leq\,-C_1r\left((r-1)\displaystyle\int_{\Omega}|w(t)|^{r-2}\displaystyle\sum_{i=1}^N|\pxi w(t)|^{p_i(x)}\,dx+
\displaystyle\int_{\Omega}\alpha|w(t)|^{r-2}|w(t)|^{p_M(x)}\,dx\right)+\indent\,\,$$
$$-C_1r\left((r-1)\displaystyle\int_{\Gamma}|w(t)|^{r-2}\displaystyle\sum_{j=1}^{N-1}|\pxti w(t)|^{q_j(x)}\,d\sigma+
\displaystyle\int_{\Gamma}\beta|w(t)|^{r-2}|w(t)|^{q_M(x)}\,d\sigma\right),$$
for some constant $C_1>0$. Therefore, (\ref{4.2.01}) is valid, as asserted.
\end{proof}

\indent For the rest of the subsection, we will assume that $\mathbf{u}(t):=T_{\sigma}(t)\mathbf{u}_0$ and
$\mathbf{v}(t):=T_{\sigma}(t)\mathbf{v}_0$ are such that $\mathbf{u}(t)\neq\mathbf{v}(t)$ \,a.e. over $\overline{\Omega}$.\\

\begin{lemma}\label{Lem02} Assuming the notations as in Lemma \ref{Lem01}, suppose that $(p_m,q_m)\in C^{0,1}(\overline{\Omega})\times C^{0,1}(\Gamma)$,
and let $r:[0,\infty)\rightarrow[2,\infty)$ be an increasing
differentiable function. Then for almost every $t\geq0$, one has\\[2ex]
$\displaystyle\frac{d}{dt}\log\left(|\|\mathbf{u}(t)-\mathbf{v}(t)\||_{_{r(t)}}\right)
\,\leq\,\frac{r'(t)}{r(t)}\Psi_{\ast}\left(\mathbf{z}(t)^{r(t)}\mathbf{\textbf{Log}}(\mathbf{z}(t))\right)+$\\
$$+\displaystyle\frac{\mathfrak{C}^{\ast}_p(r(t)-1)\|u(t)-v(t)\|^{\theta_p(x,t)}_{_{r(t)+p_m(\cdot)-2,\Omega}}}
{[\mathfrak{p}^-_r(t)]^{p^+_M}|\|\mathbf{u}(t)-\mathbf{v}(t)\||^{r(t)}_{_{r(t)}}}\left(-\frac{1}{\epsilon_1C'_{\epsilon}}\psi_{_{\Omega}}\left(g_p(t)^{p_m(x)}
\log(g_p(t))\right)-\frac{\log(\epsilon_1)}{\epsilon_1C'_{\epsilon}}+\epsilon+\mathfrak{G}_p\right)+\indent\indent\indent$$
\begin{equation}\label{4.2.02}
+\displaystyle\frac{\mathfrak{C}^{\ast}_q(r(t)-1)\|u(t)-v(t)\|^{\theta_q(x,t)}_{_{r(t)+q_m(\cdot)-2,\Gamma}}}
{[\mathfrak{q}^-_r(t)]^{q^+_M}|\|\mathbf{u}(t)-\mathbf{v}(t)\||^{r(t)}_{_{r(t)}}}\left(-\frac{1}{\epsilon_2C''_{\epsilon'}}\psi_{_{\Gamma}}\left(g_q(t)^{q_m(x)}
\log(g_q(t))\right)-\frac{\log(\epsilon_2)}{\epsilon_2C''_{\epsilon'}}+\epsilon'+\mathfrak{G}_q\right),
\end{equation}
for every $\epsilon,\,\epsilon',\,\epsilon_1,\,\epsilon_2>0$, where $\mathfrak{C}^{\ast}_p,\,\mathfrak{C}^{\ast}_q,\,\mathfrak{G}_p,\,\mathfrak{G}_q$
denote the constant defined in (\ref{4.2.09c} and \ref{4.2.09d}, respectively, the constants
$C'_{\epsilon},\,C''_{\epsilon'}$ are the ones appearing in (\ref{2.2.25}) and (\ref{2.2.25b}), respectively,
$\theta_p(x,t),\,\theta_q(x,t)$ are given by (\ref{4.2.09b}) and (\ref{theta-q}), respectively,
$\Psi_{\ast}(\varphi_1,\varphi_2):=\psi_{_{\Omega}}(\varphi_1)+\psi_{_{\Gamma}}(\varphi_2)$
(for $\psi_{_{\Omega}}(\cdot),\,\psi_{_{\Gamma}}(\cdot)$ given by (\ref{2.2.24}) and (\ref{2.2.24b}), respectively),
$\mathbf{z}(t):=\frac{|\mathbf{u}(t)-\mathbf{v}(t)|}{|\|\mathbf{u}(t)-\mathbf{v}(t)\||_{_{r(t)}}}$,
$\mathbf{\textbf{Log}}(\mathbf{z}(t)):=(\log(z(t)),\log(z(t))|_{_{\Gamma}})$, and for
$\overset{\rightarrow} s(\cdot):=(s_1(\cdot),\ldots,s_k(\cdot))$ (where $k=N$ if $s$ is defined on $\Omega$, and $k=N-1$ if $s$ is restricted to $\Gamma$),
$\mathfrak{s}_r(x,t):=\frac{r(t)+s_m(x)-2}{s_m(x)}$, and $g_s(t):=\frac{|u(t)-v(t)|^{^{\mathfrak{s}_r(x,t)}}}
{\|u(t)-v(t)\|^{^{\mathfrak{s}_r(x,t)}}_{_{r(t)+s_m(\cdot)-2,D}}}$ over $\overline{\Omega}\times[0,\infty)$ ($D$ denoting either $\Omega$ or $\Gamma$).
\end{lemma}

\begin{proof} By virtue of (\ref{4.2.01}) and the fact that $\frac{dF}{dt}=r'(t)\frac{\partial F}{\partial r}+\frac{\partial F}{\partial t}$
for a differentiable function $F=F(t,r(t))$, writing $\mathbf{w}(t):=\mathbf{u}(t)-\mathbf{v}(t)$, we infer that\\[2ex]
$\displaystyle\frac{d}{dt}|\|\mathbf{w}(t)\||^{r(t)}_{_{r(t)}}\,\leq\,r'(t)\Psi_{\ast}\left(\mathbf{w}(t)^{r(t)}
\mathbf{\textbf{Log}}(\mathbf{w}(t))\right)+$\\
$$-C_1r(t)\left((r(t)-1)\displaystyle\int_{\Omega}|w(t)|^{r(t)-2}\displaystyle\sum_{i=1}^N|\pxi w(t)|^{p_i(x)}\,dx+
\displaystyle\int_{\Omega}\alpha|w(t)|^{r(t)+p_M(x)-2}\,dx\right)+\indent\,\,$$
$$-C_1r(t)\left((r(t)-1)\displaystyle\int_{\Gamma}|w(t)|^{r(t)-2}\displaystyle\sum_{j=1}^{N-1}|\pxti w(t)|^{q_j(x)}\,d\sigma+
\displaystyle\int_{\Gamma}\beta|w(t)|^{r(t)+q_M(x)-2}\,d\sigma\right).$$
As
$$\displaystyle\frac{d}{dt}\log\left(|\|\mathbf{w}(t)\||_{_{r(t)}}\right)=
-\frac{r'(t)}{r(t)}\log\left(|\|\mathbf{w}(t)\||_{_{r(t)}}\right)+\frac{1}{r(t)|\|\mathbf{w}(t)\||^{r(t)}_{_{r(t)}}}
\displaystyle\frac{d}{dt}|\|\mathbf{w}(t)\||^{r(t)}_{_{r(t)}},$$
using the fact that $\Psi_{\ast}\left(\mathbf{z}(t)^{r(t)}\right)=1$, it follows that\\[2ex]
\noindent$\displaystyle\frac{d}{dt}\log\left(|\|\mathbf{w}(t)\||_{_{r(t)}}\right)\,
\leq\,-\frac{r'(t)}{r(t)}\log\left(|\|\mathbf{w}(t)\||_{_{r(t)}}\right)+\displaystyle\frac{\partial}{\partial t}|\|\mathbf{w}(t)\||^{r(t)}_{_{r(t)}}+
\displaystyle\frac{r'(t)}{r(t)|\|\mathbf{w}(t)\||^{r(t)}_{_{r(t)}}}\Psi_{\ast}\left(\mathbf{w}(t)^{r(t)}\mathbf{\textbf{Log}}(\mathbf{w}(t))\right)$\\[2ex]
\indent$\leq\,\displaystyle\frac{r'(t)}{r(t)}\Psi_{\ast}\left(\mathbf{z}(t)^{r(t)}\mathbf{\textbf{Log}}(\mathbf{z}(t))\right)+$\\
$$-\frac{C_1}{|\|\mathbf{w}(t)\||^{r(t)}_{_{r(t)}}}\left((r(t)-1)\displaystyle\int_{\Omega}|w(t)|^{r(t)-2}
\displaystyle\sum_{i=1}^N|\pxi w(t)|^{p_i(x)}\,dx+
\displaystyle\int_{\Omega}\alpha|w(t)|^{r(t)+p_M(x)-2}\,dx\right)+\indent\,\,$$
$$-\frac{C_1}{|\|\mathbf{w}(t)\||^{r(t)}_{_{r(t)}}}\left((r(t)-1)\displaystyle\int_{\Gamma}|w(t)|^{r(t)-2}
\displaystyle\sum_{j=1}^{N-1}|\pxti w(t)|^{q_j(x)}\,d\sigma+
\displaystyle\int_{\Gamma}\beta|w(t)|^{r(t)+q_M(x)-2}\,d\sigma\right)$$\indent\\
\indent$\leq\,\displaystyle\frac{r'(t)}{r(t)}\Psi_{\ast}\left(\mathbf{z}(t)^{r(t)}\mathbf{\textbf{Log}}(\mathbf{z}(t))\right)+$\\
\begin{equation}\label{4.2.03}
-\frac{C_1(r(t)-1)}{|\|\mathbf{w}(t)\||^{r(t)}_{_{r(t)}}}\left(\displaystyle\int_{\Omega}|w(t)|^{r(t)-2}
\displaystyle\sum_{i=1}^N|\pxi w(t)|^{p_i(x)}\,dx+\displaystyle\int_{\Gamma}|w(t)|^{r(t)-2}
\displaystyle\sum_{j=1}^{N-1}|\pxti w(t)|^{q_j(x)}\,d\sigma\right),
\end{equation}\indent\\
where we recall that $\mathbf{z}(t):=\frac{|\mathbf{w}(t)|}{|\|\mathbf{w}(t)\|_{_{r(t)}}}$. To complete the proof, we examine the terms in (\ref{4.2.03})
containing derivatives. We begin by considering the boundary term. Indeed, given $\mathfrak{q}_r(x,t):=\frac{r(t)+q_m(x)-2}{q_m(x)}$, put
$$f_q(t):=\frac{|w(t)|}{\|w(t)\|_{_{q_m(\cdot)\mathfrak{q}_r(\cdot,t),\Gamma}}}
\,\indent\,\textrm{and}\,\indent\,g_q(t):=f_q(t)^{^{\mathfrak{q}_r(\cdot,t)}}.$$
Then $g_q(t)=f_q(t)^{^{\mathfrak{q}_r(\cdot,t)}}=\frac{|u(t)-v(t)|^{^{\mathfrak{q}_r(\cdot,t)}}}
{\|u(t)-v(t)\|^{^{\mathfrak{q}_r(\cdot,t)}}_{_{q_m(\cdot)\mathfrak{q}_r(\cdot,t),\Gamma}}}$ lies in
$W^{1,\qq}(\Gamma)^+$ and fulfills (\ref{2.2.24b}). Then it follows from (\ref{2.2.25b}) that there exists a constant $M_{q}>0$ such that\\
\begin{equation}\label{4.2.04b}
-\displaystyle\sum^{N-1}_{j=1}\|\pxti g_q(t)\|_{_{q_m(\cdot),\Gamma}}\,\leq\,M_{q}\left(-\frac{1}{\epsilon_2C''_{\epsilon'}}\psi_{_{\Gamma}}\left(g_q(t)^{q_m(x)}
\log(g_q(t))\right)-\frac{\log(\epsilon_2)}{\epsilon_2C''_{\epsilon'}}+\epsilon'\right).
\end{equation}\indent\\
for every $\epsilon',\,\epsilon_2>0$. In the same way, for $\mathfrak{p}_r(x,t):=\frac{r(t)+p_m(x)-2}{p_m(x)}$, set
$$f_p(t):=\frac{|w(t)|}{\|w(t)\|_{_{p_m(\cdot)\mathfrak{p}_r(\cdot,t),\Omega}}}
\,\indent\,\textrm{and}\,\indent\,g_p(t):=f_p(t)^{^{\mathfrak{p}_r(\cdot,t)}}.$$
Then $g_p(t)\in W^{1,\pp}(\Omega)^+$ and fulfills (\ref{2.2.24}), and thus from (\ref{2.2.25}), there exists a constant $M_{p}>0$ such that\\
\begin{equation}\label{4.2.04}
-\displaystyle\sum^{N}_{i=1}\|\pxi g_p(t)\|_{_{p_m(\cdot),\Omega}}\,\leq\,M_{p}\left(-\frac{1}{\epsilon_1C'_{\epsilon}}\psi_{_{\Omega}}\left(g_p(t)^{p_m(x)}
\log(g_p(t))\right)-\frac{\log(\epsilon_1)}{\epsilon_1C'_{\epsilon}}+\epsilon\right).
\end{equation}\indent\\
for every $\epsilon,\,\epsilon_1>0$. For the remaining of the calculations, we will spell the details of the boundary terms; the interior estimates
run in the same way. In fact, since $1=\Theta_{_{q_m(\cdot),\Gamma}}(g_q(t))=\|g_q(t)\|_{_{q_m(\cdot),\Gamma}}$, one has the following calculation:\\[2ex]
\indent$\displaystyle\sum^{N-1}_{j=1}\Theta_{_{q_m(\cdot),\Gamma}}(\pxti g_q(t))+(N-1)\,\geq\,\displaystyle\sum^{N-1}_{j=1}\left(
\|\pxti g_q(t)\|^{\bar{q}_j}_{_{q_m(\cdot),\Gamma}}+1\right)\,\geq\,\displaystyle\sum^{N-1}_{j=1}\frac{1}{2^{\bar{q}_j}}\left(
\|\pxti g_q(t)\|_{_{q_m(\cdot),\Gamma}}+1\right)^{\bar{q}_j}$\\
\begin{equation}\label{4.2.06}
\indent\indent\indent\indent\indent\indent\indent\indent\indent\indent\indent\indent\geq\,\frac{1}{2^{\bar{q}_M}}
\left(\displaystyle\sum^{N-1}_{j=1}\|\pxti g_q(t)\|_{_{q_m(\cdot),\Gamma}}+(N-1)\right),
\end{equation}
where for each $j\in\{1,\ldots,N-1\}$, $$\bar{q}_j:=\left\{
\begin{array}{lcl}
q^-_j,\,\,\,\,\,\,\,\,\,\textrm{if}\,\,\,\,\|\pxti g_q(t)\|_{_{q_m(\cdot),\Gamma}}>1,\\
q^{+}_j,\,\,\,\,\,\,\,\,\,\textrm{if}\,\,\,\,\|\pxti g_q(t)\|_{_{q_m(\cdot),\Gamma}}\leq1,\\
\end{array}
\right.$$
Using (\ref{4.2.06}) together with Young's inequality (several times), we derive the following estimate:\\[2ex]
$\displaystyle\sum^{N-1}_{j=1}\|\pxti g_q(t)\|_{_{q_m(\cdot),\Gamma}}\,\leq\,2^{\bar{q}_M}\displaystyle\sum^{N-1}_{j=1}
\Theta_{_{q_m(\cdot),\Gamma}}(\pxti g_q(t))+2^{\bar{q}_M}(N-1)\Theta_{_{q_m(\cdot),\Gamma}}(g_q(t))$\\
$$=2^{\bar{q}_M}\displaystyle\sum^{N-1}_{j=1}\displaystyle\int_{\Gamma}\left|\mathfrak{q}_r(x,t)f_q(t)^{^{\mathfrak{q}_r(x,t)-1}}\frac{|\pxti w(t)|}
{\|w(t)\|_{_{r(t)+q_m(\cdot)-2,\Gamma}}}+\pxti\mathfrak{q}_r(x,t)f_q(t)^{^{\mathfrak{q}_r(x,t)}}
\log(f_q(t))\right|^{q_m(x)}d\sigma+2^{\bar{q}_M}(N-1)\Theta_{_{q_m(\cdot),\Gamma}}(g_q(t))$$
$$\leq\,2^{\bar{q}_M}\displaystyle\sum^{N-1}_{j=1}\displaystyle\int_{\Gamma}\left|\mathfrak{q}_r(x,t)|f_q(t)|^{^{\frac{r(t)-2}{q_j(x)}}}
|f_q(t)|^{^{(r(t)-2)\frac{q_j(x)-q_m(x)}{q_m(x)q_j(x)}}}\frac{|\pxti w(t)|}
{\|w(t)\|_{_{r(t)+q_m(\cdot)-2,\Gamma}}}+|\pxti\mathfrak{q}_r(x,t)f_q(t)^{^{\mathfrak{q}_r(x,t)}}\log(f_q(t))|\right|^{q_m(x)}d\sigma+$$
$$\indent\indent\indent\indent\indent\indent\indent\indent\indent\indent\indent\indent
\indent\indent\indent\indent\indent\indent\indent\indent\indent\indent\indent\indent+2^{\bar{q}_M}(N-1)\Theta_{_{q_m(\cdot),\Gamma}}(g_q(t))$$
$$\leq\,2^{\bar{q}_M+q^+_m}\displaystyle\sum^{N-1}_{j=1}\displaystyle\int_{\Gamma}\left||f_q(t)|^{^{\frac{r(t)-2}{q_m(x)}}}+\mathfrak{q}^{^{\frac{q_j(x)}
{q_m(x)}}}_r(x,t)|f_q(t)|^{^{\frac{r(t)-2}{q_m(x)}}}\left(\frac{|\pxti w(t)|}
{\|w(t)\|_{_{r(t)+q_m(\cdot)-2,\Gamma}}}\right)^{^{\frac{q_j(x)}{q_m(x)}}}\right|^{q_m(x)}d\sigma+$$
$$\indent\indent\indent\indent\indent\indent\indent\indent+2^{\bar{q}_M+q^+_m}\displaystyle\sum^{N-1}_{j=1}\displaystyle
\int_{\Gamma}\left||\pxti\mathfrak{q}_r(x,t)f_q(t)^{^{\mathfrak{q}_r(x,t)}}\log(f_q(t))|\right|^{q_m(x)}d\sigma
+2^{\bar{q}_M}(N-1)\Theta_{_{q_m(\cdot),\Gamma}}(g_q(t))$$
$$\leq\,2^{\bar{q}_M+2q^+_m}\left(\frac{[\mathfrak{q}^-_r(t)]^{q^+_M}|\|\mathbf{w}(t)\||^{r(t)}_{_{r(t)}}}{
\|w(t)\|^{\theta_q(x,t)}_{_{r(t)+q_m(\cdot)-2,\Gamma}}}\displaystyle\sum^{N-1}_{j=1}\displaystyle\int_{\Gamma}\frac{|w(t)|^{r(t)-2}}
{|\|\mathbf{w}(t)\||^{r(t)}_{_{r(t)}}}|\pxti w(t)|^{q_j(x)}\,d\sigma+\displaystyle\int_{\Gamma}
|f_q(t)|^{^{r(t)-2}}\,d\sigma\right)+$$
$$\indent\indent\indent\indent\indent\indent\indent\indent\indent\indent\indent\indent
+L_q\displaystyle\int_{\Gamma}|f_q(t)|^{^{r(t)+q_m(x)-2}}|\log(f_q(t))|^{q_m(x)}\,d\sigma+2^{\bar{q}_M}(N-1)\Theta_{_{q_m(\cdot),\Gamma}}(g_q(t))$$
$$\leq\,2^{\bar{q}_M+2q^+_m}\left(\frac{[\mathfrak{q}^-_r(t)]^{q^+_M}|\|\mathbf{w}(t)\||^{r(t)}_{_{r(t)}}}{
\|w(t)\|^{\theta_q(x,t)}_{_{r(t)+q_m(\cdot)-2,\Gamma}}}\displaystyle\sum^{N-1}_{j=1}\displaystyle\int_{\Gamma}\frac{|w(t)|^{r(t)-2}}
{|\|\mathbf{w}(t)\||^{r(t)}_{_{r(t)}}}|\pxti w(t)|^{q_j(x)}\,d\sigma+L_q(N-1)\displaystyle\int_{\Gamma}
|f_q(t)|^{^{r(t)+q_m(x)-2}}|\log(f_q(t))|^{q_m(x)}\,d\sigma\right)+$$
$$\indent\indent\indent\indent\indent\indent\indent\indent\indent\indent\indent\indent
\indent\indent\indent\indent\indent\indent\indent\indent\indent\indent+2^{\bar{q}_M+2q^+_m+2}(N-1)\Theta_{_{q_m(\cdot),\Gamma}}(g_q(t)),$$
where $L_q>0$ denotes the Lipschitz constant of $q_m(\cdot)$, \,$\mathfrak{q}^-_r(t):=(r(t)+q^-_m-2)/q^-_m$, and
\begin{equation}\label{theta-q}
\theta_q(x,t):=\left\{
\begin{array}{lcl}
r(t)+q^-_{m}-2,\,\,\,\,\,\,\,\,\,\,\,\,\textrm{if}\,\,\,\,\|w(t)\|_{_{r(t)+q_m(\cdot)-2,\Gamma}}>1,\\[1ex]
r(t)+q^+_{m}-2,\,\,\,\,\,\,\,\,\,\,\,\,\textrm{if}\,\,\,\,\|w(t)\|_{_{r(t)+q_m(\cdot)-2,\Gamma}}\leq1.\\
\end{array}
\right.
\end{equation}
Now we consider the term involving a logarithm in the above calculation. First, recall that for each $\delta_1,\,\delta_2\in(0,1)$, one has
the standard logarithmic properties:
$$\displaystyle\lim_{t\rightarrow0^+}t^{\delta_1}|\log(t)|=\displaystyle\lim_{t\rightarrow\infty}t^{-\delta_2}|\log(t)|=0.$$
Then, as $f_q(t)\geq0$, letting $\Gamma_1:=\Omega\cap\{f_q(t)\leq1\}$ and $\Gamma_2:=\Gamma\setminus\Gamma_1$, we put\\
$$J_m(f_q(t)):=\displaystyle\int_{\Gamma_m}\left(f_q(t)^{^{\mathfrak{q}_r(\cdot,t)}}|\log(f_q(t))|\right)^{q_m(x)}\,d\sigma,
\indent\indent\textrm{for}\,\,\,m\in\{1,2\}.$$
Then by the above logarithmic properties together with H\"older's inequality, we have that
\begin{equation}\label{4.2.07}
J_1(f_q(t))\,\leq\,c(\delta_1)\displaystyle\int_{\Gamma}g_q(t)^{^{q_m(x)(1-\delta_1)}}\,d\sigma\,\leq\,C_{_{\Gamma}}\|g_q(t)\|_{_{q_m(\cdot),\Gamma}}
=C_{_{\Gamma}}\Theta_{_{q_m(\cdot),\Gamma}}(g_q(t)),
\end{equation}
for some constants $c(\delta_1),\,C_{_{\Gamma}}>0$. On the other hand, since
$\Theta_{_{q_m(\cdot),\Gamma_2}}(g_q(t))<1$, we can find a constant $\eta>0$ large enough, and a sufficiently
small constant $\delta_2\in(0,1)$, such that $\left\|g_q(t)/\eta\right\|_{_{q_m(\cdot)(1+\delta_2),\Gamma_2}}\leq1$. Then we have\\[2ex]
$J_2(f_q(t))\,\leq\,c(\delta_2)\displaystyle\int_{\Gamma_2}g_q(t)^{^{q_m(x)(1+\delta_2)}}\,d\sigma$
\begin{equation}\label{4.2.08}
\indent\leq\,c(\delta_2)\eta^{^{q^{+}_m(1+\delta_2)}}\displaystyle\int_{\Gamma_2}\left(\frac{g_q(t)}{\eta}\right)^{^{q_m(x)(1+\delta_2)}}\,d\sigma
\,\leq\,C'_{_{\Gamma}}\eta^{^{q^{+}_m(1+\delta_2)-1}}\|g_q(t)\|_{_{q_m(\cdot)(1+\delta_2),\Gamma}}
\,\leq\,C''_{_{\Gamma}}\eta^{^{q^{+}_m(1+\delta_2)-1}}\|g_q(t)\|_{_{\tilde{q}(\cdot),\Gamma}},
\end{equation}
for some constants $c(\delta_2),\,C'_{_{\Gamma}},\,C''_{_{\Gamma}}>0$, where $\tilde{q}\in\mathcal{P}(\Gamma)$ is such that
$q_m(x)(1+\delta_2)\leq\tilde{q}(x)$ for a.e. $x\in\Gamma$ and $\textrm{ess}\displaystyle\inf_{x\in\Gamma}\left\{q^\partial(x)-\tilde{q}(x)\right\}>0$.
By Theorem \ref{turmanoua2cont}, the embedding $W^{1,q_m(\cdot)}(\Gamma) \hookrightarrow L^{\tilde{q}(\cdot)}({\Gamma})$ is compact,
and consequently, an application of \cite[Lemma 2.4.7]{NITTKA2010} implies that $$\|g_q(t)\|_{_{\tilde{q}(\cdot),\Gamma}}
\,\leq\,\epsilon_0\displaystyle\sum^{N-1}_{j=1}\|\pxti g_q(t)\|_{_{q_m(\cdot),\Gamma}}+C_{\epsilon_0}\|g_q(t)\|_{_{q_m(\cdot),\Gamma}},$$
for all $\epsilon_0>0$, and for some constant $C_{\epsilon_0}>0$. Using this, and selecting $\epsilon_0>0$ suitably, (\ref{4.2.08}) becomes
\begin{equation}\label{4.2.08a}
J_2(f_q(t))\,\leq\,M_{_{\Gamma}}(\eta,\epsilon'_0)\Theta_{_{q_m(\cdot),\Gamma}}(g_q(t))+\epsilon'_0\displaystyle\sum^{N-1}_{j=1}\|\pxti g_q(t)\|_{_{q_m(\cdot),\Gamma}},
\end{equation}
for each $\epsilon'_0>0$, for some constant $M_{_{\Gamma}}(\eta,\epsilon'_0)>0$, where we have been using the fact that
$\|g_q(t)\|_{_{q_m(\cdot),\Gamma}}=\Theta_{_{q_m(\cdot),\Gamma}}(g_q(t))=1$.
Choosing $\epsilon_1>0$ small enough,
we apply the inequalities (\ref{4.2.07}) and (\ref{4.2.08a}) into the previous long estimate to deduce that\\[2ex]
$\displaystyle\sum^{N-1}_{j=1}\|\pxti g_q(t)\|_{_{q_m(\cdot),\Gamma}}\,\leq\,C_{\epsilon'_0}\left(\frac{2^{\bar{q}_M+2q^+_m}
[\mathfrak{q}^-_r(t)]^{q^+_M}|\|\mathbf{w}(t)\||^{r(t)}_{_{r(t)}}}{
\|w(t)\|^{\theta_q(x,t)}_{_{r(t)+q_m(\cdot)-2,\Gamma}}}\displaystyle\sum^{N-1}_{j=1}\displaystyle\int_{\Gamma}\frac{|w(t)|^{r(t)-2}}
{|\|\mathbf{w}(t)\||^{r(t)}_{_{r(t)}}}|\pxti w(t)|^{q_j(x)}\,d\sigma\right)+$\\
\begin{equation}\label{4.2.09}
\indent\indent\indent\indent\indent\indent\indent\indent\indent\indent\indent\indent
\indent\indent\indent\indent+2^{\bar{q}_M+2q^+_m+2}(N-1)C_{\epsilon'_0}\left\{1+L_q[1+M_{_{\Gamma}}(\eta,\epsilon_1)]\right\},
\end{equation}
for some constant $C_{\epsilon'_0}>0$. Then, proceeding in the same way,
one can obtain the existence of positive constants $C^{\ast}$, $\eta'$, $M_{_{\Omega}}(\eta')$, $L_p$, fulfilling the inequality\\[2ex]
$\displaystyle\sum^{N}_{i=1}\|\pxi g_p(t)\|_{_{p_m(\cdot),\Omega}}\,\leq\,C^{\ast}[\mathfrak{p}^-_r(t)]^{p^+_M}\left(\frac{2^{\bar{p}_M+2p^+_m}
|\|\mathbf{w}(t)\||^{r(t)}_{_{r(t)}}}{\|w(t)\|^{\theta_p(x,t)}_{_{r(t)+p_m(\cdot)-2,\Omega}}}
\displaystyle\sum^{N}_{i=1}\displaystyle\int_{\Omega}\frac{|w(t)|^{r(t)-2}}{|\|\mathbf{w}(t)\||^{r(t)}_{_{r(t)}}}|\pxi w(t)|^{p_i(x)}\,dx\right)+$\\
\begin{equation}\label{4.2.09a}
\indent\indent\indent\indent\indent\indent\indent\indent\indent\indent\indent\indent
\indent\indent\indent\indent+2^{\bar{p}_M+2p^+_m+2}N\left\{1+L_p[1+M_{_{\Omega}}(\eta')]\right\},
\end{equation}
where
\begin{equation}\label{4.2.09b}
\bar{p}_i:=\left\{
\begin{array}{lcl}
p^-_i,\,\,\,\,\,\,\,\,\,\textrm{if}\,\,\,\,\|\pxi g_p(t)\|_{_{p_m(\cdot),\Omega}}>1,\\
p^{+}_i,\,\,\,\,\,\,\,\,\,\textrm{if}\,\,\,\,\|\pxi g_p(t)\|_{_{p_m(\cdot),\Omega}}\leq1,\\
\end{array}
\right.\,\,\,\,\,\textrm{and}\,\,\,\,\,
\theta_p(x,t):=\left\{
\begin{array}{lcl}
r(t)+p^-_{m}-2,\,\,\,\,\,\,\,\,\,\,\,\,\textrm{if}\,\,\,\,\|w(t)\|_{_{r(t)+p_m(\cdot)-2,\Omega}}>1,\\[1ex]
r(t)+p^+_{m}-2,\,\,\,\,\,\,\,\,\,\,\,\,\textrm{if}\,\,\,\,\|w(t)\|_{_{r(t)+p_m(\cdot)-2,\Omega}}\leq1,\\
\end{array}
\right.
\end{equation}
($1\leq i\leq N$). Setting
\begin{equation}\label{4.2.09c}
\mathfrak{C}^{\ast}_p:=\frac{C_1M_p}{2^{^{\bar{p}_M+2p^+_m}}C^{\ast}},\,\,\,\,\,\,\,\,\,\,\,\,\,\,\,\,\,\,\mathfrak{G}_{p}:=
\displaystyle\frac{2^{\bar{p}_M+2p^+_m+2}N\left\{1+L_p[1+M_{_{\Omega}}(\eta')]\right\}}{C_1M_p},
\end{equation}
and
\begin{equation}\label{4.2.09d}
\mathfrak{C}^{\ast}_q:=\frac{C_1M_q}{2^{^{\bar{q}_M+2q^+_m}}C_{\epsilon'_0}},\,\,\,\,\,\,\,\,\,\,\,\,\,\,\,\,\mathfrak{G}_{q}:=
\displaystyle\frac{2^{\bar{q}_M+2q^+_m+2}(N-1)\left\{1+L_q[1+M_{_{\Gamma}}(\eta,\epsilon_1)]\right\}}{C_1M_q},
\end{equation}\indent\\
(for $C_1>0$ the constant in Lemma \ref{Lem01}), we combine (\ref{4.2.03}), (\ref{4.2.04b}), (\ref{4.2.04}), (\ref{4.2.09}), and (\ref{4.2.09a}),
and we arrive at the desired inequality (\ref{4.2.02}), completing the proof.
\end{proof}

\begin{lemma}\label{Lem03}
Assume the same conditions, notations, and assumptions of Lemma \ref{Lem02}, and assume in addition that $q_m\in C^{0,1}(\overline{\Omega})$
with $2<q^-_m\leq q^+_M<\infty$. Then for almost all $t\geq0$ we have
\begin{equation}\label{4.2.10}
\displaystyle\frac{d}{dt}\log\left(|\|\mathbf{u}(t)-\mathbf{v}(t)\||_{_{r(t)}}\right)\,\leq\,-\mathfrak{P}_{p,q}(t)
\log\left(|\|\mathbf{u}(t)-\mathbf{v}(t)\||_{_{r(t)}}\right)-\mathfrak{Q}_{p,q}(t),
\end{equation}
where
\begin{equation}\label{P}
\mathfrak{P}_{p,q}(t):=\frac{r'(t)}{r(t)}\left(\frac{(\check{p}_m-2)\tilde{p}_m}{r(t)+\tilde{p}-2}+\frac{(\check{q}_m-2)\tilde{q}_m}{r(t)+\tilde{q}-2}
\right)
\end{equation}
and\\[2ex]
\indent$\mathfrak{Q}_{p,q}(t):=\displaystyle\frac{r'(t)\tilde{p}_m}{r(t)(r(t)+\tilde{p}_m-2)}
\log\left(\displaystyle\frac{\mathfrak{C}^{\ast}_pr(t)(r(t)-1)\tilde{p}^{^{p^+_M-1}}_m}{r'(t)
C'_{\epsilon}(r(t)+\tilde{p}_m-2)^{^{p^+_M-1}}}\right)-\frac{2\mathfrak{C}^{\ast}_p(r(t)-1)\kappa_p\tilde{p}^{^{p^+_M}}_m}{(r(t)+\tilde{p}_m-2)^{^{p^+_M}}}+$\\
\begin{equation}\label{Q}
\indent\indent\indent\indent+\frac{r'(t)\tilde{q}_m}{r(t)(r(t)+\tilde{q}_m-2)}\log\left(\displaystyle
\frac{\mathfrak{C}^{\ast}_qr(t)(r(t)-1)\tilde{q}^{^{q^+_M-1}}_m}{r'(t)
C''_{\epsilon'}(r(t)+\tilde{q}_m-2)^{^{q^+_M-1}}}\right)-\frac{2\mathfrak{C}^{\ast}_q(r(t)-1)\kappa_q\tilde{q}^{^{q^+_M}}_m}{(r(t)+\tilde{q}_m-2)^{^{q^+_M}}}
\end{equation}
(for the constants $\check{p}_m,\,\check{q}_m$ given by (\ref{check-p-q}), and $\kappa_p,\,\kappa_q$ describe in the proof of the lemma below).
\end{lemma}
\indent\\
\begin{proof} Adopting the assumptions and notations as in Lemma \ref{Lem02}, we define the functionals:
$$S_{_{\mu,D}}(\eta(\cdot),\phi(t)):=\Psi_{_{D}}\left(\left\{\frac{|\phi(t)|}{|\|\Phi(t)\||_{_{\eta(\cdot)}}}\right\}^{\eta(\cdot)}
\log\left(\frac{|\phi(t)|}{|\|\Phi(t)\||_{_{\eta(\cdot)}}}\right)\right),$$
for a measurable function $\eta(\cdot)$, where $\mu=\lambda_N$ (the $N$-dimensional Lebesgue measure) in the case when $D=\Omega$,
and $\mu=\sigma$ in the case when $D=\Gamma$. Then, for continuous functions $\eta_1(\cdot),\,\eta_2(\cdot)$ on $\overline{\Omega}$, we put
$$\mathbf{S}_{_{\overline{\Omega}}}((\Phi(t);(\eta_1(\cdot),\eta_2(\cdot)):=S_{_{\lambda_N,\Omega}}(\eta_1(\cdot),\phi(t))+S_{_{\sigma,\Gamma}}(\eta_2(\cdot),\phi(t)),$$
where we recall that $\Phi(t):=(\phi(t),\phi(t)|_{_{\Gamma}})$. Letting $\Phi(t):=\mathbf{w}(t)=\mathbf{u}(t)-\mathbf{v}(t)$, we let\\
$$\tilde{p}_m:=\left\{
\begin{array}{lcl}
p^-_m,\,\,\,\,\textrm{if}\,\,\,S_{_{\lambda_N,\Omega}}(r(t)+p_m(\cdot)-2,\phi(t))\leq0,\\
p^+_m,\,\,\,\,\textrm{if}\,\,\,S_{_{\lambda_N,\Omega}}(r(t)+p_m(\cdot)-2,\phi(t))>0,\\
\end{array}
\right.
\,\,\,\,\,\,\,\,\tilde{q}_m:=\left\{
\begin{array}{lcl}
q^-_m,\,\,\,\,\textrm{if}\,\,\,S_{_{\sigma,\Gamma}}(r(t)+q_m(\cdot)-2,\phi(t))\leq0,\\
q^+_m,\,\,\,\,\textrm{if}\,\,\,S_{_{\sigma,\Gamma}}(r(t)+q_m(\cdot)-2,\phi(t))>0,\\
\end{array}
\right.$$
and
$$\epsilon=\left\{\begin{array}{lcl}
\kappa_p\,\displaystyle\frac{|\|\mathbf{w}(t)\||^{r(t)}_{_{r(t)}}}{|\|\mathbf{w}(t)\||^{\theta_p(x,t)}_{_{r(t)+p_m(\cdot)-2}}}-\mathfrak{G}_{p},
\,\,\,\,\,\,\,\textrm{if}\,\,\,
\displaystyle\frac{|\|\mathbf{w}(t)\||^{r(t)}_{_{r(t)}}}{|\|\mathbf{w}(t)\||^{\theta_p(x,t)}_{_{r(t)+p_m(\cdot)-2}}}>\mathfrak{G}_{p},\\[2ex]
\kappa_p\,\displaystyle\frac{|\|\mathbf{w}(t)\||^{r(t)}_{_{r(t)}}}{|\|\mathbf{w}(t)\||^{\theta_p(x,t)}_{_{r(t)+p_m(\cdot)-2}}}
\,\,\,\,\,\,\,\,\,\,\,\,\,\,\,\,\,\,\,\,\,\,\,\,\,\,\,\,\,\textrm{if}\,\,\,
\displaystyle\frac{|\|\mathbf{w}(t)\||^{r(t)}_{_{r(t)}}}{|\|\mathbf{w}(t)\||^{\theta_p(x,t)}_{_{r(t)+p_m(\cdot)-2}}}\leq\mathfrak{G}_{p},
\end{array}
\right.$$
and
$$\epsilon'=\left\{\begin{array}{lcl}
\kappa_q\,\displaystyle\frac{|\|\mathbf{w}(t)\||^{r(t)}_{_{r(t)}}}{|\|\mathbf{w}(t)\||^{\theta_q(x,t)}_{_{r(t)+q_m(\cdot)-2}}}-\mathfrak{G}_{q},
\,\,\,\,\,\,\,\textrm{if}\,\,\,
\displaystyle\frac{|\|\mathbf{w}(t)\||^{r(t)}_{_{r(t)}}}{|\|\mathbf{w}(t)\||^{\theta_q(x,t)}_{_{r(t)+q_m(\cdot)-2}}}>\mathfrak{G}_{q},\\[2ex]
\kappa_q\,\displaystyle\frac{|\|\mathbf{w}(t)\||^{r(t)}_{_{r(t)}}}{|\|\mathbf{w}(t)\||^{\theta_q(x,t)}_{_{r(t)+q_m(\cdot)-2}}}
\,\,\,\,\,\,\,\,\,\,\,\,\,\,\,\,\,\,\,\,\,\,\,\,\,\,\,\,\,\textrm{if}\,\,\,
\displaystyle\frac{|\|\mathbf{w}(t)\||^{r(t)}_{_{r(t)}}}{|\|\mathbf{w}(t)\||^{\theta_q(x,t)}_{_{r(t)+q_m(\cdot)-2}}}\leq\mathfrak{G}_{q},
\end{array}
\right.$$
and
$$\epsilon_1=\displaystyle\frac{\mathfrak{C}^{\ast}_pr(t)(r(t)+\tilde{p}_m-2)(r(t)-1)|\|\mathbf{w}(t)\||^{\theta_p(x,t)}_{_{r(t)+p_m(\cdot)-2}}}
{[\mathfrak{p}^-_r(t)]^{p^+_M}C'_{\epsilon}\tilde{p}_mr'(t)|\|\mathbf{w}(t)\||^{r(t)}_{_{r(t)}}},\,\,\,\,\,\,\,\,\,\,\,\,
\epsilon_2=\displaystyle\frac{\mathfrak{C}^{\ast}_qr(t)(r(t)+\tilde{q}_m-2)(r(t)-1)|\|\mathbf{w}(t)\||^{\theta_q(x,t)}_{_{r(t)+q_m(\cdot)-2}}}
{[\mathfrak{q}^-_r(t)]^{q^+_M}C''_{\epsilon'}\tilde{q}_mr'(t)|\|\mathbf{w}(t)\||^{r(t)}_{_{r(t)}}},$$\indent\\
where $\kappa_p,\,\kappa_q
\in[1,\infty)$ are constants chosen large enough (when necessary), such that
\begin{equation}\label{kp1}
-\psi_{_{\Omega}}\left(g_p(t)^{p_m(x)}\log(g_p(t))\right)-\log(\epsilon_1)+\displaystyle\frac{\mathfrak{C}^{\ast}_p(r(t)-1)\|w(t)\|
^{\theta_p(x,t)}_{_{r(t)+p_m(\cdot)-2,\Omega}}}{[\mathfrak{p}^-_r(t)]^{p^+_M}|\|\mathbf{w}(t)\||^{r(t)}_{_{r(t)}}}\epsilon\,\geq\,0
\end{equation}
and
\begin{equation}\label{kp2}
-\psi_{_{\Gamma}}\left(g_q(t)^{q_m(x)}\log(g_q(t))\right)-\log(\epsilon_2)+\displaystyle\frac{\mathfrak{C}^{\ast}_q(r(t)-1)
\|w(t)\|^{\theta_q(x,t)}_{_{r(t)+q_m(\cdot)-2,\Gamma}}}{[\mathfrak{q}^-_r(t)]^{q^+_M}|\|\mathbf{w}(t)\||^{r(t)}_{_{r(t)}}}\epsilon'\,\geq\,0.
\end{equation}
From (\ref{4.2.02}), recalling the definition of the coefficients $\mathfrak{p}^-_r(t)$ and
$\mathfrak{q}^-_r(t)$ (given in the proof of the previous lemma), the monotonicity of the logarithmic function,
and taking into account (\ref{kp1}) and (\ref{kp2}), we deduce that\\[2ex]
$\displaystyle\frac{d}{dt}\log\left(|\|\mathbf{w}(t)\||_{_{r(t)}}\right)\,\leq\,\displaystyle\frac{r'(t)}{r(t)}
\left[\mathbf{S}_{_{\overline{\Omega}}}(\mathbf{w}(t);r(t),r(t))-\mathbf{S}_{_{\overline{\Omega}}}(\mathbf{w}(t);r(t)+p_m(\cdot)-2,r(t)+q_m(\cdot)-2)\right]+$\\
$$-\frac{r'(t)\tilde{p}_m}{r(t)(r(t)+\tilde{p}_m-2)}\left[\log\left(\displaystyle\frac{\mathfrak{C}^{\ast}_pr(t)(r(t)+\tilde{p}_m-2)(r(t)-1)}{r'(t)
C'_{\epsilon}[\mathfrak{p}^-_r(t)]^{p^+_M}\tilde{p}_m}\right)+\log\left(\displaystyle\frac{|\|\mathbf{w}(t)\||^{\theta_p(x,t)}_{_{r(t)+p_m(\cdot)-2}}}
{|\|\mathbf{w}(t)\||^{r(t)}_{_{r(t)}}}\right)\right]+\frac{2\mathfrak{C}^{\ast}_p(r(t)-1)\kappa_p}{[\mathfrak{p}^-_r(t)]^{p^+_M}}+$$\\
$$-\frac{r'(t)\tilde{q}_m}{r(t)(r(t)+\tilde{q}_m-2)}\left[\log\left(\displaystyle\frac{\mathfrak{C}^{\ast}_qr(t)(r(t)+\tilde{q}_m-2)(r(t)-1)}{r'(t)
C''_{\epsilon'}[\mathfrak{q}^-_r(t)]^{q^+_M}\tilde{q}_m}\right)+\log\left(\displaystyle\frac{|\|\mathbf{w}(t)\||^{\theta_q(x,t)}_{_{r(t)+q_m(\cdot)-2}}}
{|\|\mathbf{w}(t)\||^{r(t)}_{_{r(t)}}}\right)\right]+\frac{2\mathfrak{C}^{\ast}_q(r(t)-1)\kappa_q}{[\mathfrak{q}^-_r(t)]^{q^+_M}}$$\\
$\leq\,\displaystyle\frac{r'(t)}{r(t)}
\left[\mathbf{S}_{_{\overline{\Omega}}}(\mathbf{w}(t);r(t),r(t))-\mathbf{S}_{_{\overline{\Omega}}}(\mathbf{w}(t);r(t)+p_m(\cdot)-2,r(t)+q_m(\cdot)-2)\right]+$\\
$$-\frac{r'(t)\tilde{p}_m}{r(t)(r(t)+\tilde{p}_m-2)}\left[\log\left(\displaystyle\frac{\mathfrak{C}^{\ast}_pr(t)(r(t)-1)\tilde{p}^{^{p^+_M-1}}_m}{r'(t)
C'_{\epsilon}(r(t)+\tilde{p}_m-2)^{^{p^+_M-1}}}\right)+\log\left(\displaystyle\frac{|\|\mathbf{w}(t)\||^{\theta_p(x,t)}_{_{r(t)+p_m(\cdot)-2}}}
{|\|\mathbf{w}(t)\||^{r(t)}_{_{r(t)}}}\right)\right]+\frac{2\mathfrak{C}^{\ast}_p(r(t)-1)\kappa_p\tilde{p}^{^{p^+_M}}_m}{(r(t)+\tilde{p}_m-2)^{^{p^+_M}}}+$$\\
$$-\frac{r'(t)\tilde{q}_m}{r(t)(r(t)+\tilde{q}_m-2)}\left[\log\left(\displaystyle\frac{\mathfrak{C}^{\ast}_qr(t)(r(t)-1)\tilde{q}^{^{q^+_M-1}}_m}{r'(t)
C''_{\epsilon'}(r(t)+\tilde{q}_m-2)^{^{q^+_M-1}}}\right)+\log\left(\displaystyle\frac{|\|\mathbf{w}(t)\||^{\theta_q(x,t)}_{_{r(t)+q_m(\cdot)-2}}}
{|\|\mathbf{w}(t)\||^{r(t)}_{_{r(t)}}}\right)\right]+\frac{2\mathfrak{C}^{\ast}_q(r(t)-1)\kappa_q\tilde{q}^{^{q^+_M}}_m}{(r(t)+\tilde{q}_m-2)^{^{q^+_M}}}.$$\\[1ex]
At this stage, for simplicity, we suppose that $m_{\sigma}(\overline{\Omega}):=|\Omega|+\sigma(\Gamma)=1$. We also recall that the mapping
$s\mapsto\log\left(|\|\Phi\||^s_{_{s}}\right)$ is convex for each $\Phi\in\mathbb{X\!}^{\,s}(\overline{\Omega})$ fixed. In addition,
the mapping $s\mapsto\frac{d}{ds}\log\left(|\|\Phi\||^s_{_{s}}\right)$ is non-decreasing, with
$$\frac{d}{ds}\log\left(|\|\Phi\||^s_{_{s}}\right)=S_{_{\overline{\Omega}}}(\Phi;s,s)+\log\left(|\|\Phi\||_{_{s}}\right)\,\,\,\,\,
\textrm{for a.e.}\,\,s.$$
Consequently, using the above facts together with the assumption $m_{\sigma}(\overline{\Omega})=1$,
H\"older's inequality, and the monotonicity of the logarithmic function, it follows that\\
\begin{equation}\label{4.2.11}
\mathbf{S}_{_{\overline{\Omega}}}(\mathbf{w}(t);r(t)+p_m(\cdot)-2,r(t)+q_m(\cdot)-2)-\mathbf{S}_{_{\overline{\Omega}}}(\mathbf{w}(t);r(t),r(t))\,\leq\,
\log\left(\frac{|\|\mathbf{w}(t)\||_{_{r(t)}}}{|\|\mathbf{w}(t)\||_{_{r(t)+p_m(\cdot)-2}}}\right)+
\log\left(\frac{|\|\mathbf{w}(t)\||_{_{r(t)}}}{|\|\mathbf{w}(t)\||_{_{r(t)+q_m(\cdot)-2}}}\right).
\end{equation}\indent\\
Applying (\ref{4.2.11}) to the previous calculation gives\\[2ex]
\indent$\displaystyle\frac{d}{dt}\log\left(|\|\mathbf{w}(t)\||_{_{r(t)}}\right)$\\
$$\leq\,\frac{r'(t)}{r(t)}\left[\log\left(|\|\mathbf{w}(t)\||_{_{r(t)+p_m(\cdot)-2}}\right)-\log\left(|\|\mathbf{w}(t)\||_{_{r(t)}}\right)+
\log\left(|\|\mathbf{w}(t)\||_{_{r(t)+q_m(\cdot)-2}}\right)-\log\left(|\|\mathbf{w}(t)\||_{_{r(t)}}\right)\right]+\indent\indent\indent\indent$$
$$+\frac{r'(t)\tilde{p}_m}{(r(t)+\tilde{p}_m-2)}\log\left(|\|\mathbf{w}(t)\||_{_{r(t)}}\right)
-\frac{r'(t)\tilde{p}_m\theta_p(x,t)}{r(t)(r(t)+\tilde{p}_m-2)}\log\left(|\|\mathbf{w}(t)\||_{_{r(t)+p_m(\cdot)-2}}\right)+$$
$$+\frac{r'(t)\tilde{q}_m}{(r(t)+\tilde{q}_m-2)}\log\left(|\|\mathbf{w}(t)\||_{_{r(t)}}\right)
-\frac{r'(t)\tilde{q}_m\theta_q(x,t)}{r(t)(r(t)+\tilde{q}_m-2)}\log\left(|\|\mathbf{w}(t)\||_{_{r(t)+q_m(\cdot)-2}}\right)+$$
$$\indent\indent+\frac{2\mathfrak{C}^{\ast}_p(r(t)-1)\kappa_p\tilde{p}^{^{p^+_M}}_m}{(r(t)+\tilde{p}_m-2)^{^{p^+_M}}}-
\frac{r'(t)\tilde{p}_m}{r(t)(r(t)+\tilde{p}_m-2)}\log\left(\displaystyle\frac{\mathfrak{C}^{\ast}_pr(t)(r(t)-1)\tilde{p}^{^{p^+_M-1}}_m}{r'(t)
C'_{\epsilon}(r(t)+\tilde{p}_m-2)^{^{p^+_M-1}}}\right)+$$
$$\indent\indent+\frac{2\mathfrak{C}^{\ast}_q(r(t)-1)\kappa_q\tilde{q}^{^{q^+_M}}_m}{(r(t)+\tilde{q}_m-2)^{^{q^+_M}}}-
\frac{r'(t)\tilde{q}_m}{r(t)(r(t)+\tilde{q}_m-2)}\log\left(\displaystyle\frac{\mathfrak{C}^{\ast}_qr(t)(r(t)-1)\tilde{q}^{^{q^+_M-1}}_m}{r'(t)
C''_{\epsilon'}(r(t)+\tilde{q}_m-2)^{^{q^+_M-1}}}\right)$$
$$=\frac{r'(t)}{r(t)}\left[\left(1-\frac{\tilde{p}_m\theta_p(x,t)}{r(t)+\tilde{p}_m-2}\right)\log\left(|\|\mathbf{w}(t)\||_{_{r(t)+p_m(\cdot)-2}}\right)
+\left(1-\frac{\tilde{q}_m\theta_q(x,t)}{r(t)+\tilde{q}_m-2}\right)\log\left(|\|\mathbf{w}(t)\||_{_{r(t)+q_m(\cdot)-2}}\right)\right]+$$
$$\indent\indent+\frac{r'(t)}{r(t)}\left(\frac{r(t)\tilde{p}_m}{r(t)+\tilde{p}_m-2}+\frac{r(t)\tilde{q}_m}{r(t)+\tilde{q}_m-2}-2\right)
\log\left(|\|\mathbf{w}(t)\||_{_{r(t)}}\right)+$$
$$\indent\indent+\frac{2\mathfrak{C}^{\ast}_p(r(t)-1)\kappa_p\tilde{p}^{^{p^+_M}}_m}{(r(t)+\tilde{p}_m-2)^{^{p^+_M}}}-
\frac{r'(t)\tilde{p}_m}{r(t)(r(t)+\tilde{p}_m-2)}\log\left(\displaystyle\frac{\mathfrak{C}^{\ast}_pr(t)(r(t)-1)\tilde{p}^{^{p^+_M-1}}_m}{r'(t)
C'_{\epsilon}(r(t)+\tilde{p}_m-2)^{^{p^+_M-1}}}\right)+$$
$$\indent\indent+\frac{2\mathfrak{C}^{\ast}_q(r(t)-1)\kappa_q\tilde{q}^{^{q^+_M}}_m}{(r(t)+\tilde{q}_m-2)^{^{q^+_M}}}-
\frac{r'(t)\tilde{q}_m}{r(t)(r(t)+\tilde{q}_m-2)}\log\left(\displaystyle\frac{\mathfrak{C}^{\ast}_qr(t)(r(t)-1)\tilde{q}^{^{q^+_M-1}}_m}{r'(t)
C''_{\epsilon'}(r(t)+\tilde{q}_m-2)^{^{q^+_M-1}}}\right)$$\\[2ex]
\,\,\,\,\,\,\,\,\,\,\,\,\,\,\,\,\,\,$\indent\indent\leq\,\displaystyle\frac{r'(t)}{r(t)}
\left(\displaystyle\frac{(2-\check{p}_m)\tilde{p}_m}{r(t)+\tilde{p}-2}+\frac{(2-\check{q}_m)\tilde{q}_m}{r(t)+\tilde{q}-2}
\right)\log\left(|\|\mathbf{w}(t)\||_{_{r(t)}}\right)+$\\
$$\indent\indent+\frac{2\mathfrak{C}^{\ast}_p(r(t)-1)\kappa_p\tilde{p}^{^{p^+_M}}_m}{(r(t)+\tilde{p}_m-2)^{^{p^+_M}}}-
\frac{r'(t)\tilde{p}_m}{r(t)(r(t)+\tilde{p}_m-2)}\log\left(\displaystyle\frac{\mathfrak{C}^{\ast}_pr(t)(r(t)-1)\tilde{p}^{^{p^+_M-1}}_m}{r'(t)
C'_{\epsilon}(r(t)+\tilde{p}_m-2)^{^{p^+_M-1}}}\right)+$$
$$\indent\indent+\frac{2\mathfrak{C}^{\ast}_q(r(t)-1)\kappa_q\tilde{q}^{^{q^+_M}}_m}{(r(t)+\tilde{q}_m-2)^{^{q^+_M}}}-
\frac{r'(t)\tilde{q}_m}{r(t)(r(t)+\tilde{q}_m-2)}\log\left(\displaystyle\frac{\mathfrak{C}^{\ast}_qr(t)(r(t)-1)\tilde{q}^{^{q^+_M-1}}_m}{r'(t)
C''_{\epsilon'}(r(t)+\tilde{q}_m-2)^{^{q^+_M-1}}}\right)$$
where
\begin{equation}\label{check-p-q}
\check{p}_m:=\left\{
\begin{array}{lcl}
p^-_{m},\,\,\,\,\,\,\,\,\,\,\,\,\textrm{if}\,\,\,\,\|w(t)\|_{_{r(t)+p_m(\cdot)-2,\Omega}}>1,\\[1ex]
p^+_{m},\,\,\,\,\,\,\,\,\,\,\,\,\textrm{if}\,\,\,\,\|w(t)\|_{_{r(t)+p_m(\cdot)-2,\Omega}}\leq1,\\
\end{array}
\right.\indent\textrm{and}\indent\check{q}_m:=\left\{
\begin{array}{lcl}
q^-_{m},\,\,\,\,\,\,\,\,\,\,\,\,\textrm{if}\,\,\,\,\|w(t)\|_{_{r(t)+q_m(\cdot)-2,\Gamma}}>1,\\[1ex]
q^+_{m},\,\,\,\,\,\,\,\,\,\,\,\,\textrm{if}\,\,\,\,\|w(t)\|_{_{r(t)+q_m(\cdot)-2,\Gamma}}\leq1,\\
\end{array}
\right.
\end{equation}
From here, the proof is completed by setting $\mathfrak{P}_{p,q}(t)\geq0$ and $\mathfrak{Q}_{p,q}(t)$ as in (\ref{P}) and (\ref{Q}), respectively.
\end{proof}

\begin{remark}\label{JDE}
In \cite{BOU-AVS18}, the authors established for the first time an ultracontractivity-type bound for nonlnear semigroups associated to an anisotropic
boundary value problem. To develop such bound for anisotropic differential equations, in the last two
inequalities of \cite[proof of Lemma 10]{BOU-AVS18}, it was argued that $-\displaystyle\sum^N_{i=1}|\pxi u|^{p_i(x)}\leq -c
\displaystyle\sum^N_{i=1}|\pxi u|^{p_m(x)}$ for some constant $c>0$ and a.e.~$x\in\Omega$, and this was used to obtain the fulfillment of the
rest of the results following the approach as in \cite{VELEZ2012-3}. Unfortunately, the correct inequality reads as
$$\displaystyle\sum^N_{i=1}|\pxi u|^{p_i(x)}\leq -c_{\epsilon}
\displaystyle\sum^N_{i=1}|\pxi u|^{p_m(x)}+\epsilon\,\,\,\,\,\,\,\,\textrm{for a.e.}\,\,x\in\Omega,\,\,\,\,\,\textrm{for each}\,\,\epsilon>0,
\,\,\,\,\,\textrm{and for some constant}\,\,c_{\epsilon}>0.$$
and it may not be valid for $\epsilon=0$. Furthermore, the presence of a positive constant in the above correct inequality is sufficient
to hinder the subsequent results, so a different approach needed to be followed. Consequently, in the above lemmas, we
kept the anisotropy intact, and instead derived an alternative way to deduce the appropriate estimations and modifications
to obtain desirable conclusions in the key technical lemmas, which form a basis for the proofs of the main results of the section.
In particular, the methods employed in this paper provide a way to fix the error used in the derivation of ultracontractivity bounds in \cite{BOU-AVS18}.
\end{remark}

\indent Now we are ready to state and prove the main result of this subsection, which is following.\\

\begin{theorem}\label{MainT1}
Let $(\pp,\,\qq)\in C^{0,1}(\overline{\Omega})^N\times C^{0,1}(\overline{\Omega})^{N-1}$
be such that $2<p^-_m\leq p^+_M<\infty$ and $2<q^-_m\leq q^+_M<\infty$, assume that $(p_m,q_m)\in C^{0,1}(\overline{\Omega})\times
C^{0,1}(\overline{\Omega})$, and let $\left\{T_{\sigma}(t)\right\}_{t\geq0}$ be the submarkovian
$C_0$-semigroup on $\mathbb{X\!}^{\,2}(\overline{\Omega})$ generated by $\partial\Phi_{\sigma}$.
Let $(r,s)\in\mathcal{P}(\Omega)\times\mathcal{P}(\Gamma)$ be such that $2\leq r^-\leq r^+\leq\infty$ and
$2\leq s^-\leq s^+\leq\infty$. Then there exist positive constants $C,\,C',\,\kappa$, and $\gamma\in(0,1)$ such that\\
\begin{equation}\label{4.2.12}
|\|T_{\sigma}(t)\mathbf{u}_0-T_{\sigma}(t)\mathbf{v}_0\||_{_{\infty}}\,\leq\,
C\,e^{C't}t^{-\kappa}|\|\mathbf{u}_0-\mathbf{v}_0\||^{\gamma}_{_{r(\cdot),s(\cdot)}},
\end{equation}
for every $\mathbf{u}_0,\,\mathbf{v}_0\in\mathbb{X\!}^{\,r(\cdot),s(\cdot)}(\overline{\Omega})$, for all $t>0$,
where $$\gamma:=\left(\frac{r^-\wedge s^-}
{(r^-\wedge s^-)+\tilde{p}_m-2}\right)^{^{\frac{(\check{p}_m-2)\tilde{p}_m}{\tilde{p}_m-2}}}
\left(\frac{r^-\wedge s^-}{(r^-\wedge s^-)+\tilde{q}_m-2}\right)^{^{\frac{(\check{q}_m-2)\tilde{q}_m}{\tilde{q}_m-2}}}\in(0,1),$$
$\kappa=k_1>0$, and $C'=k_2>0$, for $k_1,\,k_2$ given by (\ref{k1}) and (\ref{k2}), respectively.
\end{theorem}
\indent\\
\begin{proof}
We first recall that since the embedding and trace theorems become sharper whenever $p^+_M\geq N$ and $q^+_M\geq N-1$,
we concentrate in the critical case, namely, when $p^+_M<N$ and $q^+_M<N-1$. We can also assume (without loss of generality) that
$\max\{r^+,s^+\}<\infty$. To proceed with the proof, we first take $\mathbf{u}_0,\,\mathbf{v}_0\in\mathbb{X\!}^{\,\infty}(\overline{\Omega})$,
and we put $\mathbf{u}(t):=T_{\sigma}(t)\mathbf{u}_0$ and $\mathbf{v}(t):=T_{\sigma}(t)\mathbf{v}_0$. Given
$\theta:[0,\infty)\rightarrow[2,\infty)$ an increasing differentiable function, let $\mathfrak{P}_{p,q}(t),\,\mathfrak{Q}_{p,q}(t)$ be defined as
in (\ref{P}) and (\ref{Q}) (for $r(t)=\theta(t)$), and set\\
$$\mathbf{Y}(t):=\log\left(|\|\mathbf{u}(t)-\mathbf{v}(t)\||_{_{\theta(t)}}\right).$$ Then by (\ref{4.2.10}) the function
$\mathbf{Y}(t)$ fulfills the ordinary differential inequality
\begin{equation}\label{4.2.13}
\mathbf{Y}'(t)+\mathfrak{P}_{p,q}(t)\mathbf{Y}(t)+\mathfrak{Q}_{p,q}(t)\,\leq\,0.
\end{equation}
Now, the unique solution of the differential equation
$$\mathbf{X}'(t)+\mathfrak{P}_{p,q}(t)\mathbf{X}(t)+\mathfrak{Q}_{p,q}(t)=0,\indent\indent\mathbf{X}(0)=\mathbf{Y}(0)$$
is given by
\begin{equation}\label{4.2.14}
\mathbf{X}(t)=\exp\left(-\int^t_0\mathfrak{P}_{p,q}(\xi)\,d\xi\right)\left(\mathbf{Y}(0)-\displaystyle\int^t_0\mathfrak{Q}_{p,q}(\xi)
\exp\left(\int^\xi_0\mathfrak{P}_{p,q}(\tau)\,d\tau\right)d\xi\right),
\end{equation}\indent\\
and thus the solution $\mathbf{Y}(t)$ of (\ref{4.2.13}) satisfies $\mathbf{Y}(t)\leq\mathbf{X}(t)$ over $[0,\infty)$.
Next, given $(r,s)\in\mathcal{P}(\Omega)\times\mathcal{P}(\Gamma)$ as in the theorem, fix $t>0$, take $\xi\in[0,t)$ and set
$$\theta(\xi):=\frac{t(r^-\wedge s^-)}{t-\xi},\indent\indent\indent\textrm{for}\,\,\,r^-\wedge s^-:=\min\{r^-,s^-\}\,\geq\,2.$$
Then $\theta(\cdot)$ is an increasing differentiable function over $[0,t)$, with
$\theta(\xi)\geq2$ for all $\xi\in[0,t)$. Moreover, using this function, we obtain that
\begin{equation}\label{P-rs}
\mathfrak{P}_{p,q}(\xi):=\displaystyle\frac{(\check{p}_m-2)\tilde{p}_m}{[(r^-\wedge s^-)+\tilde{p}_m-2]t-(\tilde{p}_m-2)\xi}+
\displaystyle\frac{(\check{q}_m-2)\tilde{q}_m}{[(r^-\wedge s^-)+\tilde{q}_m-2]t-(\tilde{q}_m-2)\xi}\indent\indent\indent\indent\indent
\end{equation}
and\\[2ex]
$\mathfrak{Q}_{p,q}(\xi):=\displaystyle\frac{\tilde{p}_m}{[(r^-\wedge s^-)+\tilde{p}_m-2]t-(\tilde{p}_m-2)\xi}
\log\left(\displaystyle\frac{\mathfrak{C}^{\ast}_p\,\tilde{p}^{^{p^+_M-1}}_m[(r^-\wedge s^-)t-\xi](t-\xi)^{^{p^+_M-1}}}{
C'_{\epsilon}\left\{[(r^-\wedge s^-)+\tilde{p}_m-2]t-(\tilde{p}_m-2)\xi\right\}^{^{p^+_M-1}}}\right)+$\\
$$\indent\indent\indent\indent\indent\indent\indent\indent\indent\indent
-2\mathfrak{C}^{\ast}_p\kappa_p\left(\displaystyle\frac{\tilde{p}_m[(r^-\wedge s^-)t-\xi]^{^{\frac{1}{p^+_M}}}(t-\xi)^{^{\frac{p^+_M-1}{p^+_M}}}}
{[(r^-\wedge s^-)+\tilde{p}_m-2]t-(\tilde{p}_m-2)\xi}\right)^{p^+_M}+$$
$$\indent\indent\indent\indent+\displaystyle\frac{\tilde{q}_m}{[(r^-\wedge s^-)+\tilde{q}_m-2]t-(\tilde{q}_m-2)\xi}
\log\left(\displaystyle\frac{\mathfrak{C}^{\ast}_q\,\tilde{q}^{^{q^+_M-1}}_m[(r^-\wedge s^-)t-\xi](t-\xi)^{^{q^+_M-1}}}{
C''_{\epsilon'}\left\{[(r^-\wedge s^-)+\tilde{q}_m-2]t-(\tilde{q}_m-2)\xi\right\}^{^{q^+_M-1}}}\right)+$$
\begin{equation}\label{Q-rs}
\indent\indent\indent\indent\indent\indent\indent\indent\indent\indent
-2\mathfrak{C}^{\ast}_q\kappa_q\left(\displaystyle\frac{\tilde{q}_m[(r^-\wedge s^-)t-\xi]^{^{\frac{1}{q^+_M}}}(t-\xi)^{^{\frac{q^+_M-1}{q^+_M}}}}
{[(r^-\wedge s^-)+\tilde{q}_m-2]t-(\tilde{q}_m-2)\xi}\right)^{q^+_M}
\end{equation}
\indent\\
(here we recall that the explicit definitions of the constants appearing in the above expressions are given in Lemma \ref{Lem02} and Lemma \ref{Lem03}).
Since\\[2ex] \,\,\,$\displaystyle\int^{\xi}_0\mathfrak{P}_{p,q}(\tau)\,d\tau$\\
$$=\frac{(\check{p}_m-2)\tilde{p}_m}{\tilde{p}_m-2}\log\left(\frac{[(r^-\wedge s^-)+\tilde{p}_m-2]t}{[(r^-\wedge s^-)+\tilde{p}_m-2]t-(\tilde{p}_m-2)\xi}\right)+
\frac{(\check{q}_m-2)\tilde{q}_m}{\tilde{q}_m-2}\log\left(\frac{[(r^-\wedge s^-)+\tilde{q}_m-2]t}{[(r^-\wedge s^-)+
\tilde{q}_m-2]t-(\tilde{q}_m-2)\xi}\right),$$ we see that
$$\exp\left(-\int^{\xi}_0\mathfrak{P}_{p,q}(\tau)\,d\tau\right)=\left(\frac{[(r^-\wedge s^-)+\tilde{p}_m-2]t-(\tilde{p}_m-2)\xi}
{[(r^-\wedge s^-)+\tilde{p}_m-2]t}\right)^{^{\frac{(\check{p}_m-2)\tilde{p}_m}{\tilde{p}_m-2}}}
\left(\frac{[(r^-\wedge s^-)+\tilde{q}_m-2]t-(\tilde{q}_m-2)\xi}
{[(r^-\wedge s^-)+\tilde{q}_m-2]t}\right)^{^{\frac{(\check{q}_m-2)\tilde{q}_m}{\tilde{q}_m-2}}}.$$ Thus,
\begin{equation}\label{lim1}
\displaystyle\lim_{\xi\rightarrow t^{-}}\exp\left(-\int^{\xi}_0\mathfrak{P}_{p,q}(\tau)\,d\tau\right)=\left(\frac{r^-\wedge s^-}
{(r^-\wedge s^-)+\tilde{p}_m-2}\right)^{^{\frac{(\check{p}_m-2)\tilde{p}_m}{\tilde{p}_m-2}}}
\left(\frac{r^-\wedge s^-}{(r^-\wedge s^-)+\tilde{q}_m-2}\right)^{^{\frac{(\check{q}_m-2)\tilde{q}_m}{\tilde{q}_m-2}}}.
\end{equation}
\indent\\
Furthermore, calculating, we find that\\
$$\displaystyle\lim_{\xi\rightarrow t^{-}}\left(\mathbf{Y}(0)-\int^{\xi}_0\mathfrak{Q}_{p,q}(\tau)\,\exp\left(\int^{\tau}_0
\mathfrak{P}_{p,q}(l)\,dl\right)\,d\tau\right)=\mathbf{Y}(0)-k_1\log(t)+k_2t-k_3$$ and
\begin{equation}\label{lim2}
\displaystyle\lim_{\xi\rightarrow t^{-}}\mathbf{X}(\xi)=\left(\frac{r^-\wedge s^-}
{(r^-\wedge s^-)+\tilde{p}_m-2}\right)^{^{\frac{(\check{p}_m-2)\tilde{p}_m}{\tilde{p}_m-2}}}
\left(\frac{r^-\wedge s^-}{(r^-\wedge s^-)+\tilde{q}_m-2}\right)^{^{\frac{(\check{q}_m-2)\tilde{q}_m}{\tilde{q}_m-2}}}
\left(\mathbf{Y}(0)-k_1\log(t)+k_2t-k_3\right),
\end{equation}
\indent\\
where the constants $k_1,\,k_2,\,k_3$ have been computed explicitly in Appendix A (see Proposition A1).
Since $T_{\sigma}$ is non-expansive over $\mathbb{X\!}^{\,\theta(\xi)}(\overline{\Omega})$, it follows that\\
\begin{equation}\label{4.2.15}
|\|\mathbf{u}(t)-\mathbf{v}(t)\||_{_{\theta(\xi)}}\,\leq\,|\|\mathbf{u}(\xi)-\mathbf{v}(\xi)\||_{_{\theta(\xi)}}
\,\leq\,e^{\mathbf{Y}(\xi)}\,\leq\,e^{\mathbf{X}(\xi)},
\end{equation}\indent\\
whenever $0\leq s<t$. Using the fact that $Y(0)=\log\left(|\|\mathbf{u}_0-\mathbf{v}_0\||_{_{\theta(0)}}\right)$,
$\theta(0)=r^-\wedge s^-$ and $\displaystyle\lim_{\xi\rightarrow t^{-}}\theta(\xi)=\infty$, setting
$\kappa:=k_1>0$ and $$\gamma:=\left(\frac{r^-\wedge s^-}
{(r^-\wedge s^-)+\tilde{p}_m-2}\right)^{^{\frac{(\check{p}_m-2)\tilde{p}_m}{\tilde{p}_m-2}}}
\left(\frac{r^-\wedge s^-}{(r^-\wedge s^-)+\tilde{q}_m-2}\right)^{^{\frac{(\check{q}_m-2)\tilde{q}_m}{\tilde{q}_m-2}}}\in(0,1),$$
we let $\xi$ tend to $t$ (from the left) in (\ref{4.2.15}) and apply H\"older's inequality to deduce that\\
\begin{equation}\label{4.2.16}
|\|T_{\sigma}(t)\mathbf{u}_0-T_{\sigma}(t)\mathbf{v}_0\||_{_{\infty}}\,\leq\,
c_0\,e^{C't}t^{-\kappa}|\|\mathbf{u}_0-\mathbf{v}_0\||^{\gamma}_{_{r^-\wedge s^-}}\,\leq\,
C\,e^{C't}t^{-\kappa}|\|\mathbf{u}_0-\mathbf{v}_0\||^{\gamma}_{_{r(\cdot),s(\cdot)}},
\end{equation}\indent\\
for some constants $c_0=e^{-k_3}$,\,$C=c_0\,c_{_{\overline{\Omega}}}$ (for a constant $c_{_{\overline{\Omega}}}>0$ depending
on $\overline{\Omega}$, given by H\"older's inequality), and $C':=k_2>0$. This leads to the inequality
(\ref{4.2.12}) in the case when $\mathbf{u}_0,\,\mathbf{v}_0\in\mathbb{X\!}^{\,\infty}(\overline{\Omega})$.
To complete the proof, it remains to prove (\ref{4.2.12}) for
$\mathbf{u}_0,\,\mathbf{v}_0\in\mathbb{X\!}^{\,r(\cdot),s(\cdot)}(\overline{\Omega})$. If this is the case, since $\Omega$ is bounded,
we clearly have that $\mathbf{u}_0,\,\mathbf{v}_0\in\mathbb{X\!}^{\,(r\wedge s)(\cdot)}(\overline{\Omega})$. Then,
let $\{\mathbf{u}_{0,n}\},\,\{\mathbf{v}_{0,n}\}\subseteq\mathbb{X\!}^{\,\infty}(\overline{\Omega})$
be sequences converging to $\mathbf{u}_0$ and $\mathbf{v}_0$, respectively (in $\mathbb{X\!}^{\,(r\wedge s)(\cdot)}(\overline{\Omega})$),
and put $\mathbf{u}_n(t):=T_{\sigma}(t)\mathbf{u}_{0,n}$ and
$\mathbf{v}_n(t):=T_{\sigma}(t)\mathbf{v}_{0,n}$. By (\ref{4.2.16}) we clearly get that $\mathbf{u}_n(t)\stackrel{n\rightarrow\infty}
\longrightarrow \mathbf{u}(t)$ and $\mathbf{v}_n(t)\stackrel{n\rightarrow\infty}\longrightarrow\mathbf{v}(t)$ in
$\mathbb{X\!}^{\,\infty}(\overline{\Omega})$ for each $t>0$. Thus for every $t>0$ the sequence $\{\mathbf{u}_n(t)-\mathbf{v}_n(t)\}$
converges in $\mathbb{X\!}^{\,\infty}(\overline{\Omega})$, and moreover
$\mathbf{u}_n(t)-\mathbf{v}_n(t)\stackrel{n\rightarrow\infty}\longrightarrow\mathbf{u}(t)-\mathbf{v}(t)$
by uniqueness of the limit. Hence (\ref{4.2.12}) is valid for
$\mathbf{u}_0,\,\mathbf{v}_0\in\mathbb{X\!}^{\,r(\cdot),s(\cdot)}(\overline{\Omega})$, as desired.
\end{proof}

\indent \indent Note that if $\check{p}_m=\check{q}_m=2$, then from the proof of Lemma \ref{Lem03}, one sees that
the inequality (\ref{4.2.10}) is valid for $\mathfrak{P}_{p,q}(t)=0$, and the procedure in the proof of Theorem \ref{MainT1} may
not be completely valid for this case. However, in this case we can establish an even sharper ultracontractivity property for
$\left\{T_{\sigma}(t)\right\}_{t\geq0}$, namely, the Lipschitz-ultracontractivity property. The next result assures this fact.\\

\begin{theorem}\label{MainT2}
In addition to all the assumptions in Theorem \ref{MainT1}, suppose that $\check{p}_m=\check{q}_m=2$,
and let $\left\{T_{\sigma}(t)\right\}_{t\geq0}$ be the submarkovian
$C_0$-semigroup on $\mathbb{X\!}^{\,2}(\overline{\Omega})$ generated by $\partial\Phi_{\sigma}$.
If $(r,s)\in\mathcal{P}(\Omega)\times\mathcal{P}(\Gamma)$ fulfills $2\leq r^-\leq r^+\leq\infty$ and
$2\leq s^-\leq s^+\leq\infty$, then there exist positive constants $C_0,\,C'_0,\,\kappa$ such that\\
\begin{equation}\label{4.2.17}
|\|T_{\sigma}(t)\mathbf{u}_0-T_{\sigma}(t)\mathbf{v}_0\||_{_{\infty}}\,\leq\,
C_0\,e^{C'_0}t^{-\kappa}|\|\mathbf{u}_0-\mathbf{v}_0\||_{_{r(\cdot),s(\cdot)}},
\end{equation}\indent\\
for every $\mathbf{u}_0,\,\mathbf{v}_0\in\mathbb{X\!}^{\,r(\cdot),s(\cdot)}(\overline{\Omega})$ and for all $t>0$, where
$$\kappa:=\displaystyle\frac{\tilde{p}_m}{\tilde{p}_m-2}\log\left(\frac{(r^-\wedge s^-)+\tilde{p}_m-2}{(r^-\wedge s^-)}\right)+
\displaystyle\frac{\tilde{q}_m}{\tilde{q}_m-2}\log\left(\frac{(r^-\wedge s^-)+\tilde{q}_m-2}{(r^-\wedge s^-)}\right),$$
$$C'_0:=2\mathfrak{C}^{\ast}_p\kappa_p\,\tilde{p}_m^{^{p^+_M}}J_{5,p}+2\mathfrak{C}^{\ast}_q\kappa_q\,\tilde{q}_m^{^{q^+_M}}J_{5,q},
\indent\indent\indent\indent\indent$$
and
$$C_0:=c_{_{\overline{\Omega}}}\exp\left(\tilde{p}_m\left[\log\left(\frac{\mathfrak{C}^{\ast}_p\,\tilde{p}_m^{^{p^+_M-1}}}{C'_{\epsilon}}\right)
J_{1,p}+(p^+_M-1)(J_{3,p}-J_{2,p})+J_{4,p}\right]+\right.\indent\indent\indent\indent\indent\indent$$
$$\indent\indent\indent\indent\indent\indent\left.
+\tilde{q}_m\left[\log\left(\frac{\mathfrak{C}^{\ast}_q\,\tilde{q}_m^{^{q^+_M-1}}}{C''_{\epsilon}}\right)J_{1,q}
+(q^+_M-1)(J_{3,q}-J_{2,q})+J_{4,q}\right]\right),$$
for the constants $J_{j,p},\,J_{j,q}$ ($1\leq j\leq 5$) given by (\ref{A13})--(\ref{A22}).
\end{theorem}
\indent\\

\begin{proof}
Under the assumptions in the theorem, by the last argument in the proof of Theorem \ref{MainT1}, it
suffices to prove (\ref{4.2.17}) for $\mathbf{u}_0,\,\mathbf{v}_0\in\mathbb{X\!}^{\,\infty}(\overline{\Omega})$. Because $\mathfrak{P}_{p,q}(t)=0$,
letting $$\mathbf{Y}(t):=\log\left(|\|\mathbf{u}(t)-\mathbf{v}(t)\||_{_{\theta(t)}}\right)\,\indent\textrm{and}\,\indent
\mathcal{L}(t):=\displaystyle\int^t_0\mathfrak{Q}_{p,q}(\tau)\,d\tau,$$ we get from (\ref{4.2.10}) that $\mathbf{Y}(t)$ fulfills the
differential inequality $$\frac{d}{dt}\left(\mathbf{Y}(t)+\mathcal{L}(t)\right)\,\leq\,0,\indent\textrm{for\, a.e.}\,\,\,t\geq0.$$
This entails that\\
\begin{equation}\label{4.2.18}
\mathbf{Y}(t)+\mathcal{L}(t)\,\leq\,\mathbf{Y}(0)+\mathcal{L}(0)=\mathbf{Y}(0),\indent\textrm{for\, a.e.}\,\,\,t\geq0.
\end{equation}
\indent\\
As in the proof of Theorem \ref{MainT1}, let $(r,s)\in\mathcal{P}(\Omega)\times\mathcal{P}(\Omega)$ be such that
$2\leq r^-\leq r^+<\infty$ and $2\leq s^-\leq s^+<\infty$, fix $t>0$, and set
$\theta(\xi):=\frac{(r^-\wedge s^-)t}{t-\xi}$ over $[0,t)$. Then proceeding in a similar way as before (see also Proposition A1), we deduce that\\[2ex]
$\displaystyle\lim_{\xi\rightarrow t^{-}}\mathcal{L}(\xi)=\left\{
\displaystyle\frac{\tilde{p}_m}{\tilde{p}_m-2}\log\left(\frac{(r^-\wedge s^-)+\tilde{p}_m-2}{(r^-\wedge s^-)}\right)+
\displaystyle\frac{\tilde{q}_m}{\tilde{q}_m-2}\log\left(\frac{(r^-\wedge s^-)+\tilde{q}_m-2}{(r^-\wedge s^-)}\right)\right\}\log(t)+$\\
$$-\left(2\mathfrak{C}^{\ast}_p\kappa_p\,\tilde{p}_m^{^{p^+_M}}J_{5,p}+2\mathfrak{C}^{\ast}_q\kappa_q\,\tilde{q}_m^{^{q^+_M}}J_{5,q}\right)t+\left\{
\tilde{p}_m\left(\log\left(\frac{\mathfrak{C}^{\ast}_p\,\tilde{p}_m^{^{p^+_M-1}}}{C'_{\epsilon}}\right)J_{1,p}+(p^+_M-1)(J_{3,p}-J_{2,p})+J_{4,p}\right)+\right.$$
\begin{equation}\label{lim3}
\indent\indent\indent\indent\indent\left.
+\tilde{q}_m\left(\log\left(\frac{\mathfrak{C}^{\ast}_q\,\tilde{q}_m^{^{q^+_M-1}}}{C''_{\epsilon}}\right)J_{1,q}+(q^+_M-1)(J_{3,q}-J_{2,q})+J_{4,q}\right)\right\},
\end{equation}
where the constants $J_{j,p},\,J_{j,q}$ are given by (\ref{A13})--(\ref{A22}) ($1\leq j\leq 5$).
Following the approach as in the proof of the previous theorem, we deduce from (\ref{4.2.18}) that\\
\begin{equation}\label{4.2.19}
|\|\mathbf{u}(t)-\mathbf{v}(t)\||_{_{\theta(\xi)}}\,\leq\,e^{-\mathcal{L}(\xi)}e^{\mathbf{Y}(0)},\,\,\,\indent\textrm{for every}\,\,\,\xi\in[0,t).
\end{equation}\indent\\
Letting $C_0:=e^{k_6}\,c_{_{\overline{\Omega}}}$,\,\,$C'_0:=-k_5$ (for $k_5,\,k_6$ given by (\ref{k5}) and (\ref{k6}), respectively), and
$$\kappa:=\displaystyle\frac{\tilde{p}_m}{\tilde{p}_m-2}\log\left(\frac{(r^-\wedge s^-)+\tilde{p}_m-2}{(r^-\wedge s^-)}\right)+
\displaystyle\frac{\tilde{q}_m}{\tilde{q}_m-2}\log\left(\frac{(r^-\wedge s^-)+\tilde{q}_m-2}{(r^-\wedge s^-)}\right),$$
we let $\xi$ tend to $t$ from the left in (\ref{4.2.19}), and then apply
H\"older's inequality to achieve (\ref{4.2.17}), as desired.
\end{proof}
\indent\\
\begin{remark}
The definition of $\mathfrak{P}_{p,q}(t)$
in (\ref{P}) allows the possibility to consider the existence of cases in which either $p^-_m\in(1,2)$, or
$q^-_m\in(1,2)$ (but not both at the same time), and $\mathfrak{P}_{p,q}(t)$ remaining nonnegative. This may be the case for sufficient high dimension of the
space ($N>2$ sufficiently large), in which case if $p^-_m\in(1,2)$, then $q^-_m$ should be sufficiently large. The same if the opposite condition holds.
We will not add the details about this here.\\
\end{remark}

\begin{center}
\textsc{Appendix A: Derivation of the explicit constants in (\ref{lim2}) and (\ref{lim3})}
\end{center}
\indent\\

This part is devoted in establishing the exact values of the constants in (\ref{lim2}) and (\ref{lim3}). More precisely,
we will establish the following result.\\

\noindent{\bf Proposition A1.} {\it If $\mathfrak{P}_{p,q}(\cdot)$ and $\mathfrak{Q}_{p,q}(\cdot)$ are the functions given by (\ref{P-rs}) and
(\ref{Q-rs}), respectively, then we have that
\begin{apequation}\label{A01}
\displaystyle\lim_{\xi\rightarrow t^{-}}\int^{\xi}_0\mathfrak{Q}_{p,q}(\tau)\,\exp\left(\int^{\tau}_0
\mathfrak{P}_{p,q}(l)\,dl\right)\,d\tau=k_1\log(t)-k_2t+k_3,
\end{apequation}
and
\begin{apequation}\label{A02}
\displaystyle\lim_{\xi\rightarrow t^{-}}\displaystyle\int^{\xi}_0\mathfrak{Q}_{p,q}(\tau)\,d\tau=k_4\log(t)-k_5t+k_6,
\end{apequation}
where the constants $k_m$ ($1\leq m\leq 6$) are given explicitly as follows:}
\begin{apequation}\label{k1}
k_1:=\frac{\tilde{p}_m}{(r^-\wedge s^-)+\tilde{p}_m-2}I_{1,p}+\frac{\tilde{q}_m}{(r^-\wedge s^-)+\tilde{q}_m-2}I_{1,q},
\end{apequation}
\begin{apequation}\label{k2}
k_2:=\frac{2\mathfrak{C}^{\ast}_p\kappa_p\,\tilde{p}^{^{p_{M}^{+}}}_m}{(r^-\wedge s^-)+\tilde{p}_m-2}I_{5,p}+
\frac{2\mathfrak{C}^{\ast}_q\kappa_q\,\tilde{q}^{^{q_{M}^{+}}}_m}{(r^-\wedge s^-)+\tilde{q}_m-2}I_{5,q},
\end{apequation}
\indent\\
\indent\indent$k_3:=\displaystyle\frac{\tilde{p}_m}{(r^-\wedge s^-)+\tilde{p}_m-2}\left(\log\left(\frac{\mathfrak{C}^{\ast}_p\,\tilde{p}^{^{p_{M}^{+}-1}}_m}
{C'_{\epsilon}}\right)I_{1,p}+(p_{M}^{+}-1)(I_{2,p}-I_{3,p})+I_{4,p}\right)+$\\
\begin{apequation}\label{k3}
\indent\indent\indent+\displaystyle\frac{\tilde{q}_m}{(r^-\wedge s^-)+\tilde{q}_m-2}\left(\log\left(\frac{\mathfrak{C}^{\ast}_q\,\tilde{q}^{^{q_{M}^{+}-1}}}
{C''_{\epsilon}}\right)I_{1,q}+(q_{M}^{+}-1)(I_{2,q}-I_{3,q})+I_{4,q}\right),
\end{apequation}
\begin{apequation}\label{k4}
k_4:=\displaystyle\frac{\tilde{p}_m}{\tilde{p}_m-2}\log\left(\frac{(r^-\wedge s^-)+\tilde{p}_m-2}{(r^-\wedge s^-)}\right)+
\displaystyle\frac{\tilde{q}_m}{\tilde{q}_m-2}\log\left(\frac{(r^-\wedge s^-)+\tilde{q}_m-2}{(r^-\wedge s^-)}\right),
\end{apequation}
\begin{apequation}\label{k5}
k_5:=2\mathfrak{C}^{\ast}_p\kappa_p\,\tilde{p}_m^{^{p^+_M}}J_{5,p}+2\mathfrak{C}^{\ast}_q\kappa_q\,\tilde{q}_m^{^{q^+_M}}J_{5,q},
\indent\indent\indent\indent\indent\indent\indent\indent\indent
\end{apequation}
and\\[2ex]
\indent$k_6:=\tilde{p}_m\left[\log\left(\frac{\mathfrak{C}^{\ast}_p\,\tilde{p}_m^{^{p^+_M-1}}}{C'_{\epsilon}}\right)
J_{1,p}+(p^+_M-1)(J_{3,p}-J_{2,p})+J_{4,p}\right]$\\
\begin{apequation}\label{k6}
\indent\indent\indent\indent\indent\indent
+\tilde{q}_m\left[\log\left(\frac{\mathfrak{C}^{\ast}_q\,\tilde{q}_m^{^{q^+_M-1}}}{C''_{\epsilon}}\right)J_{1,q}
+(q^+_M-1)(J_{3,q}-J_{2,q})+J_{4,q}\right],
\end{apequation}
{\it for the constants $I_{i,p},\,I_{j,q}$ ($1\leq i,\,j\leq 5$) given by (\ref{A03})--(\ref{A12}), and the constants
$J_{j,p},\,J_{j,q}$ ($1\leq j\leq 5$) given by (\ref{A13})--(\ref{A22})}.
\indent\\

\begin{proof}
To simplify the notation, we write\, $a:=r^-\wedge s^-$, \, $b:=\check{p}_{m}$, \, $c:=\check{q}_{m}$, \, $d_1:=p_{M}^{+}$,  \, $d_2:=q_{M}^{+}$,
\, $p:=\tilde{p}_{m}$, and\, $q:=\tilde{q}_{m}$. Then

\begin{center}
\noindent$\displaystyle\int_{0}^{t} \mathfrak{Q}_{p,q}(\tau) \ \exp\left(\int_{0}^{\tau} \mathfrak{P}_{p,q}(l) \ dl\right) d\tau =p \int_{0}^{t} \frac{1}{[a+p-2]t-(p-2)\tau} \ \log\left(\frac{C_{p}^{*}p^{d_1-1} [at-\tau] (t-\tau)^{d_1-1}}{C'_{\epsilon}[(a+p-2)t-(p-2) \tau ]^{d_1-1}} \right) $ \\

\vspace{0.2cm}

$\times\left(\displaystyle\frac{(a+p-2)t-(p-2)\tau}{[a+p-2]t}\right)^{\frac{(b-2)p}{p-2}} \ \left(\displaystyle\frac{(a+q-2)t-(q-2)\tau}{[a+q-2]t}\right)^{\frac{(c-2)q}{q-2}} \ d\tau$ \\

\vspace{0.2cm}

$-2C_{p}^{*} \kappa_{p} \displaystyle\int_{0}^{t} \left(\frac{p (at-\tau)^{\frac{1}{d_1}} (t-\tau)^{\frac{d_1-1}{d_1}}}{(a+p-2)t-(p-2)\tau}\right)^{d_1} \
\left(\displaystyle\frac{(a+p-2)t-(p-2)\tau}{[a+p-2]t}\right)^{\frac{(b-2)p}{p-2}} \left(\displaystyle\frac{(a+q-2)t-(q-2)\tau}{[a+q-2]t}\right)^{\frac{(c-2)q}{q-2}} \ d\tau$ \\

\vspace{0.2cm}

$+q  \displaystyle\int_{0}^{t} \frac{1}{[a+q-2]t-(q-2)\tau} \ \log\left(\frac{C_{q}^{*}q^{d_2-1}(at-\tau)(t-\tau)^{d_2-1}}{C''_{\epsilon}[(a+q-2)t-(q-2)
\tau ]^{d_2-1}} \right)\indent\indent\indent\indent\indent$ \\

\vspace{0.2cm}

$\times\left(\displaystyle\frac{(a+p-2)t-(p-2)\tau}{[a+p-2]t}\right)^{\frac{(b-2)p}{p-2}} \ \left(\displaystyle\frac{(a+q-2)t-(q-2)\tau}{[a+q-2]t}\right)^{\frac{(c-2)q}{q-2}} \ d\tau$ \\

\vspace{0.2cm}

$-2C_{q}^{*} \kappa_{q} \displaystyle\int_{0}^{t} \left(\frac{q(at-\tau)^{\frac{1}{d_2}}(t-\tau)^{\frac{d_2-1}{d_2}}}{(a+q-2)t-(q-2)\tau}\right)^{d_2} \
\left(\displaystyle\frac{(a+p-2)t-(p-2)\tau}{[a+p-2]t}\right)^{\frac{(b-2)p}{p-2}} \left(\displaystyle\frac{(a+q-2)t-(q-2)\tau}{[a+q-2]t}\right)^{\frac{(c-2)q}{q-2}} \ d\tau$

\end{center}

\noindent Making a standard substitution $x=\tau/t$ (for $t>0$), we find that

\vspace{0.2cm}

$p \displaystyle\int_{0}^{t} \frac{1}{[a+p-2]t-(p-2)\tau} \ \log\left(\frac{C_{p}^{*}p^{d_1-1}(at-\tau) (t-\tau)^{d_1-1}}{C'_{\epsilon}[(a+p-2)t-(p-2) \tau ]^{d_1-1}} \right) \left(\displaystyle\frac{(a+p-2)t-(p-2)\tau}{[a+p-2]t}\right)^{\frac{(b-2)p}{p-2}} $

\vspace{0.2cm}

\begin{center}
$\times\left(\displaystyle\frac{(a+q-2)t-(q-2)\tau}{[a+q-2]t}\right)^{\frac{(c-2)q}{q-2}} \ d\tau$
\end{center}

$=p \displaystyle\int_{0}^{1} \frac{t}{[a+p-2]t-(p-2)tx} \ \log\left(\frac{C_{p}^{*}p^{d_1-1}(at-tx)(t-tx)^{d_1-1}}{C'_{\epsilon}[(a+p-2)t-(p-2) tx ]^{d_1-1}} \right) \left(\displaystyle\frac{(a+p-2)t-(p-2)tx}{[a+p-2]t}\right)^{\frac{(b-2)p}{p-2}}$

\vspace{0.2cm}

\begin{center}
s$\times \left(\displaystyle\frac{(a+q-2)t-(q-2)tx}{[a+q-2]t}\right)^{\frac{(c-2)q}{q-2}} \ dx$
\end{center}

$=p \displaystyle\int_{0}^{1} \frac{1}{[a+p-2]-(p-2)x} \ \log\left(\frac{C_{p}^{*}p^{d_1-1} t (a-x) (1-x)^{d_1-1}}{C'_{\epsilon}[(a+p-2)-(p-2) x ]^{d_1-1}} \right)\left(\displaystyle\frac{(a+p-2)-(p-2)x}{[a+p-2]}\right)^{\frac{(b-2)p}{p-2}} $

\vspace{0.2cm}

\begin{center}
$\times \left(\displaystyle\frac{(a+q-2)-(q-2)x}{[a+q-2]}\right)^{\frac{(c-2)q}{q-2}} \ dx$
\end{center}

$=\displaystyle\frac{p}{a+p-2} \int_{0}^{1} \ \left(\displaystyle 1-\frac{(p-2)x}{[a+p-2]}\right)^{\frac{(b-2)p}{p-2}-1} \ \log\left(\frac{C_{p}^{*}p^{d_1-1}t (a-x) (1-x)^{d_1-1}}{ C'_{\epsilon}[a+(p-2)(1-x) ]^{d_1-1}} \right) \left(1-\displaystyle\frac{(q-2)x}{[a+q-2]}\right)^{\frac{(c-2)q}{q-2}} \ dx.$ \\

\vspace{0.2cm}

Also, \\

\vspace{0.2cm}

$-C_{p}^{*} \kappa_{p} \displaystyle\int_{0}^{t} \left(\frac{p(at-\tau)^{\frac{1}{d_1}}(t-\tau)^{\frac{d_1-1}{d_1}}}{(a+p-2)t-(p-2)\tau}\right)^{d_1} \
\left(\displaystyle\frac{(a+p-2)t-(p-2)\tau}{[a+p-2]t}\right)^{\frac{(b-2)p}{p-2}} \left(\displaystyle\frac{(a+q-2)t-(q-2)\tau}{[a+q-2]t}\right)^{\frac{(c-2)q}{q-2}} \ d\tau$ \\

\vspace{0.2cm}

$=- t \ C_{p}^{*} \ \kappa_{p} \displaystyle\int_{0}^{1} \left(\frac{p(at-tx)^{\frac{1}{d_1}}(t-tx)^{\frac{d_1-1}{d_1}}}{(a+p-2)t-(p-2)tx}\right)^{d_1} \left(\displaystyle\frac{(a+p-2)t-(p-2)tx}{[a+p-2]t}\right)^{\frac{(b-2)p}{p-2}} \left(\displaystyle\frac{(a+q-2)t-(q-2)tx}{[a+q-2]t}\right)^{\frac{(c-2)q}{q-2}} dx$ \\

\vspace{0.2cm}
$=- t  \ C_{p}^{*} \ \kappa_{p} \ p^{d_1} \displaystyle\int_{0}^{1} \left(\frac{(a-x)^{\frac{1}{d_1}}(1-x)^{\frac{d_1-1}{d_1}}}{(a+p-2)-(p-2)x}\right)^{d_1} \left(\displaystyle\frac{(a+p-2)-(p-2)x}{[a+p-2]}\right)^{\frac{(b-2)p}{p-2}} \left(\displaystyle\frac{(a+q-2)-(q-2)x}{[a+q-2]}\right)^{\frac{(c-2)q}{q-2}} dx$ \\

\vspace{0.2cm}

$=-\displaystyle\frac{\ C_{p}^{*} \ \kappa_{p} \ p^{d_1}}{(a+p-2)^{d_1}} \ t \ \displaystyle\int_{0}^{1} (a-x) (1-x)^{d_1-1} \left(1-\displaystyle\frac{(p-2)x}{[a+p-2]}\right)^{\frac{(b-2)p}{p-2}-d_1} \left(1-\displaystyle\frac{(q-2)x}{[a+q-2]}\right)^{\frac{(c-2)q}{q-2}} dx$. \\

\vspace{0.2cm}

Proceeding in the exact way, we have that \\

\vspace{0.2cm}

$q \displaystyle\int_{0}^{t} \frac{1}{[a+q-2]t-(q-2)\tau} \ \log\left(\frac{C_{q}^{*}q^{d_2-1}(at-\tau) (t-\tau)^{d_2-1}}{C'_{\epsilon}[(a+q-2)t-(q-2) \tau ]^{d_2-1}} \right) \left(\displaystyle\frac{(a+p-2)t-(p-2)\tau}{[a+p-2]t}\right)^{\frac{(b-2)p}{p-2}} $

\vspace{0.2cm}

\begin{center}
$\times\left(\displaystyle\frac{(a+q-2)t-(q-2)\tau}{[a+q-2]t}\right)^{\frac{(c-2)q}{q-2}} \ d\tau$
\end{center}

\vspace{0.2cm}

$=\displaystyle\frac{q}{a+q-2} \int_{0}^{1} \ \left(\displaystyle 1-\frac{(q-2)x}{[a+q-2]}\right)^{\frac{(c-2)q}{q-2}-1} \ \log\left(\frac{C_{q}^{*}q^{d_2-1} t (a-x)(1-x)^{d_2-1}}{ C'_{\epsilon}[a+(q-2)(1-x) ]^{d_2-1}} \right) \left(1-\displaystyle\frac{(p-2)x}{[a+p-2]}\right)^{\frac{(b-2)p}{p-2}} \ dx$,

\vspace{0.2cm}

and\\

\vspace{0.2cm}

$-C_{q}^{*} \kappa_{q} \displaystyle\int_{0}^{t} \left(\frac{q(at-\tau)^{\frac{1}{d_2}}(t-\tau)^{\frac{d_2-1}{d_2}}}{(a+q-2)t-(q-2)\tau}\right)^{d_2} \
\left(\displaystyle\frac{(a+p-2)t-(p-2)\tau}{[a+p-2]t}\right)^{\frac{(b-2)p}{p-2}} \left(\displaystyle\frac{(a+q-2)t-(q-2)\tau}{[a+q-2]t}\right)^{\frac{(c-2)q}{q-2}} \ d\tau$ \\

\vspace{0.2cm}

\begin{center}
$=-\displaystyle\frac{\ C_{q}^{*} \ \kappa_{q} \ q^{d_2}}{(a+q-2)^{d_2}} \ t \ \displaystyle\int_{0}^{1} (a-x) (1-x)^{d_2-1} \left(1-\displaystyle\frac{(p-2)x}{[a+p-2]}\right)^{\frac{(b-2)p}{p-2}} \left(1-\displaystyle\frac{(q-2)x}{[a+q-2]}\right)^{\frac{(c-2)q}{q-2}-d_2} dx$.\\
\end{center}

Now, observe that\\

\begin{center}
$\displaystyle\frac{p}{a+p-2} \int_{0}^{1} \ \left(\displaystyle 1-\frac{(p-2)x}{[a+p-2]}\right)^{\frac{(b-2)p}{p-2}-1} \ \log\left(\frac{C_{p}^{*}p^{d_1-1} t (a-x) (1-x)^{d_1-1}}{ C'_{\epsilon}[a+(p-2)(1-x) ]^{d_1-1}} \right) \left(1-\displaystyle\frac{(q-2)x}{[a+q-2]}\right)^{\frac{(c-2)q}{q-2}} \ dx$
\end{center}

\vspace{0.2cm}

$=\displaystyle\frac{p}{a+p-2}\log\left(\frac{C_{p}^{*}p^{d_1-1}t}{ C'_{\epsilon}}\right)\displaystyle\int_{0}^{1}\left(\displaystyle 1-\frac{(p-2)x}{[a+p-2]}\right)^{\frac{(b-2)p}{p-2}-1} \left(1-\displaystyle\frac{(q-2)x}{[a+q-2]}\right)^{\frac{(c-2)q}{q-2}} \ dx$

\vspace{0.2cm}

\begin{center}
$+\displaystyle\frac{p(d_1-1)}{a+p-2}\displaystyle\int_{0}^{1} \left(\displaystyle 1-\frac{(p-2)x}{[a+p-2]}\right)^{\frac{(b-2)p}{p-2}-1} \left(1-\displaystyle\frac{(q-2)x}{[a+q-2]}\right)^{\frac{(c-2)q}{q-2}} \ \log(1-x) \ dx$
\end{center}

\vspace{0.2cm}

$-\displaystyle\frac{(d_1-1)p}{a+p-2}\displaystyle\int_{0}^{1} \left(\displaystyle 1-\frac{(p-2)x}{[a+p-2]}\right)^{\frac{(b-2)p}{p-2}-1} \left(1-\displaystyle\frac{(q-2)x}{[a+q-2]}\right)^{\frac{(c-2)q}{q-2}} \ \log[a+(p-2)(1-x)] \ dx$.

\vspace{0.2cm}

\begin{center}
$+\displaystyle\frac{p}{a+p-2}\displaystyle\int_{0}^{1} \left(\displaystyle 1-\frac{(p-2)x}{[a+p-2]}\right)^{\frac{(b-2)p}{p-2}-1} \left(1-\displaystyle\frac{(q-2)x}{[a+q-2]}\right)^{\frac{(c-2)q}{q-2}} \ \log(a-x) \ dx$
\end{center}

\vspace{0.5cm}

In the exact way, we entail that

\begin{center}
$\displaystyle\frac{q}{a+q-2} \int_{0}^{1} \ \left(\displaystyle 1-\frac{(p-2)x}{[a+p-2]}\right)^{\frac{(b-2)p}{p-2}} \ \log\left(\frac{C_{q}^{*}q^{d_2-1} t (a-x) (1-x)^{d_2-1}}{ C''_{\epsilon}[a+(q-2)(1-x) ]^{d_2-1}} \right) \left(1-\displaystyle\frac{(q-2)x}{[a+q-2]}\right)^{\frac{(c-2)q}{q-2}-1} \ dx$
\end{center}

\vspace{0.2cm}

$=\displaystyle\frac{q}{a+q-2}\log\left(\frac{C_{q}^{*}q^{d_2-1}t}{ C''_{\epsilon}}\right)\displaystyle\int_{0}^{1}\left(\displaystyle 1-\frac{(q-2)x}{[a+q-2]}\right)^{\frac{(b-2)p}{p-2}} \left(1-\displaystyle\frac{(q-2)x}{[a+q-2]}\right)^{\frac{(c-2)q}{q-2}-1} \ dx$

\vspace{0.2cm}

\begin{center}
$+\displaystyle\frac{q(d_2-1)}{a+p-2}\displaystyle\int_{0}^{1} \left(\displaystyle 1-\frac{(p-2)x}{[a+p-2]}\right)^{\frac{(b-2)p}{p-2}} \left(1-\displaystyle\frac{(q-2)x}{[a+q-2]}\right)^{\frac{(c-2)q}{q-2}-1} \ \log(1-x) \ dx$
\end{center}

\vspace{0.2cm}

$-\displaystyle\frac{(d_2-1)q}{a+q-2}\displaystyle\int_{0}^{1} \left(\displaystyle 1-\frac{(p-2)x}{[a+p-2]}\right)^{\frac{(b-2)p}{p-2}} \left(1-\displaystyle\frac{(q-2)x}{[a+q-2]}\right)^{\frac{(c-2)q}{q-2}-1} \ \log[a+(q-2)(1-x)] \ dx$.

\vspace{0.2cm}

\begin{center}
$+\displaystyle\frac{q}{a+q-2}\displaystyle\int_{0}^{1} \left(\displaystyle 1-\frac{(p-2)x}{[a+p-2]}\right)^{\frac{(b-2)p}{p-2}} \left(1-\displaystyle\frac{(q-2)x}{[a+q-2]}\right)^{\frac{(c-2)q}{q-2}-1} \ \log(a-x) \ dx$
\end{center}

\vspace{0.5cm}

Then, we denote the following:

\begin{apequation}\label{A03}
I_{1,p}:=\displaystyle\int_{0}^{1}\left(\displaystyle 1-\frac{(p-2)x}{[a+p-2]}\right)^{\frac{(b-2)p}{p-2}-1} \left(1-\displaystyle\frac{(q-2)x}{[a+q-2]}\right)^{\frac{(c-2)q}{q-2}} \ dx,
\end{apequation}

\begin{apequation}\label{A04}
I_{2,p}:=\displaystyle\int_{0}^{1} \left(\displaystyle 1-\frac{(p-2)x}{[a+p-2]}\right)^{\frac{(b-2)p}{p-2}-1} \left(1-\displaystyle\frac{(q-2)x}{[a+q-2]}\right)^{\frac{(c-2)q}{q-2}} \ \log(1-x) \ dx,
\end{apequation}

\begin{apequation}\label{A05}
I_{3,p}:=\displaystyle\int_{0}^{1} \left(\displaystyle 1-\frac{(p-2)x}{[a+p-2]}\right)^{\frac{(b-2)p}{p-2}-1} \left(1-\displaystyle\frac{(q-2)x}{[a+q-2]}\right)^{\frac{(c-2)q}{q-2}} \ \log[a+(p-2)(1-x)] \ dx,
\end{apequation}

\begin{apequation}\label{A06}
I_{4,p}:=\displaystyle\int_{0}^{1} \left(\displaystyle 1-\frac{(p-2)x}{[a+p-2]}\right)^{\frac{(b-2)p}{p-2}-1} \left(1-\displaystyle\frac{(q-2)x}{[a+q-2]}\right)^{\frac{(c-2)q}{q-2}} \ \log(a-x) \ dx,
\end{apequation}

\begin{apequation}\label{A07}
I_{5,p}:=\displaystyle\int_{0}^{1} (a-x) (1-x)^{d_1-1} \left(1-\displaystyle\frac{(p-2)x}{[a+p-2]}\right)^{\frac{(b-2)p}{p-2}-d_1} \left(1-\displaystyle\frac{(q-2)x}{[a+q-2]}\right)^{\frac{(c-2)q}{q-2}} dx,
\end{apequation}

\begin{apequation}\label{A08}
I_{1,q}:=\displaystyle\int_{0}^{1}\left(\displaystyle 1-\frac{(p-2)x}{[a+p-2]}\right)^{\frac{(b-2)p}{p-2}} \left(1-\displaystyle\frac{(q-2)x}{[a+q-2]}\right)^{\frac{(c-2)q}{q-2}-1} \ dx,
\end{apequation}

\begin{apequation}\label{A09}
I_{2,q}:=\displaystyle\int_{0}^{1} \left(\displaystyle 1-\frac{(p-2)x}{[a+p-2]}\right)^{\frac{(b-2)p}{p-2}} \left(1-\displaystyle\frac{(q-2)x}{[a+q-2]}\right)^{\frac{(c-2)q}{q-2}-1} \ \log(1-x) \ dx,
\end{apequation}

\begin{apequation}\label{A10}
I_{3,q}:=\displaystyle\int_{0}^{1} \left(\displaystyle 1-\frac{(p-2)x}{[a+p-2]}\right)^{\frac{(b-2)p}{p-2}} \left(1-\displaystyle\frac{(q-2)x}{[a+q-2]}\right)^{\frac{(c-2)q}{q-2}-1} \ \log[a+(q-2)(1-x)] \ dx,
\end{apequation}

\begin{apequation}\label{A11}
I_{4,q}:=\displaystyle\int_{0}^{1} \left(\displaystyle 1-\frac{(p-2)x}{[a+p-2]}\right)^{\frac{(b-2)p}{p-2}} \left(1-\displaystyle\frac{(q-2)x}{[a+q-2]}\right)^{\frac{(c-2)q}{q-2}-1} \ \log(a-x) \ dx,
\end{apequation}

\begin{apequation}\label{A12}
I_{5,q}:=\displaystyle\int_{0}^{1} (a-x) (1-x)^{d_2-1} \left(1-\displaystyle\frac{(q-2)x}{[a+q-2]}\right)^{\frac{(c-2)q}{q-2}-d_2} \left(1-\displaystyle\frac{(p-2)x}{[a+p-2]}\right)^{\frac{(b-2)p}{p-2}} dx.
\end{apequation}

\noindent Standard arguments show that all these integrals yield positive constants; in particular, the integrals $I_{2,p},\,I_{2,q}$ are convergent.
Therefore, by letting
$$k_1:=\frac{p}{a+p-2}I_{1,p}+\frac{q}{a+q-2}I_{1,q},\indent\indent\,\,\,\,\,k_2:=-2\left(\frac{\mathfrak{C}^{\ast}_p\kappa_p\,p^{d_1}}{a+p-2}I_{5,p}+
\frac{\mathfrak{C}^{\ast}_q\kappa_q\,q^{d_2}}{a+q-2}I_{5,q}\right),$$
and
$$k_3:=\frac{p}{a+p-2}\left(\log\left(\frac{\mathfrak{C}^{\ast}_p\,p^{d_1-1}}{C'_{\epsilon}}\right)I_{1,p}+(d_1-1)(I_{2,p}-I_{3,p})+I_{4,p}\right)+
\indent\indent\indent\indent$$ $$\indent\indent\indent
+\frac{q}{a+q-2}\left(\log\left(\frac{\mathfrak{C}^{\ast}_q\,q^{d_2-1}}{C''_{\epsilon}}\right)I_{1,q}+(d_2-1)(I_{2,q}-I_{3,q})+I_{4,q}\right),$$
we are lead into the fulfillment of the equality (\ref{A01}) for the given explicit constants. It remains to do the same for (\ref{A02}),
which turns out to be an easier calculation. In fact, proceeding as before, we have:

\vspace{0.2cm}

\noindent$\displaystyle\int_{0}^{t} \mathfrak{D}_{p,q}(\tau) \ d\tau$\\
$$=\displaystyle\int_{0}^{t} \frac{p}{[a+p-2]t-(p-2)\tau} \ \log\left(\frac{C_{p}^{*} p^{d_1-1} (at- \tau) (t-\tau)^{d_1-1}}{C'_{\epsilon}
\left\{[a+(p-2)]t-(p-2)\tau \right\}^{d_1-1}}\right)\ d\tau-2C_{p}^{*} \kappa_{p} \displaystyle\int_{0}^{t} \left(\frac{p (at-\tau)^{\frac{1}{d_1}} (t-\tau)^{\frac{d_1-1}{d_1}}}{(a+p-2)t-(p-2)\tau}\right)^{d_1} d\tau$$
$$+\displaystyle\int_{0}^{t} \frac{q}{[a+q-2]t-(q-2)\tau} \ \log\left(\frac{C_{q}^{*} q^{d_2-1} (at- \tau) (t-\tau)^{d_2-1}}{C''_{\epsilon} \left\{[a+(q-2)]t-(q-2)\tau \right\}^{d_2-1}}\right) \ d\tau-2C_{q}^{*} \kappa_{q} \displaystyle\int_{0}^{t} \left(\frac{q (at-\tau)^{\frac{1}{d_2}} (t-\tau)^{\frac{d_2-1}{d_2}}}{(a+q-2)t-(q-2)\tau}\right)^{d_2} d\tau$$
$$=p \log\left(\displaystyle\frac{t C_{p}^{*} p^{d_1-1}}{ C'_{\epsilon}}\right)\displaystyle\int_{0}^{1} \frac{1}{[a+(p-2)(1-x)]} \ dx +
p\displaystyle\int_{0}^{1}  \frac{1}{[a+(p-2)(1-x)]} \ \log(a-x) \ dx\indent\indent\indent\indent\indent\indent\indent\indent\indent\indent$$
$$+p (d_1-1)\displaystyle\int_{0}^{1}  \frac{1}{[a+(p-2)(1-x)]} \ \log(1-x) \ dx
-p (d_1-1)\displaystyle\int_{0}^{1} \frac{1}{[a+(p-2)(1-x)]} \ \log(a+(p-2)(1-x)) \ dx$$
$$\indent\indent\indent\indent\indent\indent\indent\indent\indent\indent\indent\indent\indent\indent\indent\indent\indent\indent\indent\indent
- 2t  \ C_{p}^{*} \ \kappa_{p} \ p^{d_1} \displaystyle\int_{0}^{1} \left(\frac{(a-x)^{\frac{1}{d_1}}(1-x)^{\frac{d_1-1}{d_1}}}{(a+p-2)-(p-2)x}\right)^{d_1}dx$$
$$+q \log\left(\displaystyle\frac{t C_{q}^{*} q^{d_2-1}}{ C''_{\epsilon}}\right)\displaystyle\int_{0}^{1} \frac{1}{[a+(q-2)(1-x)]} \ dx +
q\displaystyle\int_{0}^{1}  \frac{1}{[a+(q-2)(1-x)]} \ \log(a-x) \ dx\indent\indent\indent\indent\indent\indent\indent\indent\indent\indent$$
$$+q (d_2-1)\displaystyle\int_{0}^{1}  \frac{1}{[a+(q-2)(1-x)]} \ \log(1-x) \ dx
-q (d_2-1)\displaystyle\int_{0}^{1} \frac{1}{[a+(q-2)(1-x)]} \ \log(a+(q-2)(1-x)) \ dx$$
$$\indent\indent\indent\indent\indent\indent\indent\indent\indent\indent\indent\indent\indent\indent\indent\indent\indent\indent\indent\indent
- 2t  \ C_{q}^{*} \ \kappa_{q} \ q^{d_2} \displaystyle\int_{0}^{1} \left(\frac{(a-x)^{\frac{1}{d_2}}(1-x)^{\frac{d_2-1}{d_2}}}{(a+q-2)-(q-2)x}\right)^{d_2}dx.$$

Then we set\\

\begin{apequation}\label{A13}
J_{1,p}:=\displaystyle\int_{0}^{1}\frac{1}{[a+(p-2)(1-x)]} \ dx =-\frac{\log(a)}{p-2}+ \frac{\log(a+p-2)}{p-2},
\end{apequation}

\begin{apequation}\label{A14}
J_{2,p}:=\displaystyle\int_{0}^{1} \frac{\log(a+(p-2)(1-x))}{[a+(p-2)(1-x)]} \ dx=-\frac{(\log(a))^2}{2(p-2)}+ \frac{(\log(a+p-2))^2}{2(p-2)},
\end{apequation}

\begin{apequation}\label{A15-16}
J_{3,p}:=\displaystyle\int_{0}^{1} \frac{\log(1-x)}{a+(p-2)(1-x)} \ dx,
\indent\indent\,\,\,
J_{4,p}:=\displaystyle\int_{0}^{1} \frac{\log(a-x)}{a+(p-2)(1-x)} \ dx,
\end{apequation}
\indent\\

\noindent$J_{5,p}:=\displaystyle\int_{0}^{1} \left(\frac{(a-x)^{\frac{1}{d_1}}(1-x)^{\frac{d_1-1}{d_1}}}{(a+p-2)-(p-2)x}\right)^{d_1}dx$\\[2ex]
$=\displaystyle\frac{1}{ad_1(1+d_1)(a+p-2)^{d_1}}\left(a^2(1+d_1)\displaystyle\frac{\Gamma(1+d_1)}{\Gamma(d_1)}\int_{0}^{1} t^{d_1-1}
\left(1-\frac{p-2}{a+p-2}t\right)^{-1}dt\right.$\\
\begin{apequation}\label{A17}
\indent\indent\indent\left.+(ad_1+p-2)\displaystyle\frac{\Gamma(2+d_1)}{\Gamma(d_1)}\int_{0}^{1} t^{d_1-1}(1-t)
\left(1-\frac{p-2}{a+p-2}t\right)^{-1}dt-(1+d_1)(a+p-2)\right),
\end{apequation}

\begin{apequation}\label{A18}
J_{1,q}:=\displaystyle\int_{0}^{1}\frac{1}{[a+(q-2)(1-x)]} \ dx =-\frac{\log(a)}{q-2}+ \frac{\log(a+q-2)}{q-2},
\end{apequation}

\begin{apequation}\label{A19}
J_{2,q}:=\displaystyle\int_{0}^{1} \frac{\log(a+(q-2)(1-x))}{[a+(q-2)(1-x)]} \ dx=-\frac{(\log(q))^2}{2(q-2)}+ \frac{(\log(a+q-2))^2}{2(q-2)},
\end{apequation}

\begin{apequation}\label{A20-21}
J_{3,q}:=\displaystyle\int_{0}^{1} \frac{\log(1-x)}{a+(q-2)(1-x)} \ dx,
\indent\indent\,\,\,
J_{4,q}:=\displaystyle\int_{0}^{1} \frac{\log(a-x)}{a+(q-2)(1-x)} \ dx,
\end{apequation}
\indent\\

\noindent$J_{5,q}:=\displaystyle\int_{0}^{1} \left(\frac{(a-x)^{\frac{1}{d_2}}(1-x)^{\frac{d_2-1}{d_2}}}{(a+q-2)-(q-2)x}\right)^{d_2}dx$\\[2ex]
$=\displaystyle\frac{1}{ad_2(1+d_2)(a+q-2)^{d_2}}\left(a^2(1+d_2)\displaystyle\frac{\Gamma(1+d_2)}{\Gamma(d_2)}\int_{0}^{1} t^{d_2-1}
\left(1-\frac{q-2}{a+q-2}t\right)^{-1}dt\right.$\\
\begin{apequation}\label{A22}
\indent\indent\indent\left.+(ad_2+q-2)\displaystyle\frac{\Gamma(2+d_2)}{\Gamma(d_2)}\int_{0}^{1} t^{d_2-1}(1-t)
\left(1-\frac{q-2}{a+q-2}t\right)^{-1}dt-(1+d_2)(a+q-2)\right),
\end{apequation}

\noindent where $\Gamma(\cdot)$ denotes the well-known Gamma function. As before, letting

$$k_4:=pJ_{1,p}+qJ_{1,q},\indent\indent\,\,\,\,\,k_5:=2\mathfrak{C}^{\ast}_p\kappa_p\,p^{d_1}J_{5,p}+
2\mathfrak{C}^{\ast}_q\kappa_q\,q^{d_2}J_{5,q},$$
and
$$k_6:=p\left(\log\left(\frac{\mathfrak{C}^{\ast}_p\,p^{d_1-1}}{C'_{\epsilon}}\right)J_{1,p}+(d_1-1)(J_{3,p}-J_{2,p})+J_{4,p}\right)+
\indent\indent\indent\indent$$ $$\indent\indent\indent
+q\left(\log\left(\frac{\mathfrak{C}^{\ast}_q\,q^{d_2-1}}{C''_{\epsilon}}\right)J_{1,q}+(d_2-1)(J_{3,q}-J_{2,q})+J_{4,q}\right),$$
we are lead into the fulfillment of the equality (\ref{A02}) with explicit constants, as desired.
\end{proof}

\noindent{\bf Acknowledgement}.  {\it The third author is supported by:
The Puerto Rico Science, Technology and Research Trust, under agreement number 2022-00014.\\
\indent The authors are thankful to Dr. Luis Medina for his help in the computational calculations and verification of some of the constants in the paper.\\

\indent\\

\noindent{\bf Disclaimer.}  {\it This content is only the responsibility of the authors and does not necessarily represent the official views of The Puerto Rico Science, Technology and Research Trust}.\\

\end{document}